\documentclass{amsart}
\usepackage{amssymb}
\usepackage{graphicx}
\usepackage[utf8]{inputenc}
\usepackage[hypertex]{hyperref}
\newtheorem{theorem}{Theorem}
\newtheorem*{theoremmain}{Theorem}
\newtheorem{definition}{Definition}
\newtheorem{proposition}{Proposition}
\newtheorem{lemma}{Lemma}
\newtheorem{corollary}{Corollary}
\newtheorem{exam}{Example}
\newenvironment{example}{\begin{exam}\rm}{\end{exam}}
\newtheorem{exams}{Examples}

\newtheorem{rmk}{Remark}
\newenvironment{remark}{\begin{rmk}\rm}{\end{rmk}}
\newtheorem{notat}{Notation}

\renewcommand{\ne}{\not =}

\title[Logarithmic Models for Non-Dicritical Foliations]{Logarithmic Models for Non-Dicritical Foliations}
\author{Felipe Cano}
\author{Nuria Corral}
\dedicatory{Dedicated to Dominique Cerveau, with respect and admiration}
\subjclass[2010]{32S65 (37F75)}
\keywords{Desingularization, singular foliation, logarithmic forms, residues}
\thanks{Both authors are supported by Ministerio de Economía y Competitividad,
Spain,  MTM2016-77642-C2-1-P}
\begin{document}
\begin{abstract}{We show the existence of an essentially unique logarithmic model for any germ of non-dicritical singular holomorphic foliation of codimension one in $({\mathbb C}^n,0)$ without saddle-nodes.}
\end{abstract}
\maketitle
\tableofcontents
\section{Introduction}
This paper is devoted to show the existence and uniqueness of logarithmic models for any holomorphic foliation on $({\mathbb C}^n,0)$ of generalized hypersurface type. In the case of $n=2$, this result has been obtained by N. Corral in \cite{Cor}. The main result is  Theorem \ref{teo:main} in the last section of the paper. We state it as follows:

 \begin{theoremmain}
 Every  generalized hypersurface on $({\mathbb C}^n,0)$ has a logarithmic model.
 \end{theoremmain}

A germ $\mathcal L$ of singular codimension one foliation  on $({\mathbb C}^n,0)$ is {\em  logarithmic} when it is given by a closed logarithmic $1$-form
$$
\eta=\sum_{i=1}^s\lambda_i\frac{df_i}{f_i},\quad f_i\in {\mathcal O}_{{\mathbb C}^n,0}.
$$
In other words the foliation $\mathcal L$ has the multivaluated first integral $f_1^{\lambda_1}f_2^{\lambda_2}\cdots f_s^{\lambda_s}$. Up to reduction of singularities of the germ of hypersurface
$$
H=(f_1f_2\cdots f_s=0),
$$
the transform of $\eta$ is a global closed logarithmic $1$-form  and  the total transform of $H$ has normal crossings. In this situation all the local holonomies are linear in terms of the coordinates given by $H$. We can say that the holonomy of $\mathcal L$ is ``globally linearizable''. Of course this picture needs to be specified, mainly by asking that we are not in a ``dicritical'' situation.

 Roughly speaking, a ``logarithmic model'' for a codimension one foliation $\mathcal F$ should be a logarithmic foliation $\mathcal L$ such that the local holonomies of $\mathcal L$ coincide with the linear parts of the local holonomies of $\mathcal F$. In this way, the logarithmic model is an object that can be considered as ``the linear part of the holonomy of $\mathcal F$" or, in some sense, a ``holonomic initial part'' of $\mathcal F$.

 The logarithmic models in ambient dimension two may be described in a more precise way. We do it for foliations $\mathcal F$ on $({\mathbb C}^2,0)$ without saddle nodes in their reduction of singularities (hidden saddle nodes) and that are also ``non-dicritical'', in the sense that we encounter only invariant exceptional divisors in the
 sequence of blowing-ups desingularizing $\mathcal F$. We  give the name {\em generalized curves} to such foliations, following a terminology that comes from the foundational paper of Camacho, Lins Neto and Sad \cite{Cam-LN-S}. In this situation, after desingularization, the foliation at a singular point is given by a local $1$-form
 $$
  (\lambda+\cdots)ydx+(\mu+\cdots)xdy,\quad \lambda\mu\ne0,\;\lambda/\mu\notin {\mathbb Q}_{\leq 0}.
 $$
 The quotient $-\lambda/\mu$ is the Camacho-Sad index of the foliation with respect to $y=0$ and it also determines the coefficient of the linear part of the holonomy. Up to multiply the 1-form by a scalar and to adapt the coordinates, a local logarithmic foliation, having holonomy with the same linear part as $\mathcal F$, is locally given by
 $$
 \lambda\frac{dx}{x}+\mu\frac{dy}{y}.
 $$
This is the way we take for approaching a germ of generalized curve $\mathcal F$ on $({\mathbb C}^2,0)$ by a logarithmic foliation $\mathcal L$:
 \begin{quote}
 A logarithmic foliation $\mathcal L$ is a {\em logarithmic model} for a generalized curve $\mathcal F$ on $({\mathbb C}^2,0)$ if it has the same invariant branches as $\mathcal F$ and the same Camacho-Sad indices after reduction of singularities.
 \end{quote}
This provides a precise definition.  With no effort one realizes that the property is independent of the chosen reduction of singularities; indeed, in dimension two we have a well defined minimal reduction of singularities and any other one is obtained by additional blowing-ups from it.

In the paper \cite{Cor} there is a proof of the existence of logarithmic models in dimension two for any generalized curve.
Logarithmic models in dimension two have been particularly useful for describing
the properties of the generic polar of a given foliation, because the main Newton Polygon parts coincide for the foliation and the logarithmic model, see \cite{Cor}. Let us also note that some results in dimension two may be stated in the dicritical case \cite{Can-Co}, anyway in this paper we consider always the non-dicritical situation.

We have two possible ways for extending the concept of logarithmic models to higher dimension. The first one is to use reduction of singularities as in the two-dimensional case. Since we are considering generalized hypersurfaces, we know the existence of reduction of singularities for our foliations, more precisely, any reduction of singularities of the finite set of invariant hypersurfaces provides a reduction of singularities of the foliation \cite{Fer-M}. The second way is to perform two-dimensional tests. It is known that certain properties in algebraic geometry are tested by valuative criteria, for instance integral dependence, or properness. In the theory of codimension one foliations, there are remarkable properties detected by testing with a two-dimensional map. The existence of holomorphic first integral is one of them, as exhibited in the paper of Mattei-Moussu \cite{Mat-Mou}. The dicriticalness and the existence of hidden-saddle nodes are also properties of this kind:
\begin{itemize}
\item Dicriticalness: A  codimension one foliation $\mathcal F$ on $({\mathbb C}^n,0)$ is {\em dicritical} if and only if there is a holomorphic map $\phi: ({\mathbb C}^2,0)\rightarrow ({\mathbb C}^n,0)$ such that $\phi^*{\mathcal F}=(dx=0)$ and  the image of $y=0$ is invariant for $\mathcal F$.
\item Existence of hidden saddle nodes:  A  codimension one foliation $\mathcal F$ on $({\mathbb C}^n,0)$ has a {\em hidden saddle-node} if there is a holomorphic map $\phi: ({\mathbb C}^2,0)\rightarrow ({\mathbb C}^n,0)$ such that $\phi^*{\mathcal F}$ is a saddle-node.
\end{itemize}
When the foliation $\mathcal F$ is non-dicritical and without hidden saddle-nodes, we say that $\mathcal F$ is a {\em generalized hypersurface}. In this context, we take a definition of logarithmic model as follows:
\begin{quote} Let $\mathcal F$ be a generalized hypersurface and consider a logarithmic foliation $\mathcal L$, both on $({\mathbb C}^n,0)$. We say that $\mathcal L$ is a {\em logarithmic model for $\mathcal F$}  if and only if $\phi^*{\mathcal L}$ is a logarithmic model for $\phi^*{\mathcal F}$, for any holomorphic map $\phi: ({\mathbb C}^2,0)\rightarrow ({\mathbb C}^n,0)$ such that $\phi^*{\mathcal F}$ exists.
\end{quote}

In this paper we show that the two above ways are confluent. The uniqueness of logarithmic models is a consequence of the same result in dimension two. We show the existence of logarithmic models for generalized hypersurfaces by working throughout a particular reduction of singularities of the foliation $\mathcal F$.

From the technical viewpoint, we develop the theory of logarithmic models in terms of $\mathbb C$-divisors. We introduce the concept of {\em divisorial model} and we state the existence in the main technical result Theorem \ref{teo:main}.  There is a relationship between $\mathbb C$-divisors and logarithmic foliations, that provides the bridge between the divisorial models and the logarithmic models as explained in the last Section \ref{Logarithmic Models}.

Consider a non-singular complex analytic space $M$.
A {\em ${\mathbb C}$-divisor\/} $\mathcal D$ on $M$ is a formal finite sum
$$
{\mathcal D}=\sum_{i=1}^s\lambda_i H_i, \quad 0\ne \lambda_i\in {\mathbb C},
$$
where the $H_i\subset M$ are hypersurfaces. The support of the divisor is the union of the hypersurfaces $H_i$. We can make the usual operations with $\mathbb C$-divisors, in particular the pull-back $\phi^*{\mathcal D}$ under a holomorphic map $\phi:N\rightarrow M$, when the image is not locally contained in the support of ${\mathcal D}$. Working locally, if we take a reduced equation $f_i=0$ of  $H_i$, we can consider the closed logarithmic $1$-form $\eta$ given by
$$
\eta=\sum_{i=1}^s\lambda_i\frac{df_i}{f_i}.
$$
The logarithmic foliation induced by $\eta$ will be called {\em $\mathcal D$-logarithmic.}

In  dimension two, we give a proof of the existence of logarithmic model in terms of $\mathbb C$-divisors, that is  divisorial models, in a more explicit way than in the paper \cite{Cor}. When the foliation $\mathcal F$ is desingularized, at a singular point we have exactly two invariant curves, $\Gamma_1$ and $\Gamma_2$, given by the equations $\Gamma_1=(x_1=0)$ and $\Gamma_2=(x_2=0)$. We know that the foliation is given by a differential $1$-form $\omega$ as
$$
\omega=(\lambda_1+\cdots)\frac{dx_1}{x_1}+ (\lambda_2+\cdots)\frac{dx_2}{x_2},
$$
where $-\lambda_i/\lambda_j$ are the Camacho-Sad indices. We say that the ${\mathbb C}$-divisor
$$
{\mathcal D}=\lambda_1\Gamma_1+\lambda_2\Gamma_2,
$$
 is a divisorial model for $\mathcal F$. Of course, the logarithmic foliation $\mathcal L$ defined by
 $$
 \eta=\lambda_1dx_1/x_1+\lambda_2dx_2/x_2
 $$
fulfils the definition of being a logarithmic model for $\mathcal F$. We pass to the general case in dimension two through the stability under blowing-ups. More precisely, we recover the general definition of Camacho-Sad indices for foliations in dimension two (see \cite{Bru} and \cite{LNet}) and we establish a similar one for $\mathbb C$-divisors. Both are compatible with the blowing-ups and in this way we obtain logarithmic models once we have proven the existence of divisorial models in Theorem \ref{th:existenciayunicidadendimensiondos}.

The above arguments pass in higher dimension, hence we have a definition of divisorial model that automatically gives a logarithmic foliation that is a logarithmic model. In this way, the main difficulty in the paper is the proof of the existence of a divisorial model for any generalized hypersurface $\mathcal F$ on $({\mathbb C}^n,0)$. We state this result in Theorem \ref{teo:main} in Section \ref{Logarithmic Models For Generalized Hypersurfaces}.

The first sections are devoted to present the theory of $\mathbb C$-divisors, the relationship between $\mathbb C$-divisors, closed logarithmic 1-forms and logarithmic foliations, the dicriticalness condition for $\mathbb C$-divisors and foliations, the property of being a generalized hypersurface, the existence of logarithmic models in dimension two throughout the generalized Camacho-Sad indices, the reduction of singularities and the properties of generic equireduction and relative transversality that are useful in the final proofs.

The proofs of the main results are given in Section \ref{Logarithmic Models For Generalized Hypersurfaces}. We first show the existence of a ${\mathbb C}$-divisor compatible with a given reduction of singularities in Theorem \ref{teo:pilogarithmicmodel} and finally we prove the main  Theorem \ref{teo:main} on the existence of divisorial models.

 All the paper has been developed having in mind that a logarithmic model is a foliation; anyway, from the technical view point, we have stated and prove results in terms of divisorial models. In the last Section \ref{Logarithmic Models}, we quickly summarize how the existence and uniqueness results on divisorial models are translated to logarithmic models. In this way, we obtain the proof of the main Theorem stated in this Introduction.

\section{$\mathbb C$-Divisors}
Let $M$ be a non singular complex analytic variety. The space of {\em generalized divisors $\operatorname{Div}_{\mathbb C}(M)$}, also called {\em $\mathbb C$-divisors},  is defined to be the ${\mathbb C}$-vector space having as a basis the set of irreducible hypersurfaces of $M$. They have been introduced in \cite{Can-Co} for the purpose of describing logarithmic models of foliations in ambient dimension two. Thus, a {\em $\mathbb C$-divisor ${\mathcal D}$ } in $M$ is a finite expression
$$
{\mathcal D}=\sum_{H}\lambda_HH,
$$
where $H$ runs over the irreducible hypersurfaces of $M$ and the
coefficients $\lambda_H$ are complex numbers, such that only finitely many of them are nonzero. The {\em support $\operatorname{Supp}({\mathcal D})$ of $\mathcal D$} is the union of the $H$ such that $\lambda_H\ne0$. We say that two nonzero $\mathbb C$-divisors ${\mathcal D}_1, {\mathcal D}_2\in \operatorname{Div}_{\mathbb C}(M)$ with connected support are {\em projectively equivalent} if and only if there is a nonnull scalar $\lambda\in {\mathbb C}^*$ such that ${\mathcal D}_2=\lambda{\mathcal D}_1$. If the support is not connected, we say that they are projectively equivalent when the condition holds at each connected component of the support.

Consider a function $f:M\rightarrow {\mathbb C}$ that is not constant at any connected component of $M$. As usual, we define the divisor $\operatorname{Div}(f)$ by
$$
\operatorname{Div}(f)=\sum_H \mu_HH,
$$
where $\mu_H\ne 0$ if and only if $H$ is an irreducible component of $f=0$ and $\mu_H$ is the multiplicity of a local reduced equation of $H$ as a factor of $f$.

Let us consider a closed hypersurface $S$ of $M$, not necessarily irreducible. Let $S=\cup_{i=1}^s H_i$ be the decomposition of $S$ into a union of irreducible components. The divisor $\operatorname{Div}(S)$ is defined to be
$$
\operatorname{Div}(S)=\sum_{i=1}^s H_i.
$$
In particular, if ${\mathcal D}=\sum_H\lambda_H H$, we also have that ${\mathcal D}=\sum_H\lambda_H\operatorname{Div}(H)$. This simple remark allows us to define the restriction of a $\mathbb C$-divisor to an open set $U\subset M$, by means of the formula
$$
{\mathcal D}\vert_U=\sum_H\lambda_H\operatorname{Div}(H\cap U).
$$
 In this way, we can also interpret the germ ${\mathcal D}_p$ at a point $p\in M$ of a $\mathbb C$-divisor ${\mathcal D}$ in $M$ as being a ${\mathbb C}$-divisor on the germified space $(M,p)$.  Of course, these ones are particular cases of the inverse image of a $\mathbb C$-divisor by a morphism to be introduced below.

\begin{remark} Most of the complex analytic varieties in this paper are germs over compact sets. In this case any hypersurface has only finitely  many irreducible components. Anyway, we consider only hypersurfaces with this property, even if we are in an analytic variety not necessarily a germ over a compact. The reader will appreciate at each statement the limits of this implicit assumption.
\end{remark}

Consider a morphism $\phi: N\rightarrow M$ between connected non singular complex analytic varieties and a hypersurface $S\subset M$.  We say that $\phi$ is {\em $S$-transverse} if and only if the image of $\phi$ is not contained in $S$. In this case $\phi$ is $H$-transverse for any irreducible component $H$ of $S$ and the inverse image is a hypersurface $\phi^{-1}(H)\subset N$.
When $\phi:N\rightarrow M$ is $H$-transverse, we define the $\mathbb C$-divisor $\phi^*(1\cdot H)$ of $N$ by the following property: for any point $q\in N$ the divisor $\phi^*(1\cdot H)$ germified at $q$ is equal to
$
\operatorname{Div} (f\circ \phi)
$,
where $f=0$ is a local reduced equation of $H$ at $\phi(q)$.

Consider a $\mathbb C$-divisor ${\mathcal D}=\sum_H\lambda_HH$. We say that the morphism
 $\phi:N\rightarrow M$ is
{\em ${\mathcal D}$-transverse} if it is $S$-transverse where
$S =\operatorname{Supp}(\mathcal{D})$.
Otherwise, we say that $\phi$ is {\em $\mathcal D$-invariant}. When $\phi$ is ${\mathcal D}$-transverse, the {\em inverse image} is defined by
$$
\phi^{*}{\mathcal D}=\sum_H\lambda_H\phi^{*}(1\cdot H).
$$
We are particularly interested in the case of
blowing-ups $\pi:M'\rightarrow M$ with irreducible non-singular center $Y\subset M$.
The blowing-up $\pi$ being a surjective morphism is
 ${\mathcal D}$-transverse for any $\mathbb C$-divisor ${\mathcal D}$. The inverse image  $\pi^{*}{\mathcal D}$, or {\em transform of $\mathcal D$ by $\pi$},  is given by
$$
\pi^{*}{\mathcal D}=\mu E+\sum_{H}\lambda_HH',
\quad E=\pi^{-1}(Y), \quad \mu= \sum_{H}\nu_Y(H)\lambda_H,
$$
where $H'$ are the strict transforms of the irreducible hypersurfaces $H\subset M$ and $\nu_Y(H)$ denotes the generic multiplicity of $H$ along $Y$.
The blowing-up $\pi$  is said to be {\em $\mathcal D$-admissible} when $Y\subset \operatorname{Supp}({\mathcal D})$. This is equivalent to say that
$$
\sum_{H\subset \operatorname{Supp}({\mathcal D})}\nu_{Y}(H)\geq 1.
$$
 A ${\mathcal D}$-admissible blowing-up $\pi$ is called {\em ${\mathcal D}$-dicritical} when
$
\sum_{H}\nu_Y(H)\lambda_H=0
$;
that is, the exceptional divisor $E$ is not contained in the support of $\pi^*{\mathcal D}$.
\begin{remark} Let us recall that for any germ of function
$f\in {\mathcal O}_{{\mathbb C}^n,0}$,
the generic multiplicity $\nu_Y(f)$ along $Y\subset ({\mathbb C}^n,0)$ is defined to be the minimum of the multiplicity of $f$ at the points of $Y$ near the origin (the multiplicity is an upper semi-continuous function). The generic multiplicity $\nu_Y(H)$ is the generic multiplicity along $Y$ of a reduced germ $f$, such that $H$ is given by $f=0$.
\end{remark}

\begin{remark}
 \label{rk:transversemaps}
 Let us consider two morphisms $\phi_2:N_2\rightarrow N_1$ and $\phi_1:N_1\rightarrow M$ and a $\mathbb C$-divisor ${\mathcal D}$ on $M$. If $\phi_1\circ\phi_2$ is ${\mathcal D}$-transverse, then $\phi_1$ is ${\mathcal D}$-transverse, the morphism
$\phi_2$ is $\phi_1^*{\mathcal D}$-transverse and
$$
(\phi_1\circ\phi_2)^*{\mathcal D}=\phi_2^*(\phi_1^*{\mathcal D}).
$$
The converse is not true. It is possible to have that $\phi_1$ is ${\mathcal D}$-transverse and
$\phi_2$ is $\phi_1^*{\mathcal D}$-transverse, but $\phi_1\circ\phi_2$ is ${\mathcal D}$-invariant. The typical example of this situation is the inclusion
$$
\pi^{-1}(Y)\stackrel{\phi_2}{\subset} M'\stackrel{\phi_1}{\rightarrow }M,
$$
where $\phi_1$ is a $\mathcal D$-dicritical blowing-up with center $Y$. The $\mathbb C$-divisor $\phi_2^*(\phi_1^*{\mathcal D})$ cannot be obtained directly from  $\pi^{-1}(Y)\rightarrow M$ (this phenomenon is an essential fact for the transcendence of leaves of singular foliations studied in \cite{Can-L-M}).
\end{remark}

If no confusion arises, we write
$
\pi:(M',{\mathcal D}')\rightarrow (M,{\mathcal D})
$
to denote a ${\mathcal D}$-transverse holomorphic map $\pi:M'\rightarrow M$, where ${\mathcal D}'=\pi^*{\mathcal D}$.

The rest of this section is devoted to characterize the dicriticalness of a $\mathbb C$-divisor. We take the following definition, which is inspired in the corresponding one for foliations:

\begin{definition}  Consider a $\mathbb C$-divisor $\mathcal D$ on a non-singular complex analytic variety $M$. We say that  $\mathcal D$ is {\em dicritical at a point $p\in M$} if and only if there is a $\mathcal D$-transverse holomorphic map
$
\phi:({\mathbb C}^2,0)\rightarrow M
$
such that
$$
\phi(0)=p,\quad \phi(y=0)\subset \operatorname{Supp} ({\mathcal D}),\quad \phi^*{\mathcal D}=0.
$$
We say that $\mathcal D$ is {\em dicritical} if there is a point $p\in M$ such that it is dicritical at $p$. In a consonant way, we say that $\mathcal D$ is {\em non-dicritical} if and only if it is non-dicritical at each point $p\in M$ (in the case of germs $(M,K)$ we ask the conditions for the points $p\in K$).
\end{definition}

\begin{proposition}
 \label{pro:divdicrunblowinup}
 Consider a $\mathbb C$-divisor ${\mathcal D}$ on $M=({\mathbb C}^n,0)$ and a non-dicritical admissible blowing-up
$
\pi:((M,\pi^{-1}(0)),{\mathcal D}')\rightarrow (({\mathbb C}^n,0),{\mathcal D}).
$
Then, the $\mathbb C$-divisor ${\mathcal D}$ is dicritical if and only if there is a point $p'\in\pi^{-1}(0)$ such that ${\mathcal D}'$ is dicritical at $p$.
\end{proposition}
\begin{proof} Let us assume that ${\mathcal D}'$ is dicritical at a  point $p'\in \pi^{-1}(0)$. Then, there is a ${\mathcal D}'$-transverse map
$\phi': ({\mathbb C}^2,0)\rightarrow (M',p')$ such that
$${\phi'}^*{\mathcal D'}=0,\quad \phi'(y=0)\subset \operatorname{Supp} ({\mathcal D}').
$$
Since $\pi$ is non-dicritical, we have that
$
\operatorname{Supp} ({\mathcal D}')=\pi^{-1} (\operatorname{Supp} ({\mathcal D}))
$. This implies that $\phi=\pi\circ\phi'$ is also a ${\mathcal D}$-transverse map and moreover, we have
$$
\phi(y=0)\subset \operatorname{Supp} ({\mathcal D}), \quad \phi^*{\mathcal D}={\phi'}^*{\mathcal D}'=0.
$$
Hence, the $\mathbb C$-divisor ${\mathcal D}$ is dicritical.

Conversely, let us assume that ${\mathcal D}$ is dicritical. Consider a $\mathcal D$-transverse map $\phi: ({\mathbb C}^2,0)\rightarrow ({\mathbb C}^n,0)$, such that $\phi(y=0)\subset \operatorname{Supp}{\mathcal D}$ and $\phi^*{\mathcal D}=0$. In view of Proposition \ref{prop:appdos} in the Appendix I, there is a morphism
$$
\sigma: (N,\sigma^{-1}(0))\rightarrow ({\mathbb C}^2,0)
$$
that is a composition of blowing-ups and a morphism $\psi: (N,\sigma^{-1}(0))\rightarrow (M,\pi^{-1}(0))$ such that $\pi\circ\psi=\phi\circ\sigma$.
Note that $\phi\circ\sigma$ is $\mathcal D$-transverse, since $\phi$ is $\mathcal D$-transverse and $\sigma$ is a surjective map. By Remark \ref{rk:transversemaps}, we have that $\psi$ is $\pi^*{\mathcal D}$-transverse and
$$
0= \sigma^*(\phi^*{\mathcal D})=(\phi\circ\sigma)^*{\mathcal D}=(\pi\circ\psi)^*{\mathcal D}=\psi^*({\mathcal D}').
$$
Now, let $(\Gamma,q)\subset (N,q)$ be the strict transform of $y=0$ by $\sigma$. We have that $\pi(\psi(\Gamma))=\phi(y=0)\subset \operatorname{Supp}({\mathcal D})$. In other words
$$
\psi(\Gamma)\subset \pi^{-1}(\operatorname{Supp}({\mathcal D}))=\operatorname{Supp}({\mathcal D}').
$$
Select local coordinates $x',y'$ at $q$ such that $\Gamma=(y'=0)$ and let
$$
\phi':(N,q)=({\mathbb C}^2,0)\rightarrow (M,p),\quad p=\psi(q),
$$
be the map between germs induced by $\psi$. Thanks to $\phi'$, we see that ${\mathcal D}'$ is dicritical at $p$.
\end{proof}
The following corollary is a direct consequence of Proposition  \ref{pro:divdicrunblowinup}:
\begin{corollary}
\label{dicriticidadexplosionnodicritica}
Consider a morphism
$
\pi:(M', {\mathcal D}')\rightarrow (M,{\mathcal D})
$
that is the composition of a sequence of non-dicritical admissible blowing-ups.
Then, the $\mathbb C$-divisor $\mathcal D$ is dicritical in $M$ if and only if ${\mathcal D}'$ is dicritical in $M'$.
\end{corollary}

Now, we characterize the dicriticalness in terms of admissible blowing-ups. We start with the normal crossings case. We say that a $\mathbb C$-divisor $\mathcal D$ on a complex analytic variety $M$ has a {\em non-negative resonance} at a point $p\in M\cap \operatorname{Supp}({\mathcal D})$ if the germ of the divisor is written
$$
{\mathcal D}_p=\sum_{i=1}^s\lambda_i H_i,\quad \lambda_i\ne 0\mbox{ for } i=1,2,\ldots,s
$$
and there is $\mathbf{m}=(m_1,m_2,\ldots,m_s)\in {\mathbb Z}_{\geq 0}^s$, with $\mathbf{m}\ne\mathbf{0}$ such that
\begin{equation}
\label{eq:resonance}
\sum_{i=1}^sm_i\lambda_i=0,\quad (m_1,m_2,\ldots,,m_s)\in {\mathbb Z}^s_{\geq 0}\setminus\{\mathbf{0}\}.
\end{equation}

\begin{lemma}
  \label{lema:normalcrossings} Let $(M,K)$ be a non-singular complex analytic variety that is a germ over a compact subset $K\subset M$. Consider a $\mathbb C$-divisor $\mathcal D$ in $(M,K)$ whose support has normal crossings. Assume that there is a point $p\in K\cap\operatorname{Supp}(\mathcal D)$ in which $\mathcal D$  has a non-negative resonance. Then, there are  morphisms
  $
  \pi':(M',{\mathcal D}')\rightarrow (M,{\mathcal D})$ and
  $\pi'':(M'',{\mathcal D}'')\rightarrow (M',{\mathcal D}')$,
  such that
  $\pi'$ is the composition of a sequence of non-dicritical admissible blowing-ups and $\pi''$ is a dicritical admissible blowing-up.
\end{lemma}
\begin{proof}
This result, in another context, is proven in \cite{FDu}. Let us give a quick idea of a proof. Choose local coordinates $(x_1,x_2,\ldots,x_n)$ at the origin such that
$$H_i=(x_i=0), \quad i=1,2,\ldots,s.
$$ Up to a reordering, we assume that $\prod_{i=1}^tm_i \ne 0$ and
$m_i=0$ for $t+1\leq i\leq s$. We proceed by induction on the lexicographical invariant $(t,\delta)$, where
$$
\delta=\min_{1\leq i<j\leq t}\{m_i+m_j\}.
$$
Assume, up to a new reordering, that $\delta=m_1+m_2$ and $m_1\leq m_2$. Choose $Y=(x_1=x_2=0)$ as a center of blowing-up. The first chart of this blowing-up gives a morphism
$$
\phi: ({\mathbb C}^n,0)\rightarrow ({\mathbb C}^n,0)
$$
defined by the equations $x_1=x'_1, x_2=x'_1x'_2$ and $x_i=x'_i$ for $i=3,4,\ldots,n$. The transform $\phi^*{\mathcal D}$ is given at the origin of this chart by
$$
\phi^*{\mathcal D}=(\lambda_1+\lambda_2)E+\sum_{i=2}^s\lambda_i H'_i,
$$
where $E=(x'_1=0)$ and $H'_i=(x'_i=0)$ for $i=2,3,\ldots,s$. If $\lambda_1+\lambda_2=0$ we are done, since then we have an admissible dicritical blowing-up. Otherwise, we obtain a resonance
$$
\mathbf{m}'=(m_1, m_2-m_1,m_3,\ldots,m_s)
$$
and the invariant $(t',\delta')$ is strictly smaller than $(t,\delta)$. In spite of the local presentation, the above procedure is in fact a global one. This ends the proof.
\end{proof}

\begin{proposition}
\label{pro:divdicritico}
  Let us consider a $\mathbb C$-divisor ${\mathcal D}$ on $M=({\mathbb C}^n,0)$. The $\mathbb C$-divisor $\mathcal D$ is dicritical if and only if there are morphisms
  $$
  \pi':(M',{\mathcal D}')\rightarrow (M,{\mathcal D}),\quad
  \pi'':(M'',{\mathcal D}'')\rightarrow (M',{\mathcal D}'),
  $$
  such that
  $\pi'$ is the composition of a sequence of non-dicritical admissible blowing-ups and $\pi''$ is a dicritical admissible blowing-up.
\end{proposition}

\begin{proof} Let us assume first the existence of $\pi',\pi''$ with the stated properties.
Since $\pi'$ is a composition of non-dicritical admissible blowing-ups, we have that
  $$
  E'\subset\operatorname{Supp}({\mathcal D}'),
 $$
 where $E'$ is the exceptional divisor of $\pi'$. Let $Y\subset M'$ be the center of $\pi''$ and denote $\pi=\pi'\circ\pi''$. The
exceptional divisor of $\pi$ is $E''={\tilde E}\cup D$, where $D={\pi''}^{-1}(Y)$ and
$\tilde E$ is the strict transform of $E'$ by $\pi''$. Since $\pi''$ is dicritical, we have that
$$
\tilde E\subset\operatorname{Supp}({\mathcal D}''),\quad D\not\subset\operatorname{Supp}({\mathcal D}'').
$$
Take a point $p\in D\setminus\operatorname{Supp}({\mathcal D}'')$. Let us identify the germ $(M'',p)$ with $({\mathbb C}^n,0)$, with coordinates $x_1,x_2,\ldots,x_n$, where $D$ is locally given at $p$ by the equation $x_n=0$. Consider the morphism
$$
\psi:({\mathbb C}^2,0)\rightarrow ({\mathbb C}^n,0)=(M'',p)\hookrightarrow M''
$$
defined by $x_1=x$, $x_n=y$ and $x_i=0$, for $2\leq i\leq n-1$.
We have that $\psi(y=0)\subset D$ and $\operatorname{Im}(\psi)\not\subset D$. Since the support of ${\mathcal D}''$ is empty around $p$, we have that $\psi^*{\mathcal D}''=0$ and we also have
$$
\operatorname{Im}(\psi)\not\subset D\cup\operatorname{Supp}({\mathcal D}'')=E''\cup \operatorname{Supp}({\mathcal D}'').
$$
Noting that $\pi^{-1}(\operatorname{Supp}({\mathcal D}))=\tilde E\cup \operatorname{Supp}({\mathcal D}'')$, we conclude that $\phi$ is ${\mathcal D}$-transverse, where $\phi=\pi\circ \psi$. Then, we have
$$
\phi^*{\mathcal D}=\psi^*(\pi^*{\mathcal D})=\psi^*({\mathcal D''})=0,\quad
\phi(y=0)\subset \pi'(Y)\subset\operatorname{Supp}({\mathcal D}).
$$
This implies that $\mathcal D$ is dicritical.

Assume now that ${\mathcal D}$ is dicritical. Let us perform a Hironaka reduction of singularities of the support of $\mathcal D$ by means of admissible blowing-ups (see
\cite{Aro-H-V, Hir}). We can assume that none of the blowing-ups in the reduction of singularities is dicritical, since then we are done. Hence, we have a morphism
$$
\tilde\pi: ((\tilde M,\tilde\pi^{-1}(0)),\tilde{\mathcal D})\rightarrow (({\mathbb C}^n,0),{\mathcal D})
$$
that is a composition of non-dicritical admissible blowing-ups such that $\operatorname{Supp}({\tilde{\mathcal D}})$ has normal crossings. Now, in view of Lemma \ref{lema:normalcrossings}, it is enough to find a point $p$ in $\tilde\pi^{-1}(0)\cap \operatorname{Supp}(\tilde{\mathcal D})$ such that $\tilde{\mathcal D}$ has a non-negative resonance at $p$.

By Proposition \ref{pro:divdicrunblowinup}, there is a point $p \in \tilde\pi^{-1}(0)$ such that $\tilde {\mathcal D}_{p}$ is dicritical. Then, there is a $\tilde{\mathcal D}$-transverse map
$$
\tilde\phi: ({\mathbb C}^2,0)\rightarrow (\tilde M,p)
$$
such that $\tilde\phi^*{\tilde{\mathcal D}}=0$ and $\tilde\phi(y=0)\subset \operatorname{Supp} (\tilde{\mathcal D})$. Let us identify $(\tilde M,p)$ with $({\mathbb C}^n,0)$ by means of a choice of local coordinates $x_1,x_2,\ldots,x_n$ at $p$ such that
$$
\tilde{\mathcal D}_p=\sum_{i=1}^s\lambda_iH_i,\quad H_i=(x_i=0),\; i=1,2,\ldots,s.
$$
Put $\tilde\phi_i=x_i\circ\tilde\phi$, for $i=1,2,\ldots, n$. We know that $\tilde\phi_i(0)=0$ for $i=1,2,\ldots,n$ and that $\tilde\phi_\ell\ne 0$, for $1\leq\ell\leq s$. Let $\Gamma\subset ({\mathbb C}^2,0)$ be an irreducible component of $\tilde\phi_1=0$, that is $\nu_\Gamma(\tilde\phi_1)\geq 1$. The coefficient of $\Gamma$ in $\tilde\phi^*{\tilde{\mathcal D}}=0$ is
$$
\sum_{i=1}^s\lambda_i\nu_\Gamma(\tilde\phi_i)=0.
$$
This is the desired non-negative resonance.
\end{proof}

\begin{remark} In other words, the $\mathbb C$-divisor $\mathcal D$ is dicritical if and only if there is a sequence of non-dicritical admissible blowing-ups that can be followed by a dicritical admissible blowing-up.
\end{remark}

The nonnegative resonances characterize dicriticalness in the case of normal crossings support, as we show in the following result:
\begin{corollary}
\label{cor:dicriticonormalcrossings}
Consider a $\mathbb C$-divisor ${\mathcal D}=\sum_{i=1}^s\lambda_iH_i$ on $M=({\mathbb C}^n,0)$ whose support $S=\cup_{i=1}^sH_i$ has normal crossings. The following statements are equivalent:
\begin{enumerate}
\item The $\mathbb C$-divisor $\mathcal D$ is dicritical.
\item There is a nonnegative resonance $\sum_{i=1}^sm_i\lambda_i=0$, with
$m_i\geq 0$ not all zero integer numbers.
\end{enumerate}
\end{corollary}
\begin{proof} See the second part of the proof of Proposition \ref{pro:divdicritico}.
\end{proof}

\section{Logarithmic Foliations and Dicriticalness}

Let $\mathcal F$ be a codimension one singular holomorphic foliation on a non-singular complex analytic variety  $M$. Given a point $p \in M$, we recall that the germ of $\mathcal F$ at $p$ is generated by an integrable meromorphic germ $\eta$ of differential $1$-form. Moreover two such differential $1$-forms $\eta$ and $\eta'$ generate the same germ of foliation if and only if $\eta'=\phi\eta$, where $\phi$ is the germ at $p$ of a meromorphic function.

We recall from \cite{Sai} that a meromorphic germ of differential $1$-form $\eta$ at a point $p\in M$ is  {\em logarithmic} when both $\eta$ and $d\eta$ have at most simple poles.  The set $\operatorname{Pol}(\eta)$ of poles of a meromorphic differential $1$-form $\eta$ is the hypersurface $g=0$, where $g\eta$ is holomorphic and $g$ divides any other $g'$ such that $g'\eta$ is holomorphic; the poles are simple when we can take $g$ to be reduced.

\begin{remark}
 \label{rk:logaritmicgenerator}
 Assume that $\mathcal F$ is a germ of foliation on $({\mathbb C}^n,0)$ locally generated by a germ of holomorphic integrable $1$-form $\omega$, without common factors in its coefficients. Let $f=0$ be a reduced equation for a (maybe non-irreducible) invariant hypersurface of $\mathcal F$; this means that $f$ divides $df\wedge\omega$. Then, the meromorphic $1$-form $\omega/f$ is  logarithmic. Indeed, we have that
$$
d(\omega/f)= (1/f)\left(-(df\wedge \omega)/f+d\omega\right)
$$
and hence $f d(\omega/f)$ is holomorphic.
\end{remark}

The following result is well known:
\begin{proposition}
  Let $\eta$ be the germ of a closed logarithmic $1$-form on $({\mathbb C}^n,0)$. There is a multivaluated function $F$ such that $\eta=dF/F$. More precisely, if we decompose the set of poles as a union $\operatorname{Pol}(\eta)=\cup_{i=1}^sH_i$ of irreducible hypersurfaces, there are  $\lambda_i\ne 0$  and reduced local equations $f_i=0$ for each   $H_i$  such that
  $$
  \eta=\sum_{i=1}^s\lambda_i\frac{df_i}{f_i}.
  $$
  Moreover, the coefficients $\lambda_i$ are unique. In the case that $\operatorname{Pol}(\eta)=\emptyset$ and hence $\eta$ is holomorphic, the statement must be interpreted by saying that there is a unit $U$ such that $\eta=dU/U$.
\end{proposition}
\begin{proof}(See \cite{Cer-M, Sai}) If $\eta$ is holomorphic, by Poincaré Lemma, there is a holomorphic function $G$ such that $\eta=dG$, taking $U=\exp(G)$, we have that $\eta=dU/U$. We know that the residue of $\eta$ along $H_i$ is non zero and constant (see (2.6) and the proof of Theorem (2.9) in \cite{Sai}), let us call $\lambda_i\in {\mathbb C}$ this residue. Taking local reduced equations $g_i=0$ of $H_i$ we have that
$$
\alpha=\eta- \sum_{i=1}^s\lambda_i\frac{dg_i}{g_i}
$$
is a closed logarithmic differential $1$-form without residues. Hence $(1/\lambda_1)\alpha$ is a closed holomorphic $1$-form and thus $(1/\lambda_1)\alpha=dU/U$, where $U$ is a unit.  Put $f_1=Ug_1$ and $f_i=g_i$, for $i=2,3,\ldots,s$. We conclude that $\eta=\sum_{i=1}^s \lambda_idf_i/f_i$.
\end{proof}

 Given a closed logarithmic differential $1$-form $\eta$ on $M$, we attach to it the $\mathbb C$-divisor $\operatorname{Div}{\eta}$ given by
 $$
 \operatorname{Div}{\eta}=\sum_{H}\lambda_HH,
 $$
 where $\lambda_H=\operatorname{Res}_H(\eta)$ is the residue of $\eta$ along $H$, that we know to be constant by \cite{Sai}. When ${\mathcal D}=\operatorname{Div}(\eta)$, we say that the closed $1$-form $\eta$ is {\em ${\mathcal D}$-logarithmic} .

\begin{definition}
 A codimension one singular holomorphic foliation $\mathcal F$ on $M$ is {\em $\mathcal D$-logarithmic} when it is locally generated by a closed $\mathcal D$-logarithmic differential $1$-form.
\end{definition}
Let us note that $\operatorname{Div}(\mu\eta)=\mu\operatorname{Div}(\eta)$, when $\mu\in{\mathbb C}$. Thus,
 if ${\mathcal F}$ is ${\mathcal D}$-logarithmic and ${\mathcal D}'$ is a $\mathbb C$-divisor projectively equivalent to ${\mathcal D}$, then ${\mathcal F}$ is also ${\mathcal D}'$-logarithmic.
\begin{remark}
\label{rk:invarianciasoporte}
If $\mathcal F$ is a $\mathcal D$-logarithmic foliation, the irreducible components of the support of $\mathcal D$ are invariant for $\mathcal F$.  This may be verified locally, assuming that $\eta=\sum_{i=1}^s\lambda_i df_i/f_i$ generates $\mathcal F$. We have that $f_i$ does not divide the coefficients of the holomorphic $1$-form $\omega=f\eta$, where $f=\prod_{i=1}^sf_i$, indeed  $f_i$ does not divide $df_i$, since $f_i$ is reduced; in this situation, we have only to verify that $f_i$ divides $df_i\wedge \omega$, and this condition is visible.
\end{remark}
\begin{remark} Consider the radial foliation ${\mathcal R}$ on $({\mathbb C}^2,0)$ defined by $\omega=0$ where $\omega=ydx-xdy$. Note that $\mathcal R$ is defined both by $\eta_0$ and $\eta_1$, where
$$
\eta_0=\frac{dx}{x}-\frac{dy}{y},\quad \eta_1=\frac{d(x+y)}{x+y}- \frac{d(x-y)}{x-y}.
$$
Put $H^0_1=(x=0)$, $H^0_2=(y=0)$, $H^1_1=(x+y=0)$ and $H^1_2=(x-y=0)$.
The $\mathbb C$-divisors
$$
\operatorname{Div}(\eta_0)= H^0_1-H^0_2,\quad \operatorname{Div}(\eta_1)=H^1_1-H^1_2
$$ are different and not proportional. Hence a codimension one foliation can be logarithmic with respect to several non projectively equivalent $\mathbb C$-divisors.
\end{remark}

\subsection{Dicriticalness}
The word dicritical comes from ancient works of Autom, following Mattei \cite{Cer-M}. The general definition of dicritical foliation, suggested by D. Cerveau, may be found in \cite{Can-RV-S}:
\begin{definition}
\label{def:foldicr} Let $\mathcal F$ be a codimension one holomorphic foliation on a non-singular complex analytic variety $M$. We say that $\mathcal F$ is {\em dicritical at a point $p\in M$} if and only if there is a holomorphic map
$$
\phi:({\mathbb C}^2,0)\rightarrow (M,p)
$$
such that $\phi^*{\mathcal F}=(dx=0)$ and $\phi(y=0)$ is invariant for $\mathcal F$.
We say that $\mathcal F$ is {\em dicritical} if there is a point $p$ such that it is dicritical at $p$. When $M$ is a germ $(M,K)$ over a compact set $K\subset M$, we ask the condition just for the points in $K$.
\end{definition}

As in the case of $\mathbb C$-divisors, we adopt the notation
$$
\pi:(M',{\mathcal F}')\rightarrow (M,{\mathcal F})
$$
to indicate a morphism $\pi:M'\rightarrow M$, a foliation $\mathcal F$ on $M$ and the transform ${\mathcal F}'=\pi^*{\mathcal F}$. When $\pi$ is a blowing-up with non-singular center $Y$, we say that $\pi$ is {\em $\mathcal F$-admissible} if the center $Y$ is invariant for $\mathcal F$, we say that $\pi$ is a {\em dicritical blowing-up} if the exceptional divisor $\pi^{-1}(Y)$ is not invariant for ${\mathcal F}'$ and it is {\em non-dicritical} when the exceptional divisor is invariant for ${\mathcal F}'$.
\begin{proposition}
\label{prop:dicriticidaddeunaexplosion}
Let $\mathcal F$ be a codimension one singular foliation on $({\mathbb C}^n,0)$ and assume that $Y$ is a non-singular invariant subvariety of $({\mathbb C}^n,0)$. If the blowing-up
$$
\pi:((M,\pi^{-1}(0)),{\mathcal F}')\rightarrow (({\mathbb C}^n,0),{\mathcal F})
$$
centered at  $Y$ is a dicritical blowing-up, then $\mathcal F$ is a dicritical foliation.
\end{proposition}
\begin{proof} Choose a point $p\in E=\pi^{-1}(Y)$ and consider local coordinates
$x_1,x_2,\ldots,x_n$
at $p$ such that $E=(x_1=0)$ and $x_1=x_2=\ldots=x_{n-1}=0$ is not invariant for ${\mathcal F}'$. This is possible, since not all the non-singular branches through $p$ contained in $E$ are invariant for ${\mathcal F}'$ (this should imply that $E$ itself is invariant). Now, let
$
\psi:
({\mathbb C}^2,0)\rightarrow (M,p)\hookrightarrow (M,\pi^{-1}(0))
$
be the map given by
$$
 x_1 \circ \psi=v,\;  x_n \circ \psi=u,\;  x_i \circ \psi=0,\, i=2,3,\ldots,n-1,
$$
where $u,v$ are local coordinates in $({\mathbb C}^2,0)$. We know that  $\Gamma=(v=0)$ is not invariant for $\psi^*{\mathcal F}'$. Let $\sigma: ({\mathbb C}^2,0)\rightarrow ({\mathbb C}^2,0)$ be the composition of a sequence of local blowing-ups following the infinitely near points of $\Gamma$ such that the strict transform of $\Gamma$ is $y=0$ and $\sigma^*(\psi^*{\mathcal F}')$ is the foliation $dx=0$. This is possible, since we do the reduction of singularities both of $\Gamma$ and $\psi^*{\mathcal F}'$. We end by considering $\phi=\pi \circ \psi\circ\sigma$, where $\phi^*({\mathcal F})=(dx=0)$ and $\phi(y=0)\subset Y$ is invariant.
\end{proof}

\begin{remark} When $M$ has dimension two, we have that $\mathcal F$ is dicritical at $p$ if and only if there are infinitely many germs of invariant branches of $\mathcal F$ at $p$ and this is also equivalent to say that we can find a sequence of blowing-ups ended by a dicritical one. This property is the classical definition of dicritical foliation in dimension two. Nevertheless, the direct generalization to higher dimension is not evident, as Jouanolou's example \cite{Jou} show: a germ of foliation $\mathcal F$ in $({\mathbb C}^3,0)$ without invariant surface, but such that the blowing-up of the origin is dicritical. See \cite{Can-C, Can-C-D} for more details.
\end{remark}

\begin{proposition}
 \label{prop:estabilidddicriticidad}
 Let $\pi:(M',{\mathcal F}')\rightarrow(M,{\mathcal F})$ be an admissible non-dicritical  blowing-up. Then ${\mathcal F}$ is a dicritical foliation if and only if ${\mathcal F}'$ is so.
\end{proposition}
\begin{proof} Assume that ${\mathcal F}'$ is dicritical. Take a holomorphic map $\phi':({\mathbb C}^2,0)\rightarrow M'$ such that ${\phi'}^*{\mathcal F}'=(dx=0)$ and ${\phi'(y=0)}\subset M'$ is invariant. Put $\phi=\pi \circ \phi'$. Since $\pi$ is non dicritical, we have that
$$
\phi^*{\mathcal F}={\phi'}^*(\pi^*{\mathcal F})
$$
(the non-dicriticalness of $\pi$ is necessary here, see the Remark \ref{rk:pullbackdicriticao}
below) and moreover $\phi(y=0)=\pi(\phi'(y=0))$ is invariant. Hence ${\mathcal F}$ is also a dicritical foliation.

Conversely, let us assume that ${\mathcal F}$ is dicritical and take a holomorphic map
$$
\phi:({\mathbb C}^2, 0)\rightarrow M
 $$
 such that $\phi^*{\mathcal F}=(dx=0)$ and $\phi(y=0)$ is invariant for ${\mathcal  F}$. In view of Proposition \ref{prop:appdos} in the Appendix I, there is a commutative diagram of morphisms
  $$
 \begin{array}{ccc}
 ({\mathbb C}^2,0)&\stackrel{\sigma}{\longleftarrow}&N\\
 \phi\downarrow\;\; &&\;\;\downarrow\psi\\
 M&\stackrel{\pi}{\longleftarrow}&M'
 \end{array},
 $$
where $\sigma$ is the composition of a finite sequence of blowing-ups. Let $(\Gamma',p')$ be the strict transform of $(y=0)$ by $\sigma$. We know that there are local coordinates $u,v$ at $p'$ such that $\Gamma'=(v=0)$ and $\sigma^*{(dx=0)}=(du=0)$. Note that
$$
(\phi\circ\sigma)^*{\mathcal F}=(\pi\circ\psi)^*{\mathcal F}.
$$
Since $\pi$ is non-dicritical, we have that $(\pi\circ\psi)^*{\mathcal F}=\psi^*{\mathcal F}'$. Moreover, since $\sigma$ is a sequence of blowing-ups centered at points, we have that
$$
(\phi\circ\sigma)^*{\mathcal F}=\sigma^*(\phi^*{\mathcal F})=(du=0).
$$
 Hence $\psi^*{\mathcal F}'=(du=0)$ and $\mathcal F'$ is a dicritical foliation.
\end{proof}
\begin{remark}
\label{rk:pullbackdicriticao}
Let $\mathcal F$ be a codimension one singular foliation of $M$ and consider two morphisms
$\phi:M'\rightarrow M$ and $\psi:M''\rightarrow M'$. The foliation $\phi^*{\mathcal F}$ is
defined locally by the pullback $\phi^*\omega$ of a differential $1$-form $\omega$ defining $\mathcal F$. The pull-back foliation $\phi^*{\mathcal F}$ {\em exists}, or {\em is defined}, if and only if $\phi^*\omega\ne 0$, when $\omega$ is chosen to be without common factors in its coefficients. When $(\phi\circ\psi)^*{\mathcal F}$, $\phi^*{\mathcal F}$ and $\psi^*(\phi^*{\mathcal F})$ exist, we have that
$$
(\phi\circ\psi)^*{\mathcal F}=\psi^*(\phi^*{\mathcal F}),
$$
but it is possible for  $\phi^*{\mathcal F}$ and $\psi^*(\phi^*{\mathcal F})$ to be well defined, whereas $(\phi\circ\psi)^*{\mathcal F}$ does not exist. An important case of this situation is the immersion $\psi: M''\rightarrow M'$ of the exceptional divisor of a dicritical blowing-up $\phi:M'\rightarrow M$, see \cite{Can-L-M}.

Anyway, when $\phi$ is a non-dicritical blowing-up, and hence the exceptional divisor is invariant for $\phi^*{\mathcal F}$, we have that $\phi^*{\mathcal F}$ exists (this is always true because a blowing-up is an isomorphism in a dense open set) and moreover $\psi^*(\phi^*{\mathcal F})$ is defined if and only if $(\phi\circ\psi)^*{\mathcal F}$ is defined; hence we have the equality.

On the other hand, when $\psi$ is a blowing-up or a sequence of blowing-ups, we also have that $\phi^*{\mathcal F}$ is defined if and only if $(\phi\circ\psi)^*{\mathcal F}$ is defined and, if this is the case, we also have that
$
(\phi\circ\psi)^*{\mathcal F}=\psi^*(\phi^*{\mathcal F})
$.

\end{remark}

\subsection{Non-dicritical Logarithmic Foliations} In this Subsection we relate the non-dicriticalness of a ${\mathcal D}$-logarithmic foliation with the same property for the $\mathbb C$-divisor $\mathcal D$.

\begin{lemma}
\label{lema:unaexplosion}
 Let $\mathcal F$ be a $\mathcal D$-logarithmic foliation. Assume that
$
\pi:(M',{\mathcal D}')\rightarrow (M,{\mathcal D})
$
is a non-dicritical ${\mathcal D}$-admissible blowing-up. Then $
\pi:(M',{\mathcal F}')\rightarrow (M,{\mathcal F})
$
is an admissible non-dicritical blowing-up
and ${\mathcal F}'$ is ${\mathcal D}'$-logarithmic.
\end{lemma}
\begin{proof} Let $Y$ be the center of $\pi$. We know that $Y\subset\operatorname{Supp}({\mathcal D})$ and hence $Y$ is ${\mathcal F}$-invariant, since
the support is $\mathcal F$-invariant, in view of Remark \ref{rk:invarianciasoporte}. Then $\pi$ is $\mathcal F$-admissible.
Put ${\mathcal D}=\sum_{i=1}^s\lambda_i H_i$ and assume that $\mathcal F$ is generated by
$$
\eta=\sum_{i=1}^s\lambda_i\frac{df_i}{f_i},
$$
 where $f_i=0$ is a reduced local equation of $H_i$ for $i=1,2,\ldots, s$. Then ${\mathcal F}'$ is generated by $\pi^*\eta$, where
 $$
 \pi^*\eta= \sum_{i=1}^s\lambda_i\frac{d(f_i\circ\pi)}{f_i\circ \pi}.
 $$
 Moreover, we have  that
 $$
 {\mathcal D}'=\pi^{*}{\mathcal D}=\sum_{i=1}^s\lambda_i \pi^*(1\cdot H_i)=\mu E+\sum_{i=1}^s\lambda_i H'_i,
 $$
 where $E=\pi^{-1}(Y)$, $\mu=\sum_{i=1}^s\lambda_i\nu_i$, with $\nu_i=\nu_Y(H_i)$ and $H'_i$ stands for the strict transform of $H_i$. By hypothesis, we have that $\mu\ne 0$.
 We can do the necessary verifications locally at the points in $E$. Take one such $q\in E$ and let $h=0$ be a local reduced equation of $E$ at $q$. We have that
 $$
 f_i\circ \pi=h^{\nu_i}f'_i,
 $$
 where $f'_i=0$ is a local reduced equation for the strict transform $H'_i$ of $H_i$ at $q$. Let us show that $\pi^*\eta$ can be written as
 \begin{equation}
 \label{eq:transformado}
 \pi^*\eta=\mu\frac{dh'}{h'}+\sum_{q\in H'_i}\lambda_{i}\frac{df'_i}{f'_i},
 \end{equation}
 where $h'=0$ is a reduced local equation of $\pi^{-1}(Y)$ at $q$.
 Recalling that $\mu\ne 0$, we see that $\pi^{-1}(Y)$ is invariant for ${\mathcal F}'$ and hence $\pi:( M', {\mathcal F}')\rightarrow (M,{\mathcal F})$ is non dicritical; moreover Equation \ref{eq:transformado} also shows that ${\mathcal F}'$ is ${\mathcal D}$-logarithmic.

 It remains to find $h'$ satisfying Equation \ref{eq:transformado}. Note that $f'_i$ is a unit if and only if $q\notin H'_i$.  In this situation,  there is a unit $U$ such that
 $$
 \mu\frac{dU}{U}= \sum_{q\notin H'_i}\lambda_{i}\frac{df'_i}{f'_i}.
 $$
 Now, it is enough to take $h'=Uh$.
\end{proof}

\begin{remark}  It is possible to have a dicritical ${\mathcal D}$-admissible blowing-up
$$
\pi:(M',{\mathcal D}')\rightarrow (M,{\mathcal D})
$$
and a ${\mathcal D}$-logarithmic foliation ${\mathcal F}$ such that $\pi$ induces a non-dicritical admissible blowing-up $
\pi:(M',{\mathcal F}')\rightarrow (M,{\mathcal F})
$. The following example may be found in  \cite{Can-Co}: take the foliation $\mathcal F$ on $({\mathbb C}^2,0)$ given by $\eta=0$ where
$$
\eta=\frac{d(y-x^2)}{y-x^2}-\frac{d(y+x^2)}{y+x^2}.
$$
Then $\mathcal F$ is $\mathcal D$-logarithmic for ${\mathcal D}=(y-x^2=0)-(y+x^2=0)$.
Note that $\mathcal F$ is also ${\mathcal D}_1$-logarithmic, where ${\mathcal D}_1=(y=0)-2(x=0)$.  The  first blowing-up $\pi$ is ${\mathcal D}$-dicritical, but the exceptional divisor is invariant for the transformed foliation. Anyway, we know that in ambient dimension two, this situation implies that ${\mathcal F}$ is actually a dicritical foliation, although the blowing-up $\pi$ could be non dicritical. In general, it is an open question to know if given a logarithmic foliation $\mathcal F$ there is a $\mathbb C$-divisor that ``faithfully'' represents the dicriticalness of the foliation, for instance in terms of blowing-ups. In this paper, we concentrate ourselves in the non-dicritical case.
\end{remark}

\begin{proposition}
 \label{pro:hipersuperficiesinvariantes}
 Let $\mathcal F$ be a $\mathcal D$-logarithmic foliation, where $\mathcal D$ is non dicritical and take a point $p\in \operatorname{Supp}({\mathcal D})$.  The only irreducible germs of hypersurface at $p$ invariant for $\mathcal F$ are the irreducible components of the germ at $p$ of support of $\mathcal D$.
\end{proposition}
\begin{proof} (See also \cite{Cer-M}). We can assume that $M=({\mathbb C}^n,0)$,  $0\ne{\mathcal D}=\sum_{i=1}^s\lambda_iH_i$ and $\mathcal F$ is generated by
$$
\eta=\sum_{i=1}^s\lambda_i\frac{df_i}{f_i},
$$
where $f_i=0$ are reduced equations of $H_i$, for $i=1,2,\ldots,s$. By Remark \ref{rk:invarianciasoporte}, we already know that each $H_i$ is invariant for $\mathcal F$.
Let us suppose now that $S$, given by $g=0$, is another germ of irreducible hypersurface invariant for $\mathcal F$. Up to make a desingularization of the support of $\mathcal D$ and by choosing a point where the strict transform of $S$ intersects the exceptional divisor, we restrict ourselves to the case when the $H_i$ are coordinate hyperplanes.  There is at least one of the components of the support, because of the non dicriticalness of ${\mathcal D}$, that makes us to add each time we blow-up the exceptional divisor to the support of $\mathcal D$. Then, we can assume that
$$
\eta=\sum_{i=1}^s\lambda_i\frac{dx_i}{x_i}.
$$
The non-dicriticalness of ${\mathcal D}$ implies in this situation that $\sum_{i=1}^sm_i\lambda_i\ne 0$ for any $0\ne \mathbf{m}\in {\mathbb Z}_{\geq 0}^s$, in view of Lemma \ref{lema:normalcrossings}. By the curve selection lemma, there is a parameterized curve
$$
\gamma:t\mapsto (\gamma_i(t))_{i=1}^n
$$
contained in $S$ and not contained in the support $\prod_{i=1}^sx_i=0$, in particular $\gamma_i(t)\ne 0$, for any $i=1,2,\ldots,s$. Let us write
$$
\gamma_i(t)=\mu_{im_i}t^{m_i}+\mu_{i,m_i+1}t^{m_i+1}+\cdots,\quad \mu_{im_i}\ne 0, \quad i=1,2,\ldots,s.
$$
 Since $S$ is invariant, we have that $\gamma^*\eta=0$. Looking at the residue of $\gamma^*\eta$, we have that
 $$
 \sum_{i=1}^sm_i\lambda_i=0.
 $$
 This is not possible.
\end{proof}

\begin{corollary} Let $\mathcal F$ be a $\mathcal D$-logarithmic foliation, where $\mathcal D$ is non dicritical. Then ${\mathcal F}$ is also non dicritical. Moreover, if $\mathcal F$ is ${\mathcal D}'$-logarithmic, then ${\mathcal D}'$ and ${\mathcal D}$ are projectively equivalent $\mathbb C$-divisors.
\end{corollary}
\begin{proof}
Let us desingularize $\operatorname{Supp}{\mathcal D}$ by means of non-dicritical admissible blowing-ups. Note that, by Lemma \ref{lema:unaexplosion}, the above blowing ups are non-dicritical for $\mathcal F$. Invoking Proposition  \ref{prop:estabilidddicriticidad}, we reduce the problem to the case when $\operatorname{Supp}({\mathcal D})$ has normal crossings. Now, if ${\mathcal F}$ is dicritical we get a nonnegative resonance in the coefficients of ${\mathcal D}$ and hence ${\mathcal D}$ should be dicritical, by Lemma \ref{lema:normalcrossings}.

We can also do the above reduction in order to prove the second part of the statement. Assume thus that $\operatorname{Supp}({\mathcal D})$ has normal crossings, the foliation is ${\mathcal D}$-logarithmic and non dicritical. Now, the support of ${\mathcal D}$ coincides with the invariant hypersurfaces of ${\mathcal F}$ by Proposition \ref{pro:hipersuperficiesinvariantes};
moreover, the coefficients (up to projective equivalence) are given by the residues. Hence ${\mathcal D}'$ is determined up to projective equivalence from ${\mathcal D}$.
\end{proof}

\begin{remark} Let $\mathcal F$ be a $\mathcal D$-logarithmic foliation. We know that if $\mathcal D$ is non-dicritical, then the foliation $\mathcal F$ is also non-dicritical. The converse is a natural question that has positive answer. That is, if $\mathcal F$ is non-dicritical, then $\mathcal D$ is also non-dicritical. This is a consequence of the theorem on existence and non-dicriticalness of the logarithmic models, that we prove in this paper.
\end{remark}

\section{Generalized Hypersurfaces and Logarithmic Forms}
We recall here some facts useful for the sequel, concerning generalized hypersurfaces and the more general case of non-dicritical codimension one singular foliations. For more details on generalized hypersurfaces the reader can look at \cite{Fer-M}. We end the section by associating logarithmic forms to the generalized hypersurfaces that are stable under blowing-ups.

We take the following definition:
\begin{definition}[\cite{Can-RV-S}]
Given a complex analytic variety $M$, a foliation $\mathcal F$ on $M$ and a point $p\in M$, we say that $\mathcal F$ is {\em complex hyperbolic} at $p$, or that $\mathcal F$ {\em has no hidden saddle-nodes at $p$},  if and only if there is no holomorphic map $\phi:({\mathbb C}^2,0)\rightarrow M$, with $\phi(0)=p$, such that $\phi^*{\mathcal F}$ is a saddle-node. We say that $\mathcal F$ is a {\em generalized hypersurface at $p$} if, in addition, it is non-dicritical at $p$. We say that $\mathcal F$ is a generalized hypersurface at $M$ when the property holds at each point of $M$.
\end{definition}

The origin of the terminology {\em generalized curve} is in the paper \cite{Cam-LN-S}, where the authors made an extensive consideration of the condition of being complex hyperbolic, in the two dimensional case. In some cases, the above name is used also for the dicritical situation. We fix the word {\em generalized hypersurface} for denoting both properties: non-dicriticalness and no hidden-saddle-nodes. Of course, in the case of ambient dimension two, the expression {\em generalized curve} also means for us to be non-dicritical and without hidden saddle-nodes.

\begin{remark} Some ``ramified'' saddle-nodes have the property of being generalized hypersurfaces. For instance, take the saddle-node given by the meromorphic $1$-form
$$
\frac{du}{u}+ u\frac{dv}{v}
$$
in dimension two. Let us consider the ramification $u=x^py^q$, $v=y$; we obtain by pull back a differential $1$-form
$$
\left(p\frac{dx}{x}+q\frac{dy}{y}\right)+ x^py^q\frac{dy}{y}
$$
that defines a generalized curve on $({\mathbb C}^2,0)$; it is an example of Martinet-Ramis resonant case \cite{Mar-R}. Note that it has no holomorphic first integral. The reader can see \cite{Can-C-D} for more details.
\end{remark}

One of the important features of generalized hypersurfaces is the following result:
\begin{proposition}[See \cite{Can,Can-M-RV,Fer-M}]
\label{prop:redsinghypgeneralizada}
Let $\mathcal F$ be a generalized hypersurface on $({\mathbb C}^n,0)$. There are only finitely many irreducible invariant hypersurfaces of $\mathcal F$, its union $S$ is non empty and any reduction of singularities of $S$ provides a reduction of singularities of $\mathcal F$. In particular, the singular locus of $\mathcal F$ is contained in $S$.
\end{proposition}

Let us state some other useful results concerning generalized hypersurfaces:
\begin{lemma}
\label{lema:curvainvariante}
Consider a generalized hypersurface  $\mathcal F$ on $({\mathbb C}^n, 0)$ and take an invariant analytic branch $(\Gamma,0)\subset ({\mathbb C}^n,0)$ not contained in the singular locus of $\mathcal F$. There is a single irreducible hypersurface $H$ invariant for $\mathcal F$ such that $\Gamma\subset H$.
\end{lemma}
\begin{proof} This is true for any non-dicritical foliation that admits a reduction of singularities, in particular for generalized hypersurfaces, see \cite{Can}.
\end{proof}

\begin{proposition}
\label{prop:pullbackgeneralizedcurve}
Consider a generalized hypersurface  $\mathcal F$ on $({\mathbb C}^n, 0)$ and let $S$ be the union of the invariant hypersurfaces of $\mathcal F$. For any $S$-transverse holomorphic map
$
\phi:({\mathbb C}^2,0)\rightarrow ({\mathbb C}^n,0)
$, the pull-back $\phi^*{\mathcal F}$ is a generalized curve.
\end{proposition}
\begin{proof} Let $\omega$ be a reduced holomorphic generator of $\mathcal F$. We have first  to show that $\phi^*\omega\ne 0$. Assume that $\phi^*\omega=0$, since $\phi$ is $S$-transverse, there is an irreducible branch $(\Gamma,0)\subset ({\mathbb C}^2,0)$ such that $\phi(\Gamma)\not\subset S$. In this situation, the curve $\phi(\Gamma)$ is an invariant curve of $\mathcal F$ not contained in $S$; this contradicts Lemma \ref{lema:curvainvariante}. Then, we have that $\phi^*{\mathcal F}$ exists. More precisely, the invariant curves for $\phi^*{\mathcal F}$ are precisely the irreducible components of $\phi^{-1}(S)$. In particular, $\phi^*{\mathcal F}$  has only finitely many invariant branches and then it is non-dicritical. Finally, let us find a contradiction if $\phi^*{\mathcal F}$ is not complex hyperbolic.
Take
$$
\varphi:({\mathbb C}^2,0)\rightarrow ({\mathbb C}^2,0)
$$
such that $\varphi^*(\phi^*{\mathcal F})$ is a saddle-node. We have that $\operatorname{Im}(\varphi)\not\subset \phi^{-1}(S)$ and hence $\phi\circ\varphi$ is $S$-transverse. We conclude that $(\phi\circ\varphi)^*{\mathcal F}$ exists and
$$
(\phi\circ\varphi)^*{\mathcal F}=\varphi^*(\phi^*{\mathcal F})
$$
is a saddle-node, contradiction, since $\mathcal F$ is complex hyperbolic.
\end{proof}

Next, we show the stability under permissible blowing-ups of being generalized hypersurface

\begin{proposition}
\label{prop:blowingupgeneralizedhip}
Consider a generalized hypersurface  $\mathcal F$ on $({\mathbb C}^n, 0)$ and let $Y$ be a non-singular subvariety of $({\mathbb C}^n,0)$ contained in the union $S$
of the invariant hypersurfaces of $\mathcal F$. Consider the admissible blowing-up
$$
\pi:((M,\pi^{-1}(0)), {\mathcal F}')\rightarrow (({\mathbb C}^n,0),{\mathcal F})
$$
with center $Y$. Then $\pi$ is non-dicritical and ${\mathcal F}'$ is a generalized hypersurface.
\end{proposition}
\begin{proof} By Proposition \ref{prop:dicriticidaddeunaexplosion}, we see that the blowing-up $\pi$ is non-dicritical, since  $\mathcal F$ is a non-dicritical foliation. Moreover, the transformed foliation
 ${\mathcal F}'$ is non-dicritical in view of Proposition \ref{prop:estabilidddicriticidad}. Let us show that ${\mathcal F}'$ is complex-hyperbolic. Assume by contradiction that it is not and thus there is a point $p\in \pi^{-1}(0)$ and a morphism
 $$
 \phi: ({\mathbb C}^2,0)\rightarrow (M,p)
 $$
 such that $\phi^*{\mathcal F}'$ is a saddle-node. Since $\pi$ is non-dicritical, we have that the exceptional divisor $E=\pi^{-1}(Y)$ is invariant and hence the image of $\phi$ is not contained in $E$; this implies that $(\pi\circ\phi)^*{\mathcal F}$ exists and, in view of Remark \ref{rk:pullbackdicriticao}, we have
 $$
 (\pi\circ\phi)^*{\mathcal F}=\phi^*(\pi^*{\mathcal F})=\phi^*{\mathcal F}'.
 $$
 It is a saddle-node, contradiction.
\end{proof}

\subsection{Transversality} We consider here the concepts of generic multiplicity, equimultiplicity and Mattei-Moussu transversality, that we need in the proof of existence of logarithmic model for generalized hypersurfaces.

Let $Y$ be a  non-singular irreducible subvariety of $({\mathbb C}^n,0)$ and consider a
holomorphic $1$-form $\omega$ on $({\mathbb C}^n,0)$. The {\em generic multiplicity $\nu_Y(\omega)$ of $\omega$ along $Y$} is the minimum of the generic multiplicity of the coefficients of $\omega$ along $Y$. When $\omega$ is a reduced (no common factors in the coefficients) generator of a codimension one singular foliation $\mathcal F$ on $({\mathbb C}^n,0)$ we say that $\nu_Y(\omega)$ is the generic multiplicity of $\mathcal F$ along $Y$ and we denote $\nu_Y({\mathcal F})=\nu_Y(\omega)$.

Let $S$ be a hypersurface of $({\mathbb C}^n,0)$ with reduced equation $f=0$; recall that
$\nu_Y(S)=\nu_Y(f)$. We say that $Y$ is {\em equimultiple at the origin} for $\omega$, $\mathcal F$ or $S$, if we respectively have that
$$
\nu_Y(\omega)=\nu_0(\omega),\quad \nu_Y({\mathcal F})=\nu_0(\mathcal F),\quad \nu_Y(H)=\nu_0(H).
$$
By taking appropriate representatives of the germs, we know that the points of equimultiplicity define a dense open set in $Y$.
\begin{remark}
  If $S$ is given by the reduced equation $f=0$, where $f$ is reduced, and we consider the foliation ${\mathcal F}=(df=0)$, we have  $\nu_Y({\mathcal F})=\nu_Y(S)-1$.
\end{remark}

Let us recall that the {\em singular locus $\operatorname{Sing}({\mathcal F})$} of a foliation $\mathcal F$ coincides locally with the singular locus of a holomorphic generator of $\mathcal F$ without common factor in its coefficients. In particular, we have that  $\operatorname{Sing}({\mathcal F})\subset M$ is an analytic subset of codimension at least two.

Take a holomorphic germ of $1$-form $\omega$ on $({\mathbb C}^n,0)$ such that $\operatorname{codim}(\operatorname{Sing}(\omega))\geq 2$.
Following \cite{Mat-Mou}, we say that a closed immersion $\phi:({\mathbb C}^2,0)\rightarrow ({\mathbb C}^n,0)$ is a {\em Mattei-Moussu transversal} for $\omega$ when the following properties hold
$$
\operatorname{Sing}(\phi^*\omega)=\phi^{-1}(\operatorname{Sing}(\mathcal F))\subset \{0\},\quad \nu_0(\phi^*\omega)=\nu_0(\omega).
$$
If $\mathcal F$ is a codimension one singular foliation on $({\mathbb C}^n,0)$ we say that $\phi$ is a {\em Mattei-Moussu transversal} for $\mathcal F$ when it is a Mattei-Moussu transversal for a holomorphic generator of $\mathcal F$. Let $S$ be a hypersurface given by a reduced equation $f=0$; we say that $\phi$ is a {\em Mattei-Moussu transversal} for $S$ when it is a Mattei-Moussu transversal for the foliation $df=0$. In this paper, we consider the following version of the Transversality Theorem of Mattei-Moussu:
\begin{theorem} Let $\mathcal F$ be a non-dicritical holomorphic foliation on $({\mathbb C}^n,0)$. There is a Zariski nonempty open set $W$ in the space of linear two-planes such that any closed immersion $\phi:({\mathbb C}^2,0)\rightarrow ({\mathbb C}^n,0)$ with tangent plane in $W$ is a Mattei-Moussu transversal for $\mathcal F$.
\end{theorem}
\begin{proof} See \cite{Mat-Mou} and \cite{Can3,Can2, Can-M}.
\end{proof}

We have the following consequence:
\begin{proposition}
\label{prop:multiplicidadgenericagh}
  Let $\mathcal F$ be a generalized hypersurface of $({\mathbb C}^n,0)$ and denote by $S$ the union of the invariant hypersurfaces of $\mathcal F$. Consider a non-singular subvariety $(Y,0)$ of $({\mathbb C}^n,0)$ with $Y\subset S$. Then $\nu_Y({\mathcal F})=\nu_Y(S)-1$.
\end{proposition}
\begin{proof} We first reduce the problem to the case $Y=\{0\}$ as follows. Taking appropriate representatives of the germs, there is a dense open subset $U$ of $Y$ such that both $S$ and $\mathcal F$
are equimultiple along $Y$ at the points in $U$, that is, we have
$$
\nu_p(S)=\nu_Y(S),\quad  \nu_p({\mathcal F})=\nu_Y({\mathcal F}),
$$
for any $p\in U$. Thus, working at a point of equimultiplicity, we can assume that $Y=\{0\}$.
Now, we apply Mattei-Moussu Transversality Theorem to get a closed immersion $\phi:({\mathbb C}^2,0)\rightarrow ({\mathbb C}^n,0)$ such that
$$
\nu_0(\phi^*{\mathcal F})=\nu_0({\mathcal F}); \quad \nu_0(\phi^{-1}(S))=\nu_0(S).
$$
Since $\phi^*{\mathcal F}$ is a generalized curve (see Proposition \ref{prop:pullbackgeneralizedcurve}) we reduce the problem to the two dimensional case. In this case, the result is known from \cite{Cam-LN-S}.
\end{proof}

\subsection{Logarithmic Forms Fully Associated to Generalized Hypersurfaces}
Let us consider a generalized hypersurface  $\mathcal F$ on $({\mathbb C}^n, 0)$. We know that there exists at least a germ of invariant hypersurface and that there are finitely many of them
$H_i$, $i=1,2,\ldots, s$. Let us select a germ of reduced function
$$
f=f_1f_2\cdots f_s\in {\mathcal O}_{{\mathbb C}^n,0}
$$
such that $f_i=0$ is a local equation for $H_i$, for $i=1,2,\ldots,s$. Thus $f=0$ gives  a reduced equation of the union $S$ of the invariant hypersurfaces of $\mathcal F$. Take a local holomorphic generator $\omega$ of $\mathcal F$ without common factors in its coefficients.
The meromorphic  $1$-form $\eta=\omega/f$ also defines $\mathcal F$. In view of Remark \ref{rk:logaritmicgenerator}, we know that $\eta$ is a logarithmic differential $1$-form, although it is not necessarily closed. Such an integrable logarithmic $1$-form $\eta=\omega/f$ will be called {\em fully associated to $\mathcal F$}.

\begin{remark} Assume that $\eta=\omega/f$ and $\eta'=\omega'/f'$ are two integrable  logarithmic
$1$-forms fully associated to $\mathcal F$. There are units $U,V\in {\mathcal O}_{{\mathbb C}^n,0}$,
such that $\omega'=U\omega$ and $f'=Vf$ and hence there is a unit $W=U/V$ such that $\eta'=W\eta$.
\end{remark}

\begin{proposition}
\label{prop:pullbackoflogaritmicformsfullyassociated}
Consider an integrable logarithmic $1$-form $\eta$ fully associated to a
generalized hypersurface  $\mathcal F$ on $({\mathbb C}^n, 0)$.  Take a non-singular irreducible subvariety $Y$ of $({\mathbb C}^n,0)$ invariant for $\mathcal F$ and let us perform the blowing-up centered at $Y$
$$
\pi:((M,\pi^{-1}(0)),{\mathcal F}')\rightarrow (({\mathbb C}^n,0),{\mathcal F}).
$$
 The pullback $\pi^*\eta$ is an integrable logarithmic $1$-form fully associated to ${\mathcal F}'$.
\end{proposition}
\begin{proof} (See \cite{Cer} for the two dimensional case). Let $\omega$ be a holomorphic generator of $\mathcal F$ without common factors in its coefficients, take a reduced equation $f=0$ of the union $S$ of the invariant hypersurfaces of $\mathcal F$ and put $\eta=\omega/f$. By Proposition \ref{prop:multiplicidadgenericagh}, we know that $$
\nu_Y(\omega)=\nu_Y(f)-1
.
$$ Put $m=\nu_Y(\omega)$. Working locally at a point $p$ of the exceptional divisor $\pi^{-1}(Y)$, where $x'=0$ is a reduced equation of $\pi^{-1}(Y)$, we know that $f\circ\pi=x'^{m+1}f'$, where $f'=0$ is a reduced local equation of the strict transform of $S$.
By Proposition
\ref{prop:dicriticidaddeunaexplosion},
we know that $\pi$ is a non-dicritical blowing-up for $\mathcal F$. Hence $x'f'=0$ is a local reduced equation of the union of the invariant hypersurfaces of ${\mathcal F}'$ at the point $p$. On the other hand, ${\mathcal F}'$ is generated by $\pi^*\omega$ and
 $$
 \pi^*\omega=x'^{m}\omega',
 $$
 where $\omega'$ has no common factors in its coefficients; for this, we use the fact that $\pi$ is a non dicritical blowing-up and thus we can divide $\pi^*\omega$ exactly by ${x'}^m$. Hence, we have that
 $$
\pi^*\eta=\frac{\pi^*\omega}{f\circ\pi}=\omega'/(x'f')
 $$
is an integrable  logarithmic $1$-form fully associated to $\pi^*{\mathcal F}$.
\end{proof}

\section{Divisorial Models in Dimension Two}
Consider a foliation $\mathcal F$ on $({\mathbb C}^2,0)$. From the work of A. Seidenberg \cite{Sei}, we know that there is an essentially unique reduction of singularities of $\mathcal F$. When there are no saddle-nodes after reduction of singularities and all the irreducible components of the exceptional divisor are invariant ones, we say that $\mathcal F$ is a {\em generalized curve} \cite{Cam-LN-S}. In this case, the Camacho-Sad indices at the singular points, after reduction of singularities, are all nonzero and they determine locally the linear part of the holonomy. This motivates the quest of a foliation with linear holonomy, with the same reduction of singularities and such that the linear part of the holonomy is the same as the one for $\mathcal F$ after reduction of singularities. Such foliations are the logarithmic ones and hence we look for a ``logarithmic model'' of a given generalized curve. This problem has been solved in dimension two by N. Corral in \cite{Cor}. In this section we recover Corral results in the language of $\mathbb C$-divisors and the indices with respect to singular invariant curves.

\subsection{Indices for $\mathbb C$-divisors in dimension two}
We develop here a notion of index for ${\mathbb C}$-divisors, directly inspired in the behavior of Camacho-Sad index in the case of holomorphic foliations in dimension two.

Let $M$ be a non-singular complex variety of dimension two. Take  a $\mathbb C$-divisor
\begin{equation}
\label{eq:escrituraparaindice}
 {\mathcal D}=\mu \operatorname{Div}(T)+\sum_{i=2}^s\lambda_i\operatorname{Div}(H_i)
\end{equation}
on it, where $T\subset M$ and $H_i\subset M$ are curves in $M$, not necessarily irreducible, such that none of the irreducible components of $T$ is an irreducible component of an $H_i$, for $i=2,3,\ldots,s$. We assume that the support of $\mathcal D$ contains $T$, that is $\mu\ne 0$. Let us take a point $p\in T$. We define the {\em Camacho-Sad index $I_p({\mathcal D},T)$ at $p$ of $\mathcal D$ with respect to $T$} by the expression
\begin{equation}
\label{eq:indice}
I_p({\mathcal D},T)=-\frac{\sum_{i=2}^s\lambda_i (T,H_i)_p}{\mu},
\end{equation}
where $(T,H_i)_p$ stands for the intersection multiplicity of $T$ and $H_i$ at $p$.

Let us note that if ${\mathcal D}_p$ denotes the germ of $\mathcal D$ at $p$ we have that
$$
I_p({\mathcal D}_p,T)=I_p({\mathcal D},T).
$$
Let us remark that the germ of $T$ at $p$ may not be irreducible, even when we choose $T$ to be irreducible as a curve in $M$.
\begin{remark} It is possible to extend the above definition in order to define the index of $\mathcal D$ with respect to any union of curves in the support passing through $p$, by using the formula
$$
I_p({\mathcal D}, T_1\cup T_2)= I_p({\mathcal D}, T_1)+I_p({\mathcal D}, T_2)+2(T_1,T_2)_p.
$$
In this way, we could recover the complete definition of the Camacho-Sad index with respect to a not necessarily irreducible invariant curve, see \cite{Bru}. Anyway, we only need the definition for the case where the coefficients of all the irreducible components of $T$ in $\mathcal D$ are equal, as defined above.
\end{remark}

\begin{proposition} Let $\mathcal D$ be a $\mathbb C$-divisor on a non-singular two dimensional complex analytic variety $M$, that we write as ${\mathcal D}=\mu \operatorname{Div}(T)+\sum_{i=2}^s\lambda_i\operatorname{Div}(H_i)$, where $\mu\ne 0$ and $T$ and $H_i$ have no common irreducible components, for any $i=2,3,\ldots,s$.
Let $\pi:(M',{\mathcal D}')\rightarrow (M,{\mathcal D})$ be the blowing-up centered at a point $p\in T$. Denote by $T'$ the strict transform of $T$ by $\pi$ and by $E=\pi^{-1}(p)$ the exceptional divisor of $\pi$. The following equality holds:
\begin{equation}
\label{eq:indiceexplosion1}
\sum_{p'\in T'\cap E}I_{p'}({\mathcal D}',T')=I_p({\mathcal D},T)-\nu_p(T)^2.
\end{equation}
Moreover, if $\pi$ is non-dicritical, we have $\sum_{p'\in E}I_{p'}({\mathcal D}',E)=-1$.
\end{proposition}
\begin{proof} If $\alpha=\mu\nu_p(T)+\sum_{i=2}^s\lambda_i\nu_p(H_i)$, we have
 ${\mathcal D}'=\alpha E+\mu T'+\sum_{i=2}^s\lambda_iH'_i$, where
we denote by $H'_i$ the strict transforms of the $H_i$ by $\pi$.
Recall Noether's formulas:
\begin{equation*}
  (H_i,T)_p=\sum_{p'\in E\cap T'}(H'_i,T')_{p'}
+\nu_p(T)\nu_p(H_i);\quad
  \nu_p(T)=\sum_{p'\in E\cap T'}(E,T')_{p'} .
\end{equation*}
Let us show that $
\mu I_p({\mathcal D}, T)=\mu\nu_p(T)^2+\sum_{p'\in E}\mu I_{p'}({\mathcal D}',T')
$,
in order to verify the identity in Equation \ref{eq:indiceexplosion1}:
\begin{eqnarray*}
\mu I_p({\mathcal D},T)&=&-\sum_{i=2}^s\lambda_i(T,H_i)_p=
-\sum_{i=2}^s\lambda_i\nu_p(T)\nu_p(H_i)-\sum_{p'\in E}\sum_{i=2}^s\lambda_i(T',H'_i)_{p'}=
\\
&=&\mu\nu_p(T)^2-\nu_p(T)\alpha-\sum_{p'\in E}\sum_{i=2}^s\lambda_i(T',H'_i)_{p'}=
\\
&=&\mu\nu_p(T)^2-\sum_{p'\in E}\left(\alpha(E,T')_{p'}+\sum_{i=2}^s\lambda_i(T',H'_i)_{p'}\right)=\\
&=&\mu\nu_p(T)^2+\sum_{p'\in E}\mu I_{p'}({\mathcal D}',T').
\end{eqnarray*}
Assume now that $\pi$ is non-dicritical, hence $\alpha\ne 0$. We have
$$
-\alpha=-\sum_{p'\in E}\left(\mu(E,T')_{p'}+\sum_{i=2}^s\lambda_i(E,H'_i)_{p'}
\right)=\alpha\sum_{p'\in E}I_{p'}({\mathcal D}', E).
$$
This ends the proof.
\end{proof}

\begin{corollary} If $T$ is non singular at $p$, there is only one point $p'\in E\cap T'$ and we have that
$
I_{p'}({\mathcal D}',T')=I_p({\mathcal D},T)-1
$.
\end{corollary}
\subsection{Camacho-Sad indices}
Let us recall here the notion of generalized Camacho-Sad index introduced by A. Lins Neto in \cite{LNet}, in the spirit of the residue theory of Saito \cite{Sai}.  A good presentation of this results  may be found in Brunella \cite{Bru, Bru2} and \cite{Lem-S,Suw}.

\begin{definition}[\cite{LNet}]
Let $\mathcal F$ be a germ of foliation on $({\mathbb C}^2,0)$ generated by a holomorphic $1$-form $\omega$ without common factors in its coefficients. Consider an invariant branch $\Gamma$ of $\mathcal F$ given by an irreducible equation $f=0$. There is an expression
$$
g\omega=hdf+f\alpha,
$$
where $\alpha$ is a holomorphic $1$-form and $f$ does not divide $g$. The {\em Camacho-Sad index $\operatorname{CS}_0({\mathcal F},\Gamma)$ of $\mathcal F$ with respect to $\Gamma$} is defined by
$$
\operatorname{CS}_0({\mathcal F},\Gamma)=\frac{-1}{2\pi i}\int_{\gamma(f)}\frac{\alpha}{h},
$$
where $\gamma(f)$ is the homological class of the image of the standard loop $z\mapsto \exp(2\pi i)$ under a Puiseux parametrization of $\Gamma$.
\end{definition}
\begin{remark}
\label{rk:indicessingsimples}
 If the origin is a simple point that is not a saddle-node, we can take $\Gamma=(y=0)$ and $\omega=(\lambda+\cdots)ydx+(\mu+\cdots)xdy$, with $\lambda\mu \ne 0$. In this case we see that
 \begin{equation}\label{eq:indicereducido}
 \operatorname{CS}_0({\mathcal F},\Gamma)=-\lambda/\mu.
 \end{equation}
 \end{remark}
We are mainly interested in the behavior of the above index under non-dicritical blowing-ups. Let us summarize those results in the following proposition:
\begin{proposition}
Let $\mathcal F$ be a germ of foliation on $({\mathbb C}^2,0)$ and let
$$
\pi:( (M,E), {\mathcal F}')\rightarrow (({\mathbb C}^2,0),{\mathcal F}),\quad E=\pi^{-1}(0),
$$
be the blowing-up of the origin of ${\mathbb C}^2$. The following properties hold:
\begin{enumerate}
\item[a)] For any invariant branch $(\Gamma,0)$ we have that
$$
\operatorname{CS}_{p'}({\mathcal F}',\Gamma')= \operatorname{CS}_{p}({\mathcal F},\Gamma)-\nu_0(\Gamma)^2,
$$
where $p'$ is the only point in $E$ belonging to the strict transform $\Gamma'$ of $\Gamma$.
\item[b)] If $\pi$ is non-dicritical, then $\sum_{q\in E}\operatorname{CS}_q({\mathcal F}',E)=-1$.
\end{enumerate}
\end{proposition}
\begin{proof} See Brunella \cite{Bru,Bru2}.
\end{proof}
 As a consequence of the above results we obtain the following proposition:
\begin{proposition} Let $\mathcal L$ be a ${\mathcal D}$-logarithmic foliation on $({\mathbb C}^2,0)$, where $\mathcal D$ is a non-dicritical $\mathcal D$-divisor. Then
$$
\operatorname{CS}_0({\mathcal L},\Gamma)=I_0({\mathcal D},\Gamma),
$$
for any irreducible invariant branch $\Gamma$ of $\mathcal L$.
\end{proposition}
\begin{proof} The behavior of the indices is the same one after a sequence of blowing-ups  that desingularizes $\mathcal L$ and $\Gamma$ along $\Gamma$. When we have a simple point, the indices coincide by Equation \ref{eq:indicereducido}. This equality projects by the sequence of blowing-ups and we are done.
\end{proof}

\subsection{Existence of Divisorial Models in Dimension Two}
In this section we present the definitions and main properties of logarithmic models in dimension two, in terms of ${\mathbb C}$-divisors. The existence of logarithmic models for generalized curves in dimension two has been proved in \cite{Can-Co, Cor}, without an extensive use of $\mathbb C$-divisors.

The particularization to the ambient dimension two of the concept of generalized hypersurface is the one of {\em generalized curve}. To avoid possible confusion with other uses of this terminology in the literature, we note that in this paper a generalized curve is given by the following definition:

\begin{definition}
\label{def:curvageneralizada}
A foliation ${\mathcal F}$ on $({\mathbb C}^2,0)$ is a {\em generalized curve} if and only if it is non-dicritical and there are no saddle-nodes in a reduction of singularities of $\mathcal F$.
\end{definition}
\begin{remark}  If there are no saddle-nodes in a reduction of singularities, we find no saddle-nodes after any finite sequence of blowing-ups. In particular the definition is independent of the choice of a reduction of singularities (note that in dimension two we can speak of a {\em minimal reduction of singularities}).  For more details, see \cite{Can-C-D}.
\end{remark}

\begin{definition}
\label{def:modelodimensiondos}
Consider a generalized curve ${\mathcal F}$ and a  let $\mathcal D$ be a ${\mathbb C}$-divisor on a two-dimensional non-singular complex analytic variety $M$. We say that $\mathcal D$ is a {\em divisorial model for ${\mathcal F}$ at a point $p$ in $M$} if the following conditions hold:
\begin{enumerate}
   \item The support $\operatorname{Supp}({\mathcal D}_p)$  of the germ ${\mathcal D}_p$ of $\mathcal D$ at $p$ is the union of the germs at $p$ of the invariant branches of $\mathcal F$.
   \item The indices of ${\mathcal D}_p$ with respect to the irreducible branches of $\operatorname{Supp}({\mathcal D}_p)$ coincide with the Camacho-Sad indices of $\mathcal F$.
   \end{enumerate}
   We say that $\mathcal D$ is a {\em divisorial model for $\mathcal F$} if it fulfils the above conditions at every point $p\in M$. In the case of a germ $(M,K)$ we ask the property at each point of the germification set $K$.
\end{definition}
\begin{remark} Let us note that if $\mathcal D$ is a divisorial model for a generalized curve $\mathcal F$ on $(M,K)$ then we necessarily have that the ``germification set'' $K$ satisfies that $K\subset \operatorname{Supp}({\mathcal D})$. Indeed, if there is a point $p\in K\setminus \operatorname{Supp}({\mathcal D})$, we know that there is at least one invariant branch (this is a general fact that does not need of the hypothesis generalized curve, see \cite{Cam-S}, also \cite{Ort-R-V}) at $p$ that obviously is not contained in the support of the divisor.
\end{remark}
\begin{example}
\label{ex:integral primera}
The first example is a foliation $\mathcal F$ of $({\mathbb C}^2,0)$ with a holomorphic first integral. That is, we take a germ of function $f=f_1^{r_1}f_2^{r_2}\cdots f_s^{r_s}$ and the foliation given by $df/f$. The divisorial model is
 $$
\mathcal D=r_1\operatorname{Div}(f_1)+r_2\operatorname{Div}(f_2)+\cdots+
r_s\operatorname{Div}(f_s).
$$
The verification of this statement can be done, first in the normal crossings situation and second after Corollary \ref{cor:modlogsucblowing} and reduction of singularities.
\end{example}

Our objective in this subsection is to give a proof of the following result, in terms of $\mathbb C$-divisors:

\begin{theorem}
 \label{th:existenciayunicidadendimensiondos}
 Given a generalized curve ${\mathcal F}$ on $({\mathbb C}^2,0)$, there is a divisorial model $\mathcal D$ for $\mathcal F$. Moreover, if $\tilde {\mathcal D}$ is another divisorial model for $\mathcal F$, then $\tilde {\mathcal D}$ is projectively equivalent to $\mathcal D$; conversely, any ${\mathbb C}$-divisor projectively equivalent to $\mathcal D$ is also a divisorial model for $\mathcal F$.
\end{theorem}

Let us work in a matricial way. First of all, we recall a basic fact of linear algebra:
\begin{lemma}
 \label{lema:matrizsimetrica}
 Let $A=(\alpha_{ij})$ be a $s\times s$ symmetric matrix of rank $s-1$, having coefficients  in a field $k$. Assume that there is a vector $\lambda=(\lambda_1,\lambda_2,\ldots,\lambda_s)\in k^s$ such that
$\lambda A= 0$ and $\lambda_i\ne 0$ for any $i=1,2,\ldots,s$. Consider the diagonal minors
$$
\Delta_\ell=\det B_\ell; \quad B_\ell=(\alpha_{ij})_{i,j\in \{1,2,\ldots,s\}\setminus \{\ell\}}.
$$
Then $\Delta_\ell\ne 0$, for all $\ell=1,2,\ldots,s$.
\end{lemma}
\begin{proof} Up to reordering, we may assume that $\ell=1$. Let $F_i$ be the files and $C_i$ the columns of $A$. We know that
$$
F_1=(-1/\lambda_1)\sum_{i=2}^s\lambda_iF_i,\quad C_1=(-1/\lambda_1)\sum_{i=2}^s\lambda_iC_i.
$$
Let $A'$ be the matrix obtained by changing the first row of $A$ by $F_1$ plus the linear combination $(1/\lambda_1)\sum_{i=2}^s\lambda_iF_i$ and let $A''$ be obtained from $A'$ by changing the first column $C'_1$ of $A'$ by $C'_1$ plus the linear combination $(1/\lambda_1)\sum_{i=2}^s\lambda_iC'_i$, where the $C'_i$ are the columns of $A'$. We have that $\operatorname{rank}(A'')=s-1$ and
$$
A''=
\left(
\begin{array}{c|c}
0& 0\\
\hline
0&B_1
\end{array}
\right).
$$
We conclude that $\Delta_1\ne 0$.
\end{proof}

Denote by $H=\cup_{i=1}^sH_i$ the union of invariant branches of $\mathcal F$, where we fix an ordering $H_1,H_2,\ldots,H_s$. We define the $s\times s$ symmetric matrix $A_0({\mathcal F})=(\alpha_{ij})$ by
$$
\alpha_{ij}=
\left\{
\begin{array}{ccc}
\operatorname{CS}_0({\mathcal F},H_i)&\text{ if }& i=j,\\
(H_i,H_j)_0 &\text{ if }& i\ne j.
\end{array}
\right.
$$
Let us denote $B_0({\mathcal F})=(\alpha_{ij})_{2\leq i,j\leq s}$, that is, we have
\begin{equation}
\label{eq:matrizacero}
A_0({\mathcal F})=
\left(
\begin{array}{c|ccc}
\operatorname{CS}_0({\mathcal F},H_1)&(H_1,H_2)_0&\cdots&(H_1,H_s)_0\\
\hline
(H_2,H_1)_0&&&\\
\vdots&&B_0({\mathcal F})&\\
(H_s,H_1)_0&&&
\end{array}
\right).
\end{equation}
\begin{lemma}
\label{rk:matriciallogaritmicmodel} Let $\mathcal F$ be a generalized curve on $({\mathbb C}^2,0)$ and consider a $\mathbb C$-divisor of the form ${\mathcal D}=\sum_{i=1}^s\lambda_i H_i$. The following statements are equivalent:
\begin{enumerate}
\item The divisor $\mathcal D$ is a divisorial model for $\mathcal F$.
\item We have that $\lambda A_0({\mathcal F})=0$ and $\lambda_i\ne 0$ for any $i=1,2,\ldots,s$.
\end{enumerate}
\end{lemma}
\begin{proof} Assume that $\mathcal D$ is a divisorial model for $\mathcal F$. We have that  $\lambda_i\ne 0$ for any $i=1,2,\ldots,s$, since the support of $\mathcal D$ is the union $H=\cup_{i=1}^sH_i$ of the invariant curves of $\mathcal F$. Moreover, the indices of $\mathcal D$ coincide with the Camacho-Sad indices of $\mathcal F$. That is, for any $H_i$ we have that $\operatorname{CS}_0({\mathcal F},H_i)=I_0({\mathcal D},H_i)$; noting that
$$
I_0({\mathcal D},H_i)=\frac{-\sum_{j\ne i}\lambda_j(H_i,H_j)_0}{\lambda_i},
$$
we conclude that $\lambda A_0({\mathcal F})=0$.
Conversely, let us assume that $\lambda A_0({\mathcal F})=0$ and $\lambda_i\ne 0$ for any $i=1,2,\ldots,s$. Then, the support of ${\mathcal D}$ is equal to $H$. Moreover, the fact that $\lambda A_0({\mathcal F})=0$ implies that
$$
\operatorname{CS}_0({\mathcal D},H_i)=\frac{-\sum_{j\ne i}\lambda_j(H_i,H_j)_0}{\lambda_i}=
I_0({\mathcal D},H_i),\quad i=1,2,\ldots,s,
$$
and we are done.
\end{proof}
\begin{lemma}
\label{lema:matrices}
Let $\mathcal F$ be a generalized curve on $({\mathbb C}^2,0)$. Then, we have:
\begin{enumerate}
\item The rank $\operatorname{rk}(A_0({\mathcal F}))$ of $A_0({\mathcal F})$ is equal to $s-1$.
\item The determinant $\det B_0({\mathcal F})$ of $B_0({\mathcal F})$ is nonzero.
\item There is a vector $\lambda=(\lambda_1,\lambda_2,\ldots,\lambda_s)$ such that $\lambda_i\ne 0$ for any $i=1,2,\ldots,s$ and
 $\lambda A_0({\mathcal F})=0$.
\end{enumerate}
\end{lemma}
\begin{proof} We work by induction on the length of a reduction of singularities of $H$.
 If this length is zero, either $H$ is non singular or $H=H_1\cup H_2$ is the union of two transverse non singular branches. When $H$ is non singular, we have that $\mathcal F$ is non singular too, since it is a generalized curve; then we have $\operatorname{CS}_0({\mathcal F},H)=0$ and we are done.  If $H=H_1\cup H_2$ is the union of two transverse non singular branches, the origin is a simple point which is not a saddle-node. In this case we have
 $$
 \operatorname{CS}_0({\mathcal F},H_1)\operatorname{CS}_0({\mathcal F},H_2)=1.
 $$
We are done since
$$
A_0({\mathcal F})=
\left(
\begin{array}{cc}
\operatorname{CS}_0({\mathcal F},H_1)&1\\
1&1/\operatorname{CS}_0({\mathcal F},H_1)
\end{array}
\right).
$$
In order to prove the induction step, let us do the blowing-up centered at the origin
$$
\pi:(M,E)\rightarrow ({\mathbb C}^2,0);\quad E=\pi^{-1}(0).
$$
Denote by $p_1,p_2, \ldots,p_t$ the points of intersection between $E$ and the strict transform $H'$ of $H$. Up to a reordering in $H_2,H_3,\ldots,H_s$, we may assume that $p_j\in H'_i$ if and only if $i\in I_j$, where
\begin{equation}
\label{eq:indicesjota}
I_j=\{n_{j-1}+1,n_{j-1}+2,\ldots,n_j\};\quad n_0=0, n_t=s.
\end{equation}
Put $\nu_i=\nu_0(H_i)$, for $i=1,2,\ldots,s$. Note that $\nu_i=(E,H'_i)_{p_j}$ when $i\in I_j$. Denote by $\underline{\nu}$ the vector $\underline{\nu}=(\nu^{(1)},\nu^{(2)},\ldots,\nu^{(t)})$, where
$$
\nu^{(j)}=(\nu_{n_{j-1}+1}, \nu_{n_{j-1}+2},\ldots,\nu_{n_{j}}),\quad j=1,2,\ldots,t.
$$
The matrices $A_{p_j}({\mathcal F}')$ are given by
$$
A_{p_j}({\mathcal F}')=
\left(
\begin{array}{c|c}
\operatorname{CS}_{p_j}({\mathcal F}',E)&\nu^{(j)}\\
\hline
(\nu^{(j)})^t&B_{p_j}({\mathcal F}')
\end{array}
\right),
$$
where $B_{p_j}({\mathcal F}')$ is the matrix
$$
\left(
\begin{array}{cccc}
\operatorname{CS}_{p_j}({\mathcal F},H'_{n_{j-1}+1})&(H'_{n_{j-1}+1},H'_{n_{j-1}+2})_{p_j}&\cdots&(H'_{n_{j-1}+1},H'_{n_{j}})_{p_j}\\
(H'_{n_{j-1}+2},H'_{n_{j-1}+1})_{p_j}&\operatorname{CS}_{p_j}({\mathcal F},H'_{n_{j-1}+2})&\cdots&(H'_{n_{j-1}+2},H'_{n_{j}})_{p_j}\\
\vdots&\vdots&&\vdots\\
(H'_{n_{j}},H'_{n_{j-1}+1})_{p_j}&(H'_{n_{j}},H'_{n_{j-1}+2})_{p_j}&\cdots&\operatorname{CS}_{p_j}({\mathcal F},H'_{n_{j}})
\end{array}
\right).
$$
We can apply the induction hypothesis at the points $p_j$. Thus, for any $j=1,2,\ldots,t$ we have:
\begin{enumerate}
\item $\det B_{p_j}({\mathcal F}')\ne 0$.
\item There are vectors $
    \lambda^{(j)}=(\lambda_{n_{j-1}+1}, \lambda_{n_{j-1}+2}, \ldots, \lambda_{n_j} ),
    $ with nonzero entries such that
   \begin{equation}
   \label{eq:lambdaapj}
    (1,\lambda^{(j)})A_{p_j}({\mathcal F}')=0.
   \end{equation}
\end{enumerate}
Now, let us define the matrix $A'$ by
$$
A'=
\left(
\begin{array}{c|c|c|c|c}
-1&\nu^{(1)}&\nu^{(2)}&\cdots&\nu^{(t)}\\
\hline
({\nu^{(1)}})^t&B_{p_1}({\mathcal F}')&0&\cdots&0\\
\hline
({\nu^{(2)}})^t&0&B_{p_2}({\mathcal F}')&\cdots&0\\
\hline
\vdots&\vdots&\vdots&\cdots&\vdots\\
\hline
({\nu^{(t)}})^t&0&0&\cdots&B_{p_t}({\mathcal F}')
\end{array}
\right)
=
\left(
\begin{array}{c|c}
-1&\nu\\
\hline
(\nu)^t& B'
\end{array}
\right).
$$
Denote $\lambda=(\lambda_1,\lambda_2,\ldots,\lambda_s)=
(\lambda^{(1)},\lambda^{(2)},\ldots,\lambda^{(t)})$. Let us show that
\begin{equation}
\label{eq:lambdaaprima}
(1,\lambda)A'=0.
\end{equation}
Recall that $(1,\lambda^{(j)})A_{p_j}({\mathcal F}')=0$. Looking at the first column of $(1,\lambda^{(j)})A_{p_j}({\mathcal F}')$, we have that
$$
\operatorname{CS}_{p_j}({\mathcal F}',E)+\sum_{i=n_{j-1}+1}^{n_j}\lambda_i \nu_i=0, \quad j=1,2,\ldots,t.
$$
Noting that $\sum_{j=1}^t\operatorname{CS}_{p_j}({\mathcal F}',E)=-1$, we conclude that
\begin{equation}
\label{eq:lambdamultiplicidadesmenosuno}
-1+\sum_{i=1}^s\lambda_i \nu_i= \sum_{j=1}^t\left(\operatorname{CS}_{p_j}({\mathcal F}',E)+\sum_{i=n_{j-1}+1}^{n_j}\lambda_i \nu_i\right)=0.
\end{equation}
Hence, the first column of $(1,\lambda)A'$ is zero. The $(1+i)$-th column of $(1,\lambda)A'$, where $i=\ell+n_{j-1}\in I_j$ coincides with the $(1+\ell)$-th column of $ (1,\lambda^{(j)})A_{p_j}({\mathcal F}')$. This shows that $(1,\lambda)A'=0$.

Note that $\det B'\ne 0$, since $\det B_{p_j}({\mathcal F}')\ne 0$ for any $j=1,2,\ldots,t$. On the other hand, $(1,\lambda)A'=0$ implies that
\begin{equation}
\label{eq:lambdabprima}
\nu+\lambda B'=0.
\end{equation}
Moreover, recall the equalities
\begin{eqnarray*}
\operatorname{CS}_0({\mathcal F},H_i)&=& \operatorname{CS}_{p_j}({\mathcal F}',H'_i)+\nu_i^2,\quad i\in I_j\\
(H_i,H_\ell)_0&=&
\left\{
\begin{array}{ccc}
\nu_i\nu_\ell+(H'_i,H'_\ell)_{p_j}&\mbox{ if }& i,\ell\in I_j.
\\
\nu_i\nu_\ell&\mbox{ if }& i\in I_j,\ell\notin I_j.
\end{array}
\right.
\end{eqnarray*}
 Then, we have
$$
A_0({\mathcal F})= B'+\operatorname{Diag}(\nu_1,\nu_2,\ldots,\nu_s)N,
$$
where $N$ is the matrix that has all the rows equal to $\nu$. We conclude that
$$\operatorname{rank} A_0({\mathcal F})\geq s-1,$$ since $\operatorname{rank} B'=s$ and the rows of $A_0({\mathcal F})$ are obtained from the ones of $B'$ by adding vectors that are proportional to the single vector $\nu$.

Let us show that $\lambda A_0({\mathcal F})=0$, having in mind equations (\ref{eq:lambdabprima}) and  (\ref{eq:lambdamultiplicidadesmenosuno}). We have
\begin{eqnarray*}
\lambda A_0({\mathcal F})&=&\lambda B'+\lambda\operatorname{Diag}(\nu_1,\nu_2,\ldots,\nu_s)N=\\
&=&-\nu+(\lambda_1\nu_1,\lambda_2\nu_2,\ldots,\lambda_s\nu_s)N =\\
&=& -\nu+(\sum_{i=1}^s\lambda_i\nu_i)\nu=(-1+\sum_{i=1}^s\lambda_i\nu_i)\nu=0.
\end{eqnarray*}
Since $\lambda\ne 0$ and $\operatorname{rank}(A_0({\mathcal F}))\geq s-1$,  we conclude that $\operatorname{rank}(A_0({\mathcal F}))=s-1$, this shows property (1) of the statement. By construction, we have that $\lambda_i\ne 0$ for all $i=1,2,\ldots,s$, this shows property (3). Finally, property (2) follows from properties (1) and (3) in view of Lemma \ref{lema:matrizsimetrica}.
\end{proof}
\begin{remark}
Note that the above proof implies that $\operatorname{CS}_0({\mathcal F},H)=0$ when there is only one invariant branch $H$, even if $H$ is singular.
\end{remark}

Let us end the proof of Theorem \ref{th:existenciayunicidadendimensiondos}.
 Take the matrix $A_0({\mathcal F})$ as in Equation \ref{eq:matrizacero}. We have a vector $\lambda=(\lambda_1,\lambda_2,\ldots,\lambda_s)$ with only nonzero entries such that $\lambda A_0({\mathcal F})=0$. Define the ${\mathbb C}$-divisor ${\mathcal D}$ as
 $$
 {\mathcal D}= \lambda_1H_1+\lambda_2H_2+\cdots+\lambda_sH_s.
 $$
Applying Lemma \ref{rk:matriciallogaritmicmodel}, we see that $\mathcal D$ is a divisorial model for $\mathcal F$. If $\tilde {\mathcal D}=\sum_{i=1}^s\tilde\lambda H_i$ is projectively equivalent to ${\mathcal D}$, there is a constant $c\in {\mathbb C}^*$ such that $\tilde\lambda=c\lambda$. Hence we also have that $\tilde\lambda A_0({\mathcal F})=0$ and  $\tilde{\mathcal D}$ is also a divisorial model for $\mathcal F$.  Assume now that ${\mathcal D}'=\sum_{i=1}^s\lambda'_i H_i$ is another divisorial model for $\mathcal F$, by Lemma \ref{rk:matriciallogaritmicmodel}, we have that $\lambda' A_0({\mathcal F})=0$.
Since $A_0({\mathcal F})$ has rank $s-1$, there is a constant $c\in {\mathbb C}^*$ such that $\lambda'=c\lambda$ and thus the ${\mathbb C}$-divisor ${\mathcal D}'$ is projectively equivalent to ${\mathcal D}$. This ends the proof of Theorem \ref{th:existenciayunicidadendimensiondos}.
\subsection{Stability under Morphisms} In this Subsection, we characterize the two-dimensional divisorial models in terms of blowing-ups and also in terms of transverse maps. This properties are the essential facts we need for extending to higher dimension the concept of divisorial model.

\begin{proposition}
  \label{pro:modelostrasunaexplosion}
  Consider a generalized curve ${\mathcal F}$ on $({\mathbb C}^2,0)$ and a ${\mathbb C}$-divisor ${\mathcal D}$. Let $\pi:(M,\pi^{-1}(0))\rightarrow ({\mathbb C}^2,0)$ be the blowing-up of the origin and denote by ${\mathcal F}'$ the transform of $\mathcal F$ by $\pi$. The following statements are equivalent:
 \begin{enumerate}
 \item The $\mathbb C$-divisor $\mathcal D$ is a divisorial model for $\mathcal F$.
 \item The transform $\pi^*{\mathcal D}$ of $\mathcal D$ is a divisorial model for ${\mathcal F}'$.
 \end{enumerate}
 \end{proposition}
\begin{proof} Take notations as in the proof of Lemma \ref{lema:matrices}
and  put ${\mathcal D}=\sum_{i=1}^s\mu_i H_i$. Assume first that $\mathcal D$ is a divisorial model for $\mathcal F$. We have to prove that $\pi^*{\mathcal D}$ is a divisorial model for $\pi^*{\mathcal F}$ at any point $p\in \pi^{-1}(0)=E$. In view of the proof of Lemma \ref{lema:matrices}, we find vectors $(1,\lambda^{(j)})$ for any $j=1,2,\ldots,t$ such that
$$
(1,\lambda^{(j)})A_{p_j}({\mathcal F}')=0,
$$
with $\lambda^{(j)}=(\lambda_{n_{j-1}+1}, \lambda_{n_{j-1}+2},\ldots,\lambda_{n_{j}})$ and $\lambda_i\ne 0$ for $i=n_{j-1}+1, n_{j-1}+2,\ldots,n_{j}$. By Lemma~\ref{rk:matriciallogaritmicmodel}, this means that
$$
{\mathcal D}^{(j)}=E+ \sum_{i=n_{j-1}+1}^{n_j}\lambda_i H'_i
$$
is a divisorial model  for ${\mathcal F}'$ at the point $p_j$. Let us take
$$
{\mathcal D}'= E+ \sum_{i=1}^s\lambda_i H'_i.
$$
We have that ${\mathcal D}'$ is a divisorial model for ${\mathcal F}'$ at each of the points $p_j$, since the germ of ${\mathcal D}'$ at $p_j$ is equal to ${\mathcal D}^{(j)}$. Moreover, the $\mathbb C$-divisor ${\mathcal D}'$ is also a divisorial model for ${\mathcal F}'$ at any point $p\in E\setminus\{p_1,p_2,\ldots,p_t\}$, since the germ of ${\mathcal D}'$ at such points $p$ is just the ${\mathbb C}$-divisor $1\cdot E$. On the other hand, by  Equation \ref{eq:lambdamultiplicidadesmenosuno} we have
\begin{equation*}
\sum_{i=1}^s\lambda_i \nu_i=1.
\end{equation*}
This implies that ${\mathcal D}'=\pi^*{\mathcal D}_0$, where ${\mathcal D}_0=\sum_{i=1}^s\lambda_i H_i$. Since ${\mathcal D}_0$ is a divisorial model for ${\mathcal F}$ at the origin $0\in {\mathbb C}^2$, we have that ${\mathcal D}=c{\mathcal D}_0$ for a nonzero constant $c\in {\mathbb C}^*$. Hence $\pi^*{\mathcal D}=c{\mathcal D}'$ and it is a divisorial model for ${\mathcal F}'$.

Conversely, write ${\mathcal D}=\sum_{i=1}^s\mu_i H_i$ and assume that $\pi^*{\mathcal D}$ is a divisorial model for ${\mathcal F}'$. The exceptional divisor $E$ is invariant for ${\mathcal F}'$ and thus $\sum_{i=1}^s\mu_i\nu_i\ne 0$. Up to change ${\mathcal D}$ by a proportional $\mathbb C$-divisor, we can assume that $\sum_{i=1}^s\mu_i\nu_i=1$. This implies that $\pi^*{\mathcal D}={\mathcal D}'$ since they are both divisorial models for ${\mathcal F}'$ with  fixed coefficient equal to $1$ for the exceptional divisor $E$. This implies also that ${\mathcal D}={\mathcal D}_0=\sum_{i=1}^s\lambda_i H_i$. We are done.
\end{proof}
Next corollary is a direct consequence of the preceding proposition:
\begin{corollary}
\label{cor:modlogsucblowing}
  Consider a generalized curve ${\mathcal F}$ on $({\mathbb C}^2,0)$ and a ${\mathbb C}$-divisor ${\mathcal D}$. Let $\pi:(M,\pi^{-1}(0))\rightarrow ({\mathbb C}^2,0)$ be the composition of a finite sequence of blowing-ups. The following statements are equivalent:
 \begin{enumerate}
 \item $\mathcal D$ is a divisorial model for $\mathcal F$.
 \item $\pi^*{\mathcal D}$ is a divisorial model for the transform ${\mathcal F}'$ of $\mathcal F$ by $\pi$.
 \end{enumerate}
\end{corollary}
In particular, when $\pi$ is a reduction of singularities of ${\mathcal F}$, we have that  ${\mathcal D}$ is a divisorial model for $\mathcal F$ if and only if $\pi^*{\mathcal D}$ is a divisorial model for $\pi^*{\mathcal F}$ . This is the point of view taken in \cite{Cor} in the construction of divisorial models in dimension two.

The property stated in next Proposition \ref{prop:pullbacklogmod} is the starting point for defining divisorial models in higher ambient dimension.

\begin{proposition}
 \label{prop:pullbacklogmod}
 Let $\mathcal F$ be a generalized curve on $({\mathbb C}^2,0)$ and consider a $\mathbb C$-divisor ${\mathcal D}$ on $({\mathbb C}^2,0)$. The following statements are equivalent:
\begin{enumerate}
\item The $\mathbb C$-divisor $\mathcal D$ is a divisorial model for $\mathcal F$.
\item For any $\mathcal D$-transverse holomorphic map $\phi: ({\mathbb C}^2,0)\rightarrow ({\mathbb C}^2,0)$, we have that $\phi^*{\mathcal D}$ is a divisorial model for $\phi^*{\mathcal F}$.
\end{enumerate}
\end{proposition}
\begin{proof}
We see that (2) implies (1) by choosing the identity morphism. Let us show now that (1) implies (2). We have to test that $\phi^*{\mathcal D}$ is a divisorial model for $\phi^*{\mathcal G}$.

Note that the existence of the pull-back $\phi^*{\mathcal F}$ is guarantied by Proposition
\ref{prop:pullbackgeneralizedcurve}
and we also know that $\phi^*{\mathcal F}$ is a generalized curve. Let $\pi:(M,\pi^{-1}(0))\rightarrow ({\mathbb C}^2,0)$ be  a reduction of singularities of $\mathcal F$ by blowing-ups centered at points. In view of Proposition \ref{prop:appdos} in the Appendix I, there is a commutative diagram of morphisms
  $$
 \begin{array}{ccc}
 ({\mathbb C}^2,0)&\stackrel{\sigma}{\longleftarrow}&(N,\sigma^{-1}(0))\\
 \phi\downarrow\;\; &&\;\;\downarrow\psi\\
 ({\mathbb C}^2,0)&\stackrel{\pi}{\longleftarrow}&(M,\pi^{-1}(0))
 \end{array},
 $$
where $\sigma$ is the composition of a finite sequence of blowing-ups. Let us recall that $S=\operatorname{Supp}({\mathcal D})$ is the union of the invariant curves of $\mathcal F$.

Note that $\phi^*{\mathcal F}$ exists, since $\phi$ is $S$-transverse and we also have that the pullback $\phi^*{\mathcal D}$ exists by the same reason.
We have that $\phi\circ\sigma$ is $S$-transverse, since $\phi$ is $S$-transverse, by hypothesis, and $\sigma$ is a composition of blowing-ups. Hence $\pi\circ\psi$ is $S$-transverse, because we have $\pi\circ\psi=\phi\circ\sigma$. This implies that $\psi$ is $\pi^{-1}(S)$-transverse and then it is $\pi^*{\mathcal D}$-transverse, since
$$
\operatorname{Supp}(\pi^*{\mathcal D})\subset \pi^{-1}(S).
$$
In this situation, we have that
\begin{eqnarray}
\label{eq:transformados1}
 \sigma^*(\phi^*{\mathcal F})=(\phi\circ\sigma)^*{\mathcal F}&=& (\pi\circ\psi)^*{\mathcal F}
 = \psi^*(\pi^*{\mathcal F}),
 \\
 \label{eq:transformados2}
  \sigma^*(\phi^*{\mathcal D})=(\phi\circ\sigma)^*{\mathcal D}&=& (\pi\circ\psi)^*{\mathcal D}= \psi^*(\pi^*{\mathcal D}).
\end{eqnarray}
 We know that $\pi^*{\mathcal D}$ is a divisorial model for $\pi^*{\mathcal F}$ in view of Corollary \ref{cor:modlogsucblowing}. Also by Corollary \ref{cor:modlogsucblowing}, we have that $\phi^*{\mathcal D}$ is a divisorial model of $\phi^*{\mathcal F}$ if and only if $\sigma^*(\phi^*{\mathcal D})$ is a divisorial model of $\sigma^*(\phi^*{\mathcal F})$. In view of Equations \ref{eq:transformados1} and \ref{eq:transformados2}, it is enough to show that $\psi^*(\pi^*{\mathcal D})$ is a divisorial model for
$\psi^*(\pi^*{\mathcal F})$.

Recalling that $\pi^*{\mathcal F}$ is desingularized,  that $\pi^*{\mathcal D}$ is a divisorial model for $\pi^*{\mathcal F}$ and that the desired verification may be done in a local way, we have reduced the problem to the case when $\mathcal F$ is desingularized. More precisely:
\begin{quote}
In order to prove that (1) implies (2), it is enough to consider only the case when $\mathcal F$ is desingularized.
\end{quote}
Then, we assume that $\mathcal F$ is desingularized. If $\mathcal F$ is non-singular, we have that ${\mathcal F}=(dx=0)$ and (up to multiply by a constant) ${\mathcal D}=\operatorname{Div}(x)=1\cdot (x=0)$. In this case $\phi^*({\mathcal D})=\operatorname{Div}(x\circ \phi)$ and $\phi^*{\mathcal F}=(d(x\circ\phi=0))$, we are done by Example \ref{ex:integral primera}. Assume that $\mathcal F$ has a simple singular point at the origin, then it is generated by a logarithmic $1$-form
$$
\eta=(\lambda+f(x,y))\frac{dx}{x}+(\mu+g(x,y))\frac{dy}{y},\quad f(0,0)=g(0,0)=0,
$$
where $(\lambda,\mu)$ is non resonant, in the sense that $m\lambda+n\mu\ne 0$ for any pair of non-negative integer numbers $n,m$ such that $n+m\geq 1$. Moreover, the divisor $\mathcal D$ is given by
$${
\mathcal D}=\lambda \operatorname{Div}(x)+\mu \operatorname{Div}(y).
$$
Now, we apply Proposition \ref{prop:appuno} to desingularize the list of functions $(x\circ \phi, y\circ \phi)$, by means of a sequence of blowing-ups $\sigma':(N',{\sigma'}^{-1}(0))\rightarrow ({\mathbb C}^2,0)$. It is enough to verify that ${\sigma'}^*(\phi^*{\mathcal D})$ is a divisorial model for ${\sigma'}^*(\phi^*{\mathcal F})$ at the points $p\in {\sigma'}^{-1}(0)$. This reduces the problem to the case in which $\phi$ has the form
$$
x\circ\phi=Uu^av^b,\quad y\circ\phi=Vu^cv^d, \quad U(0,0)\ne 0\ne V(0,0),
$$
where $a+b\geq 1$ and $c+d\geq 1$ (note that none of these functions is identically zero, since $\phi$ is $S$-transverse). Put $(\lambda',\mu')=(\lambda a+\mu c, \lambda b+\mu d)$, we have
\begin{eqnarray*}
\phi^*\eta&=& \lambda'\frac{du}{u}+\mu'\frac{dv}{v}+\alpha,
\quad \mbox{ $\alpha$ holomorphic}, \\
\phi^*{\mathcal D}&=& \lambda'\operatorname{Div}(u)+\mu'\operatorname{Div}(v).
\end{eqnarray*}
Now, we see that $\phi^*{\mathcal D}$ is a divisorial model of ${\phi^*{\mathcal F}}$. Note that either $\lambda'\ne 0$ or $\mu'\ne 0$, since $a+b+c+d\geq 2$ and there are no resonances between $\lambda,\mu$. If $\lambda'\ne 0=\mu'$, we have a non singular foliation with $u=0$ the only invariant curve and $\phi^*{\mathcal D}=\lambda'\operatorname{Div}(u)$, we are done. If $\lambda'\ne 0\ne \mu'$ we have a simple singularity and $\phi^*{\mathcal D}=\lambda'\operatorname{Div}(u)+\mu'\operatorname{Div}(v)$ is a divisorial model, as we know by Remark \ref{rk:indicessingsimples}.
\end{proof}

\begin{corollary} Let $\mathcal F$ be a generalized curve on $({\mathbb C}^2,0)$ and consider a divisorial model ${\mathcal D}$ of ${\mathcal F}$. Then the $\mathbb C$-divisor $\mathcal D$ is non-dicritical.
\end{corollary}
\begin{proof} Assume that there is a $\mathcal D$-transverse map
$\phi:({\mathbb C}^2,0)\rightarrow ({\mathbb C}^2,0)$ such that $\phi^*{\mathcal D}=0$
 and $\phi(y=0)\subset \operatorname{Supp}({\mathcal D})$. By Proposition \ref{prop:pullbacklogmod}, we know that $\phi^*{\mathcal D}$ is a divisorial model of $\phi^*{\mathcal F}$ at the origin. But this is not possible, since we know that $\phi^*{\mathcal F}$ exists and hence the divisorial model at the origin  cannot be zero.
 \end{proof}

\section{Reduction of Singularities of Foliated Spaces}
Before considering reduction of singularities, let us precise what we mean by a {\em desingularized foliated space } in the case of generalized hypersurfaces. This concept is developed for any foliated space in \cite{Can}.
\begin{definition}
\label{def:foliatedspace}
A {\em foliated space} is just a data
$((M,K),E,{\mathcal F})$, where
\begin{enumerate}
\item The {\em ambient space} $(M,K)$ is a germ of non-singular complex analytic variety along a connected and compact analytic subset $K\subset M$.
\item  The {\em divisor} $E\subset M$ is a normal crossings divisor on $M$. More precisely, it is a germ along $E\cap K$.
\item  The {\em foliation} $\mathcal F$ is a germ of holomorphic foliation on $M$ along the germification set $K$.
\end{enumerate}
We say that a foliated space $((M,K),E,{\mathcal F})$ is of {\em generalized hypersurface type} if and only if $\mathcal F$ is a generalized hypersurface and all the irreducible components of $E$ are invariant for $\mathcal F$.
\end{definition}

We say that a foliated space $((M,K),E,{\mathcal F})$ is {\em desingularized} if it is {\em simple } at any point $p\in K$,  in the sense that we detail in Subsection \ref{definicionsimplepoint}. The property of being a simple point is an open property and hence it is also satisfied in an open neighborhood of $K$.

\subsection{Simple Points}
 \label{definicionsimplepoint}
 The definition of ``simple point'' in any dimension has been introduced in \cite{Can-C, Can}. Here we recall this concept
particularized to the case of foliated spaces of generalized hypersurface type.

Consider a foliated space $((M,K), E, {\mathcal F})$ of generalized hypersurface type.
Let us define when a point $p\in K$ is a simple point for the foliated space.

Denote by $\tau$ the {\em dimensional type} of ${\mathcal F}$ at $p$ (see \cite{Can,Can-C, Can-M-RV}). Roughly speaking, the dimensional type $\tau$ is the minimum number of local coordinates needed to describe ${\mathcal F}$ at $p$. Denote by $e$ the number of irreducible components of $E$ through $p$. The first request for $p$ to be simple is that $\tau-1\leq e \leq \tau$. In this way we have two categories of simple points:
\begin{enumerate}
\item[a)] {\em Simple corner points:} the simple points where $e=\tau$.
\item[b)]{\em Simple trace points:} the simple points where $e=\tau-1$.
\end{enumerate}
Assume that $e=\tau$. Then, there are coordinates $(x_1,x_2,\ldots,x_n)$ at $p$ such that $E=(\prod_{i=1}^\tau x_i=0)$ and  ${\mathcal F}$ is locally defined at $p$ by a meromorphic differential $1$-form $\eta$ written as
\begin{equation}
\label{simplecorners}
\eta=
\sum_{i=1}^\tau (\lambda_i+a_i(x_1,x_2,\ldots,x_\tau))\frac{dx_i}{x_i},\quad a_i\in {\mathcal O}_{M,p},
\end{equation}
where $a_i(0)=0$ for $i=1,2,\ldots,\tau$. We say that $p$ is a {\em simple corner} if the following {\em non resonance property} holds:
\begin{quote}
\label{quote:resonance}
 ``For any $\mathbf{0}\ne \mathbf{m}=(m_i)_{i=1}^\tau\in {\mathbb Z}_{\geq 0}^\tau$, we have that $\sum_{i=1}^\tau{m_i}\lambda_i\neq 0$.''
\end{quote}
Let us note that $\prod_{i=1}^\tau\lambda_i\ne0$.
 \begin{remark} It is known that the germs at $p$ of the irreducible components of $E$ are the only invariant germs of hypersurface for $\mathcal F$ at a simple corner $p$. One way of verifying this is as follows.  First of all, we can assume that $\tau=n$, because of the ``cylindric shape'' of the foliation over its projection on the first $\tau$ coordinates. Assume now that there is another invariant hypersurface. Then we should have an invariant curve $t\mapsto \gamma(t)$ as follows:
 $$
 \gamma(t)=(t^{m_1}U_1(t),t^{m_2}U_2(t),\ldots,t^{m_n}U_n(t)),\quad U_i(0)\ne 0, \; i=1,2,\ldots,n.
 $$
 Let $\eta$ be as in Equation \ref{simplecorners}. The fact that $\gamma^*\eta=0$ implies that $\sum_{i=1}^n m_i\lambda_i=0$ and this contradicts the property  of non-resonance.
 \end{remark}

Assume now that $e=\tau-1$. The point $p$ is a {\em simple trace point} if and only if there is a invariant germ of non-singular hypersurface $H_p$ at $p$, not contained in $E$ and having normal crossings with $E$, in such a way that the germ of $\mathcal F$ at $p$ is a simple corner with respect to the normal crossings divisor $E\cup H_p$.

\begin{remark}
 \label{rk:regular implicasimple}
 Given a foliated space $((M,K),E,{\mathcal F})$ of generalized hypersurface type, any point $p\in M\setminus\operatorname{Sing}({\mathcal F})$ is a simple point. In this case the dimensional type is $\tau=1$. If $e=0$, we have an ``improper'' trace point and the foliation is locally given by $dx=0$, where $x$ is a local coordinate, we write it as $dx/x=0$ and we see that $(M,\emptyset,\mathcal F)$ fulfils the definition of simple point for generalized hypersurfaces. If $e\geq 1$ we necessarily have that $e=1$, since all the components of $E$ are invariant and we have only one of them; we can choose an appropriate coordinate such that $x=0$ is the divisor $E$ and $\mathcal F$ is given by $dx/x=0$; hence it satisfies the definition of simple point.

The above property is true in the general case when all the components of $E$ are invariant. In presence of dicritical components, we have to assure the normal crossings property between $\mathcal F$ and the divisor. See \cite{Can}.
\end{remark}

In our current case of generalized hypersurface type foliated spaces, simple points may be described by means of the {\em logarithmic order} as follows. Consider a foliated space $((M,K), E, {\mathcal F})$  of generalized hypersurface type. Take a point $p\in K$. There are
 local coordinates $(x_1,x_2,\ldots,x_n)$ such that $E=(\prod_{i=1}^ex_i=0)$ and ${\mathcal F}$ is generated locally at $p$ by an integrable meromorphic  $1$-form
$$
\eta=\sum_{i=1}^e a_i(x)\frac{dx_i}{x_i}+\sum_{i=e+1}^na_i(x)dx_i,\quad a_i\in {\mathcal O}_{M,p},
$$
where the coefficients $a_i$ do not have a common factor, for $i=1,2,\ldots,n$. The {\em logarithmic order $\operatorname{LogOrd}_p(\eta,E)$} of $\eta,E$ at the origin is defined by
$$
\operatorname{LogOrd}_p(\eta,E)=\min\{\nu_0(a_i);\; i=1,2,\ldots,n\}.
$$
We also put $\operatorname{LogOrd}_p({\mathcal F},E)=\operatorname{LogOrd}_p(\eta,E)$, when $\eta$ generates ${\mathcal F}$ as above.
\begin{proposition}
 \label{prop:simplepointsandlogorder}
 Assume that $((M,K), E, {\mathcal F})$ is a foliated space of generalized hypersurface type. Take a point $p\in K$. The following statements are equivalent
\begin{enumerate}
\item The point $p$ is a simple point for $((M,K), E, {\mathcal F})$.
\item $\operatorname{LogOrd}_p({\mathcal F},E)=0$.
\end{enumerate}
\end{proposition}
\begin{proof} See also \cite{Can-M-RV, Moli}.
We provide a direct proof in Appendix II. \end{proof}

Thus, the locus of non simple points coincides with the {\em log-singular locus\/}:
$$\operatorname{LogSing}({\mathcal F}, E)=\{p\in M;\quad \operatorname{LogOrd}_p({\mathcal F},E)\geq 1 \}.$$

\subsection{Reduction of Singularities}
\label{reductionofsingularities}
Let us recall now what we mean by a reduction of singularities of a foliated space of generalized hypersurface type. The existence of reduction of singularities for germs of codimension one holomorphic foliations is know from the paper of Seidenberg \cite{Sei} in ambient dimension two; when the ambient dimension is three, it has been proven in \cite{Can}. In general ambient dimensions it is still an open problem, but there is reduction of singularities for foliated spaces of generalized hypersurface type \cite{Fer-M}.

Take a foliated space $((M,K), E,{\mathcal F})$ of generalized hypersurface type. A {\em reduction of singularities of $((M,K), E,{\mathcal F})$}
is a transformation of foliated spaces
\begin{equation}
\label{eq:reducciondesingularidades}
\pi: ((M',K'),E',{\mathcal F}')\rightarrow ((M,K),E,{\mathcal F})
\end{equation}
obtained by composition of a finite sequence of admissible blowing-ups of foliated spaces in such a way that $((M',K'),E',{\mathcal F}')$ is desingularized.

A non-singular and connected closed analytic subset $(Y,Y\cap K)\subset (M,K)$ is an {\em admissible center} for $((M,K), E,{\mathcal F})$ when it is invariant for $\mathcal F$  and it has normal crossings with $E$. In this situation, we can perform the admissible blowing-up with center $Y$:
$$
\pi_1:((M_1,K_1), E^1,{\mathcal F}_1)\rightarrow ((M,K), E,{\mathcal F}),\quad  K_1=\pi_1^{-1}(K),
$$
where ${\mathcal F}_1=\pi_1^*{\mathcal F}$ is the transform of ${\mathcal F}$ and $E^1=\pi_1^{-1}(E\cup Y)$. Such transformations may be composed. Then, a reduction of singularities $\pi$ as in Equation  \ref{eq:reducciondesingularidades} is a finite composition $$
\pi=\pi_1\circ\pi_2\circ\cdots\circ\pi_s,
$$
where each $\pi_i$ is an admissible blowing-up of foliated spaces, for $i=1,2,\ldots,s$. The number $s$ is called the {\em length } of $\pi$ and it will be important in order to perform inductive arguments.
\begin{remark}
We recall that a reduction of singularities of the (finite) union of invariant hypersurfaces induces a reduction of singularities of the foliated space, in the framework of generalized hypersurfaces, see \cite{Fer-M, Can-M-RV}. Then, we can assure the existence of reduction of singularities for a given foliated space $((M,K), E,{\mathcal F})$ of generalized hypersurface type.
\end{remark}

\subsection{Notations on a Reduction of Singularities}
\label{Notations on a Reduction of Singularities}

Let us introduce some useful notations concerning a given reduction of singularities $\pi$ as in Equation \ref{eq:reducciondesingularidades}.
The morphism $\pi$ is a finite composition  $\pi=\pi_1\circ\pi_2\circ\cdots\circ\pi_s$, where
$$
\pi_j:((M_j,K_{j}), E^j, {\mathcal F}_j)\rightarrow ((M_{j-1},K_{j-1}), E^{j-1},{\mathcal F}_{j-1}),\quad j=1,2,\ldots,s,
$$
is the admissible blowing-up with center $Y_{j-1}\subset M_{j-1}$. The initial and final foliated spaces are given by
\begin{eqnarray*}
((M_0,K_0),E^0,{\mathcal F}_0)&=&((M,K),E,{\mathcal F}),\\
((M_s,K_s),E^s,{\mathcal F}_s)&=&((M',K'),E',{\mathcal F}')
\end{eqnarray*}
The exceptional divisor of $\pi_j$ is $E^j_j=\pi_j^{-1}(Y_{j-1})$.
 Moreover, for any $j=1,2,\ldots, s$  we write the decomposition into irreducible components of $E^{j-1}$ and $E^j$ as
 $$
 E^{j-1}=\cup_{i\in I_0\cup \{1,2,\ldots,j-1\}}E^{j-1}_i,\quad
 E^{j}=\cup_{i\in I_0\cup \{1,2,\ldots,j\}}E^{j}_i,
 $$
where $E^{j}_i$ is the strict transform of $E^{j-1}_i$, for $i\in I_0\cup \{1,2,\ldots,j-1\}$.
If we denote $I=I_0\cup \{1,2,\ldots,s\}$ and $E'_i=E^s_i$ for $i\in I$, we have that $E'=\cup_{i\in I}E'_i$. In the same way, we can express the decomposition of $E$ into irreducible components as $E=\cup_{i\in I_0}E_i$, where $E_i=E^0_i$, for $i\in I_0$.

The inductive arguments on the length of $\pi$ are just based on the fact that after a first blowing-up, we have a reduction of singularities of smaller length. That is, when $s\geq 1$, we consider the decomposition $\pi=\pi_1\circ \sigma$, where $\sigma=\pi_2\circ\pi_3\circ\cdots\circ\pi_s$. Thus, we have
\begin{equation}
\begin{array}{ccc}
\pi_1: ((M_1,K_1), E^1,{\mathcal F}_1)&\rightarrow& ((M,K), E,{\mathcal F}),
\\
\sigma: ((M',K'), E',{\mathcal F}')&\rightarrow& ((M_1,K_1), E^1,{\mathcal F}_1).
\end{array}
\end{equation}
Note that $\sigma$ is a reduction of singularities of length $s-1$.
\begin{remark}
For the shake of simplicity,
we do not detail certain properties about the germs of spaces.
We will just use expressions as ``a point close enough to the germification set'' or ``by taking appropriate representatives''. In each case, we hope the reader to supply the exact meaning of these expressions.
\end{remark}

Take a point $p\in M$, close enough to the germification set $K$. Then $\pi$ induces a reduction of singularities over the ambient space $(M,p)$ that we denote
\begin{equation}
\label{eq:pisobrep}
\pi_p: ((M',K'_p), E',{\mathcal F}')\rightarrow ((M,p),E,{\mathcal F}),\quad  K'_p=\pi^{-1}(p).
\end{equation}
We can decompose it as $\pi_p=\pi_{1,p}\circ\sigma_p$, where
\begin{equation}
\label{eq.redsingsobrep}
\begin{array}{cccc}
\pi_{1,p}: ((M_1,K_{1,p}), E^1,{\mathcal F}_1)&\rightarrow& ((M,p), E,{\mathcal F}),&
K_{1,p}=\pi_1^{-1}(p),
\\
\sigma_p: ((M',K'_p), E',{\mathcal F}')&\rightarrow& ((M_1,K_{1,p}), E^1,{\mathcal F}_1).&
\end{array}
\end{equation}

Let us unify next the notations between the components of the exceptional divisors and the irreducible invariant hypersurfaces, not necessarily contained in them.
Denote by $S'\subset M'$ the union of invariant hypersurfaces of ${\mathcal F}'$ not contained in the  divisor $E'$. We know that $S'$ is a disjoint union of non singular hypersurfaces and $D'=E'\cup S'$ is also a normal crossings divisor on $M'$. Since the irreducible components of $E'$ are invariant, we have that $D'$ is the union of all invariant hypersurfaces of ${\mathcal F}'$. Let us denote
$$
S'=\cup_{b\in B}S'_b
$$
the decomposition into irreducible components of $S'$, where we choose the set of indices $B$ in such a way that $B\cap I=\emptyset$. Denote $D'_i=E'_i$ if $i\in I$ and $D'_b=S'_b$ if $b\in B$. We have that
$$
D'=\cup_{j\in I\cup B}D'_j,
$$
is the decomposition into irreducible components of $D'$. Moreover, let us denote by $D_j=\pi(D'_j)$, for $j\in I_0\cup B$. Then $D=\cup_{j\in I_0\cup B}D_j\subset M$ is the union of the irreducible invariant hypersurfaces of $\mathcal F$.

\subsection{Equidesingularization}

We recall here the concept of {\em equireduction point} for a foliated space $((M,K),E,{\mathcal F})$ of generalized hypersurface type. This idea has already been useful in \cite{Can-M} and \cite{Can-M-RV}.
\begin{remark}
In our applications we will consider points that can be outside of the germification set, but close enough to it. So, if we say ``take a point $p\in M$'' we understand that it is ``close enough to $K$''.
\end{remark}

Let us take a point $p\in M$. We say that $p$ is an {\em even point} for $((M,K),E,{\mathcal F})$ if either $p\not\in  \operatorname{Sing}({\mathcal F})$ (see Remark \ref{rk:regular implicasimple}) or $p\in \operatorname{Sing}({\mathcal F})$ and  the singular locus $\operatorname{Sing}({\mathcal F})$ satisfies the following properties, locally at $p$:
\begin{enumerate}
\item[a)] The singular locus  $\operatorname{Sing}({\mathcal F})$ has codimension two in $M$, it is non-singular and it has normal crossings with $E$.
\item[b)] The foliation ${\mathcal F}$ is equimultiple along $\operatorname{Sing}({\mathcal F})$. In particular, each irreducible component of $E$ through $p$ contains $\operatorname{Sing}({\mathcal F})$ and there are at most two of them.
\end{enumerate}
We say that $p$ is an {\em equireduction point}, or a point of {\em $2$-equireduction}, for the foliated space $(M,E,{\mathcal F})$ if it is an even point and this is stable under blowing-ups centered at the singular locus. More precisely, we say that an even point $p\in M$ is an {\em equireduction point} for $((M,K), E, {\mathcal F})$ if for any finite sequence of local blowing-ups over $p$
\begin{equation}
\label{eq:sucesiondeequirreduccion}
((M,p),E,{\mathcal F})\stackrel{\sigma_1}{\leftarrow}((M_1,p_1),E^1,{\mathcal F}_1)
\stackrel{\sigma_2}{\leftarrow}\cdots
\stackrel{\sigma_m}{\leftarrow}
((M_m,p_m),E^m,{\mathcal F}_m)
\end{equation}
such that the center of $\sigma_i$ is $\operatorname{Sing}({\mathcal F}_{i-1})$, for $i=1,2,\ldots,m$, we have the following properties:
 \begin{enumerate}
 \item The point $p_m$ is an even point for $((M_m,p_m),E^m,{\mathcal F}_m)$.
 \item If $p_m\in \operatorname{Sing}({\mathcal F}_m)$, the induced morphism
 $
\operatorname{Sing}({\mathcal F}_{m})\rightarrow
\operatorname{Sing}({\mathcal F}_{0})
$ is étale.
 \end{enumerate}

 Let us recall that a {\em local blowing-up} is the composition of a blowing-up
 $$\pi:(M',\pi^{-1}(p))\rightarrow (M,p)$$ with an immersion of germs $(M',p')\rightarrow (M',\pi^{-1}(p))$.

Next two results may be obtained by a direct adaptation of the statements proved in \cite{Can-M} to the case of generalized hypersurfaces:

\begin{proposition}
 \label{pro:codimensionnoequireduccion}
 Let $((M,K),E,{\mathcal F})$ be a foliated space of generalized hypersurface type. The set $\operatorname{Z}({\mathcal F}, E)$ of non-equireduction points is a closed analytic subset of $M$ of codimension at least three.
 \end{proposition}
 \begin{proposition}
  \label{pro:secciontransversalequirreduccion}
  Let $p\in M$ be a singular equireduction point for a foliated space $((M,K),E,{\mathcal F})$ of generalized hypersurface type. Any two dimensional section
  $$
  (\Delta,p)\subset (M,p)
  $$ transverse to $\operatorname{Sing}({\mathcal F})$ is a Mattei-Moussu transversal for ${\mathcal F}$ and it induces a foliated space $((\Delta,p), E\cap \Delta,{\mathcal F}\vert_{{\Delta}})$ that is a generalized curve.
\end{proposition}

Consider a singular equireduction point $p\in M$ for the generalized hypersurface type foliated space $((M,K),E,{\mathcal F})$. Let us perform the blowing-up with center at the whole singular locus $(\operatorname{Sing}(\mathcal F), p)\subset (M,p)$:
$$
\varsigma_1: ((M_1,\varsigma_1^{-1}(p)),E^1,{\mathcal F}_1))
\rightarrow
((M,p),E,{\mathcal F}).
$$
There are only finitely many points $\{p_{j}\}_{j=1}^{n_1}$ over $p$ in the singular locus $\operatorname{Sing}({\mathcal F}_1)$ and the morphism of germs
$$
 (\operatorname{Sing}({\mathcal F}_1),p_{j})
 \rightarrow (\operatorname{Sing}({\mathcal F}),p)
$$
is étale for all $j=1,2,\ldots,n_1$. Then, we can blow-up $((M_1,\varsigma_1^{-1}(p)),E^1,{\mathcal F}_1))$ with center $\operatorname{Sing}({\mathcal F}_1)$ to obtain a morphism
$$
\varsigma_2: ((M_2,\varsigma_2^{-1}(\varsigma_1^{-1}(p))),E^2,{\mathcal F}_2))
\rightarrow ((M_1,\varsigma_1^{-1}(p)),E^1,{\mathcal F}_1)).
$$
Note that the center of $\varsigma_2$ has exactly $n_1$ connected components, each one passing through a point $p_{j}$, for $j=1,2,\ldots,n_1$. Locally at each point $p_{j}$ we have an induced blowing-up with center $(\operatorname{Sing}({\mathcal F}_1),p_{j})$:
$$
\varsigma_{p_{j}}:((M_2,\varsigma_2^{-1}(p_{j})),E^2,{\mathcal F}_2))
\rightarrow ((M_1,p_{j}),E^1,{\mathcal F}_1)).
$$
Continuing indefinitely in this way, we obtain the {\em equireduction sequence}
\begin{equation}
\label{eq:sucesiondeequirreduccion2}
{\mathcal E}_{M,E,{\mathcal F}}^p:
((M,p),E,{\mathcal F})\stackrel{\varsigma_1}{\leftarrow}((M_1,\varsigma_1^{-1}(p)),E^1,{\mathcal F}_1)
\stackrel{\varsigma_2}{\leftarrow}
\cdots .
\end{equation}
The {\em infinitely near points of $p$} in $M_\ell$ are the points in
$$
(\varsigma_1\circ\varsigma_2\circ \cdots\circ \varsigma_{\ell})^{-1}(p)\cap
\operatorname{Sing}({\mathcal F}_\ell)
$$
that we can write as $p_{j_1j_2\cdots j_\ell}$, with the ``dendritic'' property that
$$
\varsigma_\ell(p_{j_1j_2\cdots j_{\ell-1} j_\ell})=p_{j_1j_2\cdots j_{\ell-1}}.
$$

Let us detail some consequences of Proposition \ref{pro:secciontransversalequirreduccion} relative to plane sections at an equireduction point $p\in M$. Take a two dimensional section $(\Delta,p)\subset (M,p)$ transverse to $\operatorname{Sing}({\mathcal F})$. First of all, we consider the following remark:
\begin{remark}
\label{rk:sectiondimdos}
 In view of \cite{Mat}, we know that $p$ is a simple point for  $((M,p),E,{\mathcal F})$ if and only if it is a simple point for the restriction
$((\Delta,p),E\cap \Delta,{\mathcal F}\vert_{\Delta})
$.
\end{remark}
If we consider a local holomorphic generator $\omega$ of ${\mathcal F}$ at $p$, without common factors in its coefficients,
 we know that $\eta=\omega\vert_{\Delta}$ is a local generator of ${\mathcal F}\vert_{\Delta}$ and moreover
$$
\nu_{\Sigma}(\omega)=\nu_p(\omega)=\nu_p(\eta), \quad \Sigma=(\operatorname{Sing}({\mathcal F}), p).
$$
This makes the blowing-ups $\varsigma_i$ in the equireduction sequence \ref{eq:sucesiondeequirreduccion2} to be ``compatible'' with the transversal section $\Delta$. More precisely, the equireduction sequence induces a sequence of blowing-ups $\bar\varsigma_i$ of the two-dimensional section $((\Delta,p),E\cap \Delta,{\mathcal F}\vert_{\Delta})$ with center at the points $p_{j_1j_2\cdots j_\ell}$, in such a way that the following diagram is commutative:
\begin{equation}
\label{eq:equrreducciondomtwo}
\begin{array}{ccccc}
((\Delta,p),E\cap \Delta,{\mathcal F}\vert_{\Delta})&
\stackrel{\bar\varsigma_1}{\longleftarrow}&
((\Delta_1,\varsigma_1^{-1}(p)),E^1\cap \Delta_1,{\mathcal F}_1\vert_{\Delta_1})&
\stackrel{\bar\varsigma_2}{\longleftarrow}&
\cdots
\\
\downarrow&&\downarrow&&
\\
((M,p),E,{\mathcal F})&\stackrel{\varsigma_1}{\longleftarrow}&
((M_1,\varsigma_1^{-1}(p)),E^1,{\mathcal F}_1)&
\stackrel{\varsigma_2}{\longleftarrow}&\cdots
\end{array}
\end{equation}
Looking at the diagram in Equation \ref{eq:equrreducciondomtwo}, we know that the sequence of the $\bar\varsigma_i$ desingularizes the two dimensional foliated space $((\Delta,p),E\cap \Delta,{\mathcal F}\vert_{\Delta})$, since we blow-up each time at the singular points; hence we apply the existence of reduction of singularities in dimension two, see \cite{Sei,Can-C-D}. As a consequence, at a finite step of the equireduction sequence given in Equation \ref{eq:sucesiondeequirreduccion2}, we reach a reduction of singularities of $((M,p), E,{\mathcal F})$.
\subsection{Relative Equireduction Points and Relative Transversality}
Let us introduce a version of the equireduction points relative to a fixed reduction of singularities
\begin{equation*}
\pi: ((M',K'),E',{\mathcal F}')\rightarrow ((M,K),E,{\mathcal F})
\end{equation*}
as in Equation \ref{eq:reducciondesingularidades}. We also introduce the locus of {\em $\pi$-good} points that will be essential in our proof of the existence of divisorial models.

Let us take the notations introduced in Subsection
\ref{Notations on a Reduction of Singularities}.
We define the locus $Z_\pi({\mathcal F},E)$ of the {\em points that are not points of $\pi$-equireduction} as being
$$
Z_\pi({\mathcal F},E)=Z({\mathcal F},E)\cup B_\pi({\mathcal F}, E),
$$
with $B_\pi({\mathcal F}, E)=\cup_{j\in J_\pi}(\pi_1\circ\pi_2\circ\cdots\circ\pi_j)(Y_j)
$,
where $J_\pi$ is the set of indices $j$ in $\{0,1,\ldots,s-1\}$ such that the center $Y_j$ of $\pi_{j+1}$ has codimension at least three.

 The complement $M\setminus Z_\pi({\mathcal F},E)$ is the set of {\em $\pi$-equireduction points}.
\begin{remark}
\label{rk:codimpiequireduccion}
  We have that $Z_\pi({\mathcal F},E)$ is a closed analytic subset of $M$ of codimension at least three.
  \end{remark}

  If we take a point $p\in M\setminus Z_\pi({\mathcal F},E)$, the morphism
  $$
  \pi_p:((M',\pi^{-1}(p)), E',{\mathcal F}')\rightarrow ((M,p), E,{\mathcal F})
  $$
is a ``part of the equireduction sequence'' in Equation \ref{eq:sucesiondeequirreduccion2}, in the sense that there is a finite step in the equireduction sequence that can be obtained from
$$
((M',\pi^{-1}(p)), E',{\mathcal F}')
$$
 by repeatedly blowing-up the singular locus.

 Recall now that we have the decomposition $\pi=\pi_1\circ \sigma$ and that $\pi_1^{-1}(Y)=E^1_1$ is the exceptional divisor of the first blowing-up $\pi_1$, where $Y=Y_0$.    Take a point $q\in E^1_1$. Let us put $p=\pi_1(q)$ and denote by $F_p=\pi_1^{-1}(p)$ the fiber of $p$ by $\pi_1$. Recall that $F_p$ is isomorphic to a complex projective space of dimension equal to $n-d-1$, where $n=\dim M$ and $d=\dim Y$. We say that $q$ is a point of {\em $\pi_1$-transversality} if $q\notin \operatorname{Sing}({\mathcal F}_1)$ or $q\in \operatorname{Sing}({\mathcal F}_1)$ and the germ $(\Sigma,q)$ of the singular locus   $\operatorname{Sing}({\mathcal F}_1)$ at $q$ is non singular, it is contained in $(E^1_1,q)$, it has codimension two in $(M_1,q)$ and, moreover, we have the transversality property with respect to the fiber $F_p$ given by
$T_q(F_p)\not\subset T_q\Sigma$,
where $T_q$ stands for the tangent space. Let us denote by $T^1_\pi \subset E^1_1$ the locus of points that are not of $\pi_1$-transversality.

 Note that we have the closed analytic set
$
Z_\sigma({\mathcal F}_1,E^1)
$
defining the locus of points in $M_1$ that are not of $\sigma$-equireduction. The codimension of $Z_\sigma({\mathcal F}_1,E^1)$ in $M_1$ is at least three.
We say that a point $p\in Y$ is a {\em $\pi$-bad point}  if and only if
$$
\dim
\left(
Z_\sigma({\mathcal F}_1,E^1)
\cup
T^1_\pi
\right)
\cap F_p
\geq n-d-2,\quad d=\dim Y.
$$
Denote by $B_\pi\subset Y$ the set of $\pi$-bad points.
\begin{lemma} The set $B_\pi$ is a closed analytic subset of $Y$ such that
$B_\pi\ne Y$.
\end{lemma}
\begin{proof}
We know  that
$\dim
\left(
Z_\sigma({\mathcal F}_1,E^1)
\cup
T^1_\pi
\right)
\cap F_p$ is the maximum
$$
\max\{
\dim
(Z_\sigma({\mathcal F}_1,E^1)
\cap F_p),
\dim(
T^1_\pi
\cap F_p)\}.
$$
Hence $B_\pi=B'\cup B''$ where
\begin{eqnarray*}
B'
&=&
\{p\in Y;\; \dim
(Z_\sigma({\mathcal F}_1,E^1)
\cap F_p)
\geq n-d-2\},
\\
B''&=&\{p\in Y;\; \dim
(T^1_\pi
\cap F_p)
\geq n-d-2 \}.
\end{eqnarray*}
Recall that the projection $E^1_1\rightarrow Y$ is a fibration with fiber ${\mathbb P}^{n-d-1}_{\mathbb C}$. Since the dimension of the fibers is upper semicontinuous, we see that both $B'$ and $B''$ are closed analytic subsets of $Y$. The codimension of $Z_\sigma({\mathcal F}_1,E^1)\cap E^1_1$ in $E^1_1$ is at least two, hence there is a closed subset $Z'\subset Y$, with $Z'\ne Y$ such that the codimension of $Z_\sigma({\mathcal F}_1,E^1)\cap F_p$ in $F_p$ is at least two, for any $p\in Y\setminus Z'$; in particular, we have that $B'\ne Y$.

Let us show now that $B''\ne Y$.
Decompose the analytic subset $\operatorname{Sing}({\mathcal F}_1)\cap E^1_1$ of $E^1_1$ as a union
$$
\operatorname{Sing}({\mathcal F}_1)\cap E^1_1= \Sigma_1\cup \Sigma_2\cup\cdots\Sigma_t\cup R_1,
$$
where the $\Sigma_i$ are the irreducible components of  $\operatorname{Sing}({\mathcal F}_1)$ that have codimension two and that are contained in $E^1_1$, for $i=1,2,\ldots,t$. The closed analytic set $R_1\subset E^1_1$ is the union of the intersection with $E^1_1$ of the other irreducible components of  $\operatorname{Sing}({\mathcal F}_1)$. Let us note that $R_1$ has codimension at least two in $E_1$ and that $R_1\subset T^1_\pi$. Moreover, we have that
$$
T^1_\pi=R_1\cup\bigcup_{i=1}^t(\Sigma_i\cap T^1_\pi).
$$
We have that $B''=B''_{1}\cup B''_{2}\cup\cdots\cup B''_{t}\cup B'''$, where
\begin{eqnarray*}
B''_{i}&=&
\{p\in Y;\; \dim
(\Sigma_i\cap T^1_\pi
\cap F_p)\geq n-d-2 \},
 \\
B'''&=&
\{p\in Y;\; \dim
(R_1\cap F_p)
\geq n-d-2 \}.
\end{eqnarray*}

Since the codimension of $R_1$ in $E^1_1$ is at least two, we have that $B'''\ne Y$.

Let us show that $B''_{i}\ne Y$, for $i=1,2,\ldots,t$. If $\dim \Sigma_i\cap T^1_\pi\leq n-3$, we are done. Thus, we assume that $\dim\Sigma_i\cap T^1_\pi= n-2$ and hence $\Sigma_i\subset T^1_\pi$. Take $U=\Sigma_i\setminus \Upsilon$, where $\Upsilon$ is the set of singularities of the closed analytic set $\operatorname{Sing}({\mathcal F}_1)$. The fact that $U\subset T^1_\pi$ means that for any point $q\in U$ we have that $T_q(F_{\pi_1(q)})\subset T_q\Sigma_i$. This property implies that $\Sigma_i$ is a union of fibers. Then, if $B''_i=Y$, we have that $\Sigma_i=E^1_1$, contradiction, since $\Sigma_i$ is a hypersurface of $E^1_1$.

We have a decomposition of $B_\pi$ into a finite union of strict closed analytic subsets of $Y$. Since $Y$ is irreducible, we conclude that $B_\pi\ne Y$.
\end{proof}
Next Corollary is an important step in our proof of the existence of divisorial models in higher dimension:
\begin{corollary}
\label{cor: equirrducciongenerica}
There is a strict closed analytic subset $Z\subset Y$, $Z\ne Y$ such that any point $p\in Y\setminus Z$ satisfies the following properties
\begin{enumerate}
  \item The center $Y$ is equimultiple for $\mathcal F$ at the point $p$.
  \item There is a Mattei-Moussu transversal $(\Delta,p)$ of ${\mathcal F}$ at $p$ such that the strict transform $\Delta^1$ of $\Delta$ by $\pi_{1,p}$ intersects $\operatorname{Sing}({\mathcal F}_1)$ transversely and only in points of $\sigma$-equireduction.
\end{enumerate}
\end{corollary}
\begin{proof} Let $C\subset Y$ be the points where $Y$ is not equimultiple for $\mathcal F$. We know that $C\ne Y$ is a strict closed analytic subset of $Y$. Now, take $Z=C\cup B_\pi$, that is also a strict closed analytic subset of $Y$. Let us show that $Z$ satisfies the statement. Take a point $p\in Y\setminus Z$. Since $p\notin C$, it is a point of equimultiplicity of $Y$ with respect to $\mathcal F$. Consider the fiber $F_p$ of $p$; we know that it is a projective space of dimension $n-d-1$. Since $p\notin B_\pi$, we have that
$$
\dim
\left(
Z_\sigma({\mathcal F}_1,E^1)
\cup
T^1_\pi
\right)
\cap F_p
\leq n-d-3,\quad d=\dim Y.
$$
This means that for a nonempty Zariski open set $U$ of the grassmanian of lines $\ell$ in the projective space $F_p$, we have the property that
$$
\ell\cap (Z_\sigma({\mathcal F}_1,E^1)
\cup
T^1_\pi)=\emptyset.
$$
We also know that for a nonempty Zariski open set $W$ of the grassmanian of lines $\ell$ in the projective space $F_p$, we have the property that $\ell$ meets transversely $\operatorname{Sing}({\mathcal F}_1)$.

By the generic properties of Mattei-Moussu transversals, we can choose $(\Delta,p)$ for $\mathcal F$ such that the line $\ell=F_p\cap \Delta_1$ lies in $U\cap W$. The desired property follows from the definition of the sets $Z_\sigma({\mathcal F}_1,E^1)$ and $T^1_\pi$.
\end{proof}
\section{Divisorial Models For Generalized Hypersurfaces}
\label{Logarithmic Models For Generalized Hypersurfaces}
 We introduce the definition of divisorial model as follows:
 \begin{definition}
Consider a generalized hypersurface $\mathcal F$ on  a non-singular complex analytic variety $M$ and a $\mathbb C$-divisor $\mathcal D$ on $M$. We say that $\mathcal D$
  is a {\em divisorial model for ${\mathcal F}$ at a point $p$ in $M$} if the following conditions hold:
   \begin{enumerate}
   \item The support $\operatorname{Supp}({\mathcal D}_p)$  of the germ ${\mathcal D}_p$ of $\mathcal D$ at $p$ is the union of the  germs at $p$ of invariant hypersurfaces of $\mathcal F$.
   \item For any ${\mathcal D}$-transverse map $\phi:({\mathbb C}^2,0)\rightarrow (M,p)$, the ${\mathbb C}$-divisor $\phi^*{\mathcal D}$ is a divisorial model for $\phi^*{\mathcal F}$.
   \end{enumerate}
   We say that $\mathcal D$ is a divisorial model for $\mathcal F$ if it fulfils the above conditions at every point $p$ in $M$.
  \end{definition}
   In view of Proposition  \ref{prop:pullbacklogmod}, this definition extends the one for dimension two, stated in Definition \ref{def:modelodimensiondos}.

 This Section is devoted to giving a proof of the following results:
 \begin{theorem}
 \label{teo:main}
 Every  generalized hypersurface on $({\mathbb C}^n,0)$ has a divisorial model.
 \end{theorem}

  \begin{proposition}
 \label{pro: uniquenessandnondicriticalness}
 Let $\mathcal D$ be a divisorial model for a generalized hypersurface $\mathcal F$ on $({\mathbb C}^n,0)$. Then  $\mathcal D$ is a non-dicritical $\mathbb C$-divisor. Moreover, any other $\mathbb C$-divisor is a divisorial model  for $\mathcal F$ if and only if it is projectively equivalent to
 $\mathcal D$.
 \end{proposition}

 In order to build a divisorial model, we take a reduction of singularities $\pi$ of the generalized hypersurface $\mathcal F$. There is a natural projective class of $\mathbb C$-divisors associated to the desingularized foliated space obtained from $\pi$. The divisorial model will be a $\mathbb C$-divisor that is transformed into this class by the reduction of singularities. For technical convenience, we systematically consider foliated spaces including a highlighted  normal crossings divisor, although the concept of divisorial model does not involve such a normal crossings divisor.

  Note that the uniqueness part in Proposition \ref{pro: uniquenessandnondicriticalness} comes from the corresponding fact in dimension two, just by taking a Mattei-Moussu transversal.

\subsection{$\mathbb C$-Divisors Associated to Desingularized Foliated Spaces}
\label{DivisorsAssociatedtoDesingularizedFoliateSpaces}
Let us consider a foliated space $((M',K'), E', {\mathcal F}')$ of generalized hypersurface type, where $K'$ is a connected and compact analytic subset of $M'$.  Let us  take  the following hypothesis:
\begin{itemize}
\item There is a logarithmic $1$-form $\eta'$ on $(M',K')$  fully associated to ${\mathcal F}'$.
\item The foliated space $((M',K'),E',{\mathcal F}')$ is desingularized.
\end{itemize}
 \begin{remark}
 \label{rk:existenciadeetaprima}
 Consider a foliated space $((M,K), E, {\mathcal F})$ of generalized hypersurface type,  where $K$ is a connected and compact analytic subset of $M$. Assume that there is a logarithmic $1$-form $\eta$ on $(M,K)$  fully associated to ${\mathcal F}$ and consider a reduction of singularities
 $$
 \pi:((M',K'), E',{\mathcal F}')\rightarrow ((M,K),E,{\mathcal F})
 $$
 as in Equation \ref{eq:reducciondesingularidades}.
 Then $\eta'=\pi^*{\eta}$ is a logarithmic $1$-form on $(M',K')$  fully associated to ${\mathcal F}'$, in view of Proposition
\ref{prop:pullbackoflogaritmicformsfullyassociated}. We note that the existence of such $\eta$, and hence $\eta'$, is assured when $K=\{0\}$ is a single point.
 \end{remark}

 In this subsection we build a $\mathbb C$-divisor ${\mathcal D}^{\eta'}$ on $(M',K')$,
 defining a divisorial model for ${\mathcal F}'$ on $(M',K')$, that is
 obtained from $\eta'$ in terms of residues.

Taking notations compatible with Subsection \ref{Notations on a Reduction of Singularities}, we denote by $S'\subset M'$ the union of invariant hypersurfaces of ${\mathcal F}'$ not contained in the  divisor $E'$. We know that $S'$ is a disjoint union of non singular hypersurfaces and $D'=E'\cup S'$ is also a normal crossings divisor on $M'$. Since the irreducible components of $E'$ are invariant, we have that $D'$ is the union of all invariant hypersurfaces of ${\mathcal F}'$. Accordingly with Subsection  \ref{Notations on a Reduction of Singularities}, we denote
$D'=\cup_{j\in I\cup B}D'_j$
the decomposition of $D'$ into irreducible components, where $E'=\cup_{i\in I}E'_i$ and $S'=\cup_{b\in B}S'_b$.

 Taking in account Saito's residue theory in \cite{Sai} and noting that $D'$ has the normal crossings property, the residue $\operatorname{Res}_p(\eta')$ of the germ of $\eta'$  at a point $p\in K'\subset M'$ is an element
$$
\operatorname{Res}_p(\eta')\in \oplus_{j\in I_p\cup B_p}{\mathcal O}_{D'_j,p},
$$
where $I_p=\{i\in I; p\in E'_i\}$ and $B_p=\{b\in B; p\in H'_b\}$. Note that $B_p$ is empty (corner points) or it has exactly one element (trace points) in view of the description of simple points in Subsection  \ref{definicionsimplepoint}. More generally, the residues induce global holomorphic functions
\begin{equation}
\label{eq:residuefunctions}
f_j:D'_j\rightarrow {\mathbb C},\quad j\in I\cup B.
\end{equation}
Let us note that each $D'_j$ is a germ along $D'_j\cap K'$.
Thus, the functions $f_j$ are constant along the connected components of the compact sets $D'_j\cap K'$.

More precisely, following Saito's Theory, we have local coordinates at $p$ that can be labelled as $(x_j; j\in I_p\cup B_p)\cup (y_s;s\in A)$ such that the germ $\eta'_p$ of $\eta'$ at $p$ can be written as
\begin{equation*}
\eta'_p=\sum_{j\in I_p\cup B_p}\tilde f_j\frac{dx_j}{x_j}+\alpha,\quad \tilde f_j\vert_{D'_j}=f_j,
\end{equation*}
where $\alpha$ is a germ of holomorphic $1$-form and moreover, the functions $\tilde f_j$ satisfy that
$\partial \tilde f_j/\partial x_j=0$, that is $\tilde f_j$
does not depend on the coordinate  $x_j$.
By evaluating $f_j$ at the point $p$, we get a local $\mathbb C$-divisor ${\mathcal D}_{\eta',p}$   defined by
\begin{equation}
\label{eq:descripciondedeetaprima}
{\mathcal D}_{\eta',p}=\sum_{j\in I_p\cup B_p} f_j(p)\operatorname{Div}_p(D'_j).
\end{equation}

\begin{proposition} There is a $\mathbb C$-divisor ${\mathcal D}^{\eta'}$ on $(M',K')$ such that the germ ${\mathcal D}^{\eta'}_p$ of ${\mathcal D}^{\eta'}$  at any $p\in K'$ satisfies that ${\mathcal D}^{\eta'}_p={\mathcal D}_{\eta',p}$.
\end{proposition}
\begin{proof}
  The residual functions $f_j$ of Equation
\ref{eq:residuefunctions} are constant along each connected component of $K'\cap D'_j$, for $j\in I\cup B$. We have only to remark that the compact set $K'\cap D'_j $ is connected for any $j\in I\cap B$; otherwise $D'_j$ would not be irreducible, since each connected component of $K'\cap D'_j$ determines a connected component of $D'_j$ as a germ of hypersuface.
\end{proof}

\begin{remark} The $\mathbb C$-divisor ${\mathcal D}^{\eta'}$ is a divisorial model for ${\mathcal F}'$. The proof of this statement is a consequence of the more general results in Subsection \ref{Existence of Logarithmic Models}, by considering a trivial reduction of singularities.
\end{remark}

\subsection{The $\mathbb C$-Divisor Induced by a Reduction of Singularities}
\label{The Divisor Induced by a Reduction of Singularities}

Let us consider a foliated space $((M,K),E,{\mathcal F})$ of generalized hypersurface type, where $K$ is a connected and compact analytic subset $K\subset M$. Assume that we have a logarithmic $1$-form $\eta$ on $(M,K)$ fully associated to $\mathcal F$. It is not evident how to define a ${\mathbb C}$-divisor associated to $\eta$ as in Subsection
\ref{DivisorsAssociatedtoDesingularizedFoliateSpaces}, unless we are in the situation of normal crossings outside a subset of codimension $\geq 3$, described in Saito \cite{Sai} and also in \cite{Cer-LN}. In this subsection  we do it, once we have fixed a reduction of singularities of $((M,K), E,{\mathcal F})$. More precisely, this subsection is devoted to the proof of the following result:

\begin{theorem}
 \label{teo:pilogarithmicmodel}
Consider a foliated space $((M,K),E,{\mathcal F})$ of generalized hypersurface type, where $K$ is a connected and compact analytic subset of $M$.  Let $\eta$ be a logarithmic differential $1$-form on $(M,K)$ fully associated to ${\mathcal F}$. Given a  reduction of singularities
$$
\pi:((M',K'), E',{\mathcal F}')\rightarrow ((M,K), E,{\mathcal F}),
$$
there is a unique $\mathbb C$-divisor ${\mathcal D}_\pi^{\eta}$ of $(M,K)$ such that $\pi^*({\mathcal D}_\pi^{\eta})={\mathcal D}^{\eta'}$, where $\eta'=\pi^*\eta$.  Moreover ${\mathcal D}_\pi^\eta$ is non dicritical.
\end{theorem}

\begin{remark} We know that Theorem  \ref{teo:pilogarithmicmodel} is true when $\dim M=2$.
Let us see this. Assume that $\dim M=2$ and take a logarithmic differential $1$-form $\eta$ fully associated to $\mathcal F$.
We have that $\eta'=\pi^*\eta$ is locally given at the singular points $p\in K'$ as
$$
\eta'=(\lambda+a(x,y))\frac{dx}{x}+ (\mu+b(x,y))\frac{dy}{y},
$$
where $xy=0$ is a local equation of $S' \cup E'$ at $p$.
The coefficient in ${\mathcal D}^{\eta'}$ of the irreducible component of $S'\cup E'$ of equation $x=0$ is equal to $\lambda$.
This tells us that ${\mathcal D}^{\eta'}$ is a divisorial model for ${\mathcal F}'$.
Consider now a divisorial model ${\mathcal D}$ for  ${\mathcal F}$, it exists in view of Theorem  \ref{th:existenciayunicidadendimensiondos}.
By Corollary \ref{cor:modlogsucblowing}, the pullback $\pi^*{\mathcal D}$ is a divisorial model for ${\mathcal F}'$. Recall that two divisorial models are projectively equivalent, by Theorem \ref{th:existenciayunicidadendimensiondos}. Noting that the supports are connected (each component of the support intersects the connected subset $K'=\pi^{-1}(0)$), there is a non-zero scalar $\mu\in {\mathbb C}$ such that $ {\mathcal D}^{\eta'}=\mu \pi^*{\mathcal D}$. If we take
${\mathcal D}_\pi^{\eta}=\mu{\mathcal D}$, we obtain that $\pi^*({\mathcal D}_\pi^{\eta})={\mathcal D}^{\eta'}$. The non-dicriticalness of ${\mathcal D}_\pi^{\eta}$ is also a consequence of the fact that it is a divisorial model for $\mathcal F$, in view of the statement of Theorem  \ref{th:existenciayunicidadendimensiondos}.
\end{remark}

From now on, let us take the general notations introduced in Subsection
\ref{Notations on a Reduction of Singularities}.

Let us recall that
$I\setminus I_0=\{1,2,\ldots,s\}$,
where $s$ is the length of $\pi$ as a composition of a sequence of blowing-ups. We assume that $s\geq 1$, otherwise we take ${\mathcal D}^{\eta}_\pi={\mathcal D}^{\eta}$ and we are done. Let us recall also the decomposition $\pi=\pi_1\circ\sigma$, where $\pi_1$ is the first blowing-up and $\sigma$ is a composition of $s-1$ blowing-ups.

Let us show that  ${\mathcal D}_\pi^{\eta}$ is necessarily unique. If we  write ${\mathcal D}^{\eta'}=\sum_{j\in I\cup B}\lambda_j D'_j$,
the condition  $\pi^*({\mathcal D}_\pi^\eta)={\mathcal D}^{\eta'}$ implies that
\begin{equation}
\label{eq:depieta}
{\mathcal D}_\pi^\eta=\sum_{j\in I_0\cup B}\lambda_j D_j.
\end{equation}
Then ${\mathcal D}_\pi^{\eta}$ is unique, if it exists. From now on, we fix ${\mathcal D}_\pi^\eta$ as being the one given by Equation \ref{eq:depieta}.

Let us see that ${\mathcal D}_\pi^\eta$ is a non-dicritical $\mathbb C$-divisor, under the assumption that $\pi^*({\mathcal D}^\eta_\pi)={\mathcal D}^{\eta'}$. We know that ${\mathcal D}^{\eta'}$ is non-dicritical, by applying Corollary \ref{cor:dicriticonormalcrossings}. Moreover, for any $i\in \{1,2,\ldots,s\}$, the $i$-blowing-up is non dicritical, since $\lambda_i\ne 0$. Hence $\pi$ is a composition of admissible non-dicritical blowing-ups for ${\mathcal D}^\eta_\pi$. Since ${\mathcal D}^{\eta'}$ is non-dicritical,
 by Corollary \ref{dicriticidadexplosionnodicritica} we conclude that ${\mathcal D}^\eta_\pi$ is non-dicritical.

Finally, we have to verify that $\pi^*({\mathcal D}_\pi^\eta)={\mathcal D}^{\eta'}$.

We proceed by induction on the length $s$ of the reduction of singularities $\pi$, using the fact that Theorem \ref{teo:pilogarithmicmodel} is true when $n=2$. Let us put $\eta_1=\pi_1^*\eta$.
We have that the logarithmic 1-form $\eta_1$ is fully associated to ${\mathcal F}_1$
and
$\eta'=\sigma^*\eta_1$.
Our induction hypothesis implies that the statement of Theorem \ref{teo:pilogarithmicmodel} is true for the morphism
$$
\sigma:((M',K'), E',{\mathcal F}')\rightarrow ((M_1,K_1), E^1,{\mathcal F}_1).
$$
This means that $\sigma^*({\mathcal D}^{\eta_1}_\sigma)={\mathcal D}^{\eta'}$. Then, in order to show that $\pi^*({\mathcal D}^{\eta}_\pi)={\mathcal D}^{\eta'}$, we have only to verify that $\pi_1^*({\mathcal D}^{\eta}_\pi)={\mathcal D}^{\eta_1}_\sigma$. This is equivalent to saying that the following equality holds:
\begin{equation}
\label{eq:lambdauno}
\lambda_1=\sum_{j\in J_Y}\lambda_j\nu_Y(D_j),
\end{equation}
where $J_Y=\{j\in I_0\cup B;\; Y\subset D_j\}$.
\begin{remark} Take a point $p\in Y$, not necessarily in $Y\cap K$, but close enough to $K$. Assume that $p$ is a point of equimultiplicity for $\mathcal F$; then, for any $D_j$, with  $j\in I_0\cup B$, we have that $\nu_p(D_j)=\nu_Y(D_j)$. In this case, we have the equality
\begin{equation}
\label{eq:indicesdey}
J_Y=\{j\in I_0\cup B;\quad p\in D_j\}.
\end{equation}
\end{remark}

Let us take a point $p\in Y\setminus Z$ and a Mattei-Moussu transversal $(\Delta,p)$ as stated in Corollary \ref{cor: equirrducciongenerica}.
We recall that $\pi$ induces a morphism $\pi_p$ as in Equation \ref{eq:pisobrep} that splits as $\pi_p=\pi_{1,p}\circ\sigma_p$ as in Equation \ref{eq.redsingsobrep}. Moreover, the
morphism $\pi$ induces a two-dimensional reduction of singularities of foliated spaces of generalized curve type that we denote as:
\begin{equation}
\label{eq:pibarrap}
\bar\pi_p: ((\Delta',E'(p)), E'(p), {\mathcal F}'\vert_{\Delta'})
\rightarrow ((\Delta,p), \emptyset, {\mathcal F}\vert_\Delta),\quad E'(p)=\Delta'\cap\pi^{-1}(p),
\end{equation}
where $\Delta'$ is the strict transform of $\Delta$ by $\pi$.
The reader may find similar situations in \cite{Can-M-RV}.
In view of the equireduction properties, there is an increasing sequence of integers
\begin{equation}
\label{eq:listaele}
\ell_1=1<\ell_2< \cdots<\ell_r\leq s,\quad T_Y=\{1,\ell_2,\ldots,\ell_s\},
\end{equation}
depending only on $Y$ with the following property: for any $j=1,2,\ldots,s$, the inverse image $K'_p=\pi^{-1}(p)$ intersects $D'_j$ if and only if $j\in T_Y$. We have that
$$
E'(p)=\cup_{t\in T_Y} (\Delta'\cap E'_t)
$$
is the decomposition into irreducible components of $E'(p)=\bar\pi^{-1}_p(p)$. We can decompose $\bar\pi_p$ as $\bar\pi_p=\bar\pi_{1,p}\circ \bar\sigma_p$ where
\begin{eqnarray*}
\bar\pi_{1,p}:((\Delta_1,E^1(p)), E^1(p),{\mathcal F}_1\vert_{\Delta_1})&\rightarrow &
((\Delta,p), \emptyset,{\mathcal F}\vert_{\Delta}), \quad E^1(p)=\Delta_1\cap \pi_1^{-1}(p),
\\
\bar\sigma_p: ((\Delta',E'(p)), E'(p),{\mathcal F}'\vert_{\Delta'})&\rightarrow&((\Delta_1,E^1(p)), E^1(p),{\mathcal F}_1\vert_{\Delta_1}),
\end{eqnarray*}
where $\Delta_1$ is the strict transform of $\Delta$ by $\pi_1$.

Since $(\Delta,p)$ is a Mattei-Moussu transversal, we have that
$
\bar\eta=\eta\vert_\Delta
$
is a logarithmic differentiable $1$-form fully associated to ${\mathcal F}\vert_{\Delta}$. Moreover, we know that $\Delta_1$ is also a Mattei-Moussu transversal for ${\mathcal F}_1$ in all the points of $E^1(p)$, then $\bar\eta_1=\eta_1\vert_{\Delta_1}$ is also a logarithmic $1$-form fully associated to ${\mathcal F}_1\vert_{\Delta_1}$. By the same reason, we see that $\bar\eta'=\eta'\vert_{\Delta'}$ is a logarithmic $1$-form fully associated to ${\mathcal F}'\vert_{\Delta'}$. On the other hand, an elementary functoriality assures that
\begin{equation}
\bar\eta'=\bar\pi_p^*(\bar\eta)=\bar\sigma_p^*(\bar\eta_1),
\quad \bar\eta_1=\bar\pi_{1,p}^*(\bar\eta).
\end{equation}

Note that $((M', K'_p ),E',{\mathcal F}')$ is desingularized. Then, we can follow the construction in Subsection
\ref{DivisorsAssociatedtoDesingularizedFoliateSpaces} to obtain a $\mathbb C$-divisor
$
{\mathcal D}^{\eta'}_p
$,
defined in $(M', K'_p )$  and  associated to the logarithmic $1$-form $\eta'$. To be precise, the $\mathbb C$-divisor
$
{\mathcal D}^{\eta'}_p
$
is associated to the germ of $\eta'$ along $K'_p$; this germ may be considered for $p$ close enough to $K$, by taking appropriate representatives.

We can write
 the $\mathbb C$-divisor ${\mathcal D}^{\eta'}_p$ on $(M', K'_p)$ and the ${\mathbb C}$-divisor  ${\mathcal D}^{\eta}_{\pi_p}$ on $(M,p)$ as
\begin{eqnarray}
\label{eq:enlafibra}
 {\mathcal D}^{\eta'}_p&=&\sum_{j\in J_Y}\lambda_j(p)\operatorname{Div}_{K'_p}(D'_j)
 +
 \sum_{j\in T_Y}\lambda_j(p)\operatorname{Div}_{K'_p}(D'_j),
 \\
 {\mathcal D}^{\eta}_{\pi_p}&=&\sum_{j\in J_Y}\lambda_j(p)\operatorname{Div}_{p}(D_j).
\end{eqnarray}
We denote by $\operatorname{Div}_{K'_p}(D'_j)$ the germ of $D'_j$ along $K'_p=\pi^{-1}(p)$.
 In the same way,
we denote by $\operatorname{Div}_{p}(D_j)$ the germ of $D_j$ at the point $p$.  Note that Equation \ref{eq:enlafibra} is written without null coefficients.
\begin{remark}
\label{rk:igualdaddecoefficientes}
 Recall that $K'_p=\pi^{-1}(p)$. If $p\in K$, we have that $K'_p\subset K'$ and the reader can verify that ${\mathcal D}^{\eta'}_p$ is just the germ of ${\mathcal D}^{\eta'}$ along $\pi^{-1}(p)$. In this case, we have that
\begin{equation}
\label{eq:igualdaddecoeficientes}
\lambda_j(p)=\lambda_j, \mbox{ for any } j\in J_Y\cup T_Y.
\end{equation}
When $p$ is not in the germification set $K$, we describe further the relationship between ${\mathcal D}^{\eta'}_p$ and the germ of ${\mathcal D}^{\eta'}$ along $\pi^{-1}(p)$.
\end{remark}
The following Lemma \ref{lema:thconequrreduccion} is our first step in the proof of Theorem \ref{teo:pilogarithmicmodel}:
\begin{lemma}
\label{lema:thconequrreduccion} Under the induction hypothesis, we have that
\begin{equation}
\label{eq:teoenpuntosdeequirreduccion}
\lambda_1(p)=\sum_{j\in J_Y}\lambda_j(p)\nu_Y(D_j)
\end{equation}
and thus $\pi_p^*({\mathcal D}^{\eta}_{\pi_p})={\mathcal D}^{\eta'}_{p}$.
\end{lemma}
 \begin{proof}  By induction hypothesis, we have that $\sigma_p^*({\mathcal D}^{\eta_1}_{\sigma_p})= {\mathcal D}^{\eta'}_{p}$. Reasoning as before, we have that  Equation \ref{eq:teoenpuntosdeequirreduccion} holds if and only if  $\pi_p^*({\mathcal D}^{\eta}_{\pi_p})={\mathcal D}^{\eta'}_{p}$.

 We know that
 the $\mathbb C$-divisor ${\mathcal D}^{\bar\eta'}$ on $(\Delta', K'_p)$,
 the $\mathbb C$-divisor ${\mathcal D}^{\bar\eta_1}_{\bar\sigma_p}$ on $(\Delta_1, K_{1,p})$
 and
 the $\mathbb C$-divisor ${\mathcal D}^{\bar\eta}_{\bar\pi_p}$ on $(\Delta, p)$
 are given by
$$
{\mathcal D}^{\bar\eta'}={\mathcal D}^{\eta'}_{p}\vert_{\Delta'},
\quad
{\mathcal D}^{\bar\eta_1}_{\bar\sigma_p}={\mathcal D}^{\eta_1}_{\sigma_p}\vert_{\Delta_1},
\quad
{\mathcal D}^{\bar\eta}_{\bar\pi_p}= {\mathcal D}^{\eta}_{\pi_p}\vert_{\Delta}
.
$$
Recalling that  Theorem \ref{teo:pilogarithmicmodel} is true for
two dimensional ambient spaces, we have
$$
\bar\pi_p^*({\mathcal D}^{\bar\eta}_{\bar\pi_p})={\mathcal D}^{\bar\eta'},
\quad
\bar\sigma_p^*({\mathcal D}^{\bar\eta_1}_{\bar\sigma_p})={\mathcal D}^{\bar\eta'},
\quad
\bar\pi_{1,p}^*({\mathcal D}^{\bar\eta}_{\bar\pi_p})={\mathcal D}^{\bar\eta_1}_{\bar\sigma_p}.
$$
The property $\bar\pi_{1,p}^*({\mathcal D}^{\bar\eta}_{\bar\pi_p})={\mathcal D}^{\bar\eta_1}_{\bar\sigma_p}$ gives us that Equation \ref{eq:teoenpuntosdeequirreduccion} holds, as follows.
The coefficient $\mu(p)$ of
$\operatorname{Div}(E^1_1\cap \Delta_1)$ in $\pi_{1,p}^*({\mathcal D}^{\bar\eta}_{\bar\pi_p})$ is given by
$$
\mu(p)=
\sum_{j\in J_Y}\lambda_j(p)\nu_p(D_j\cap \Delta).
$$
Recalling that $(\Delta,p)$ is a Mattei-Moussu transversal, we have that
$$
\nu_p(D_j\cap \Delta)=\nu_p(D_j),\quad \mbox{ for any } j\in J_Y.
$$
Since $p$ is a point of $Y$-equimultiplicity for the generalized hypersurface $\mathcal F$, we have that $\nu_p(D_j)=\nu_Y(D_j)$, for all $j\in J_Y$. We conclude that
$$
\mu(p)= \sum_{j\in J_Y}\lambda_j(p)\nu_Y(D_j).
$$
On the other hand, the coefficient of $\operatorname{Div}(E^1_1\cap \Delta_1)$ in ${\mathcal D}^{\bar\eta_1}_{\bar\sigma_p}$ is equal to $\lambda_1(p)$. Since  $\bar\pi_{1,p}^*({\mathcal D}^{\bar\eta}_{\bar\pi_p})={\mathcal D}^{\bar\eta_1}_{\bar\sigma_p}$, he have that $\lambda_1(p)=\mu(p)$ and we are done.
\end{proof}

In view of Remark \ref{rk:igualdaddecoefficientes} and  Equation \ref{eq:igualdaddecoeficientes}, if we can choose $p\in K$, the equality in Equation
\ref{eq:lambdauno} holds and we are done. This is the situation when $Y=\{p\}$ and more generally when $Y\subset K$.

We have to consider the case when we cannot chose $p\in K$. This means that $Y\cap K\subset Z$, where $Z\subset Y$ is the closed analytic subset presented in Corollary \ref{cor: equirrducciongenerica}.
We end with the following lemma:
\begin{lemma}
\label{lema.csindices} For any $j\in J_Y$ and $p\in Y\setminus Z$, we have that
$
\lambda_j/\lambda_1=\lambda_j(p)/\lambda_1(p)$.
\end{lemma}
\begin{proof}
Consider the reduction of singularities $\bar\pi_p$ as in Equation \ref{eq:pibarrap}. For any index $\ell\in J_Y\cup T_Y$, let us denote $D'_\ell(p)=D'_\ell\cap \Delta'$. The exceptional divisor of $\bar\pi_p$ is
$$
E'(p)=\bar\pi_p^{-1}(p)=\cup_ {t\in T_Y}D'_t(p)
$$
 and  $\bar\pi_p$ is a composition of $r$ blowing-ups, the morphisms corresponding to the indices $t\in T_Y$. By connectedness of the dual graph of $\bar\pi_p$, there is a finite sequence $(t_u)_{u=1}^{v+1}$ of elements of $T_Y\cup \{j\}$, with $t_1=1$, $t_{v+1}=j$ such that
 $$
 D'_{t_u}(p)\cap D'_{t_{u+1}}(p)\ne \emptyset, \quad u=1,2,\ldots,v.
 $$
Take a point $q_{u}\in D'_{t_u}(p)\cap D'_{t_{u+1}}(p)$, for $u=1,2,\ldots,v$.
By the local description of simple singularities in dimension two, we have that
$$
\operatorname{CS}_{q_u}({\mathcal F}'\vert_{\Delta'}, D'_{t_u}(p))
=-\lambda_{t_{u+1}}(p)/\lambda_{t_{u}}(p)
,
$$
see Remark \ref{rk:indicessingsimples}.
The Camacho-Sad index $\operatorname{CS}_{q_u}({\mathcal F}'\vert_{\Delta'}, D'_{t_u}(p))$ is the Camacho-Sad index of the transversal type of ${\mathcal F}'$ along
$D'_{t_u}\cap D'_{t_{u+1}}$. It can be read locally at the points in
$
(D'_{t_u}\cap D'_{t_{u+1}})\cap K'
$,
from the germ of the $1$-form $\eta'$ along $K'$. We deduce that
$$
\operatorname{CS}_{q_u}({\mathcal F}'\vert_{\Delta'}, D'_{t_u}(p))=-\lambda_{t_{u+1}}/\lambda_{t_{u}}.
$$
Hence $\lambda_{t_{u+1}}(p)/\lambda_{t_{u}}(p)=\lambda_{t_{u+1}}/\lambda_{t_{u}}$ for any $u=1,2,\ldots,v$. Making the product of these equalities, we conclude that $\lambda_{j}(p)/\lambda_{1}(p)=\lambda_j/\lambda_1$ and we are done.
\end{proof}
As a consequence of Lemmas  \ref{lema:thconequrreduccion} and \ref{lema.csindices},
we obtain that Equation \ref{eq:lambdauno} holds. This ends the proof of Theorem \ref{teo:pilogarithmicmodel}.

\subsection{Existence of Divisorial Models}
\label{Existence of Logarithmic Models}

 We apply first Theorem \ref{teo:pilogarithmicmodel} to a reduction of singularities
$$
\pi:((M',K'), E',{\mathcal F}')\rightarrow (({\mathbb C}^n,0), \emptyset,{\mathcal F}),\quad K'= \pi^{-1}(0),
$$
of the foliated space  $(({\mathbb C}^n,0), \emptyset,{\mathcal F})$. That is, we take an integrable logarithmic differential $1$-form $\eta$ fully associated to ${\mathcal F}$, we consider the $\mathbb C$-divisor ${\mathcal D}'={\mathcal D}^{\eta'}$ on $M'$, where $\eta'=\pi^*\eta$ and finally, we consider the ${\mathbb C}$-divisor ${\mathcal D}$ on $({\mathbb C}^n,0)$ defined by the property $\pi^*{\mathcal D}={\mathcal D}'$, whose existence has been shown in Theorem  \ref{teo:pilogarithmicmodel}. We are going to verify that ${\mathcal D}$ is a divisorial model for $\mathcal F$.
Note that the support of $\mathcal D$ is the union $D$ of the invariant hypersurfaces of $\mathcal F$.

Consider a ${\mathcal D}$-transverse holomorphic map $\phi:({\mathbb C}^2,0)\rightarrow ({\mathbb C}^n,0)$. By Proposition  \ref{prop:appdos} in the Appendix I, there is a commutative diagram of holomorphic maps
\begin{equation}
\begin{array}{ccc}
({\mathbb C}^2,0)&\stackrel{\sigma}{\longleftarrow}& (N',\sigma^{-1}(0))
\\
\downarrow \phi& & \downarrow \psi
\\
({\mathbb C}^n,0)&\stackrel{\pi}{\leftarrow}& (M',\pi^{-1}(0))
\end{array}
\end{equation}
such that $\sigma$ is the composition of  a finite sequence of blowing-ups.
\begin{lemma}
\label{lem:conmutatividad}
 In the above situation, we have that $\phi^*{\mathcal F}$ exists and
$\sigma^*({\phi^*{\mathcal F}})=\psi^*{\mathcal F}'$.
\end{lemma}
\begin{proof} Let $\omega$ be an integrable holomorphic $1$-form defining $\mathcal F$, without common factors in its coefficients. Assume that $\phi^*\omega\ne 0$, then we are done. Indeed, in this case $\phi^*{\mathcal F}$ is defined by $\phi^*\omega$, since $\sigma$ is a sequence of blowing-ups, we have that $\sigma^*(\phi^*\omega)\ne 0$ and $\sigma^*(\phi^*{\mathcal F})$ is defined by the nonzero $1$-form $\sigma^*(\phi^*\omega)$. Noting that
$$
\sigma^*(\phi^*\omega)=(\phi\circ\sigma)^*\omega=(\pi\circ\psi)^*\omega=\psi^*(\pi^*\omega),
$$
we conclude that $\phi^*{\mathcal F}$ exists and
$\sigma^*({\phi^*{\mathcal F}})=\psi^*{\mathcal F}'$.

Let us show that $\phi^*\omega\ne 0$. Assume by contradiction that $\phi^*\omega= 0$. Take an analytic germ of curve $(\Gamma, 0)$ such that $\phi(\Gamma)\not\subset D$. The existence of such a $(\Gamma, 0)$ is guarantied by the hypothesis that $\operatorname{Im}(\phi)\not\subset D$.
The assumption that
$\phi^*\omega=0$ implies that $(\phi(\Gamma),0)$ is an invariant germ of curve of $\mathcal F$. Taking a reduction of singularities of $D$, that induces a reduction of singularities of $\mathcal F$, we see that any germ of invariant curve must be contained in $D$. Contradiction.
\end{proof}
In view of Proposition \ref{pro:modelostrasunaexplosion}, in order to prove that $\phi^*{\mathcal D}$ is a divisorial model for $\phi^*{\mathcal F}$, it is enough to prove that
$
\sigma^*(\phi^*{\mathcal D})
$
is a divisorial model for $\sigma^*(\phi^*{\mathcal F})$. By Lemma \ref{lem:conmutatividad} above, we have that $\sigma^*({\phi^*{\mathcal F}})=\psi^*{\mathcal F}'$. Moreover, we also have
$$
\sigma^*(\phi^*{\mathcal D})=\psi^*{\mathcal D}'.
$$
 Thus,  we have to verify that $\psi^*{\mathcal D}'$ is a divisorial model for $\psi^*{\mathcal F}'$.

Let us work locally at a point $p\in\sigma^{-1}(0)$ and put $q=\psi(p)$. First of all, we take local coordinates $(x_1,x_2,\ldots,x_n)$ at $q$  such that the foliation ${\mathcal F}'$ is locally given at $q$ by an integrable meromorphic $1$-form
$$
\eta'=\sum_{i=1}^\tau (\lambda_i+f_i(x_1,x_2,\ldots,x_\tau))\frac{dx_i}{x_i},\quad f_i(0,0,\ldots,0)=0,
$$
  where  $\sum_{i=1}^\tau n_i\lambda_i\ne 0$ for any $\mathbf{n}\in {\mathbb Z}^\tau_{\geq 0}\setminus\{0\}$. Recall that then the total transform of $D$ is locally given at $q$ by the union of the hyperplanes $x_i=0$, with  $i=1,2,\ldots,\tau$. Moreover, we know that the $\mathbb C$-divisor ${\mathcal D}'$ is locally given at $q$ by
$$
{\mathcal D}'=\sum_{i=1}^\tau \lambda_i (x_i=0).
$$
We have to show that $\psi^*{{\mathcal D}'}$ is a divisorial model for $\psi^*{{\mathcal F}'}$ at $p$.  We apply now Proposition \ref{prop:appuno} in the Appendix I as follows. Take the list of functions
$$
{\mathcal L}'_p=\{\psi_{i,p}\}_{i=1}^\tau,
$$
 where $\psi_i=x_i\circ\psi$, for $i=1,2,\ldots,\tau$ and $\psi_{i,p}$ denotes the germ at $p$ of $\psi_i$.  There is a composition of blowing-ups centered at points
$$
\sigma': (N'',\sigma'^{-1}(p))\rightarrow (N',p)
$$
in such a way that the transformed list ${\mathcal L}''=\{f_i\}_{i=1}^\tau$  is desingularized, where $f_i=\psi_{i,p}\circ \sigma'$, see Appendix I. In particular, for any point $p'\in \sigma'^{-1}(p)$, there are local coordinates $u,v$ such
that
for any $i\in\{1,2,\ldots,\tau\}$ with $f_{i,p'}\ne 0$,  there is a unit $U_{i,p'}\in {\mathcal O}_{N'',p'}$ and  $(a_i,b_i)\in {\mathbb Z}^2_{\geq 0}$ with $f_{i,p'}=U_{i,p'} u^{a_i}v^{b_i}$; note also that we have $a_i+b_i\geq 1$, since $\sigma'(p')=p$.

 Now, in order to prove that $\psi^*{{\mathcal D}'}$ is a divisorial model for $\psi^*{{\mathcal F}'}$ at $p$, it is enough to prove that  ${\sigma'}^*(\psi^*{{\mathcal D}'})$ is a divisorial model for ${\sigma'}^*(\psi^*{{\mathcal F}'})$ at any point $p'$ of
 ${\sigma'}^{-1}(p)$.
 By the local expression of $\psi\circ\sigma'$ at $p'$ in appropriate local coordinates $u,v$, we conclude that ${\sigma'}^*\psi^*{{\mathcal F}'}$ is generated by

$$
{\sigma'}^*\psi^*\eta'=
\left(\sum_{i=1}^\tau a_i\lambda_i+g(u,v) \right)
\frac{du}{u}+\left(\sum_{i=1}^\tau b_i\lambda_i+h(u,v) \right)
\frac{dv}{v}+\alpha,
$$
where $\alpha$ is a holomorphic $1$-form.

Put $\mu_1=\sum_{i=1}^\tau a_i\lambda_i$ and $\mu_2=\sum_{i=1}^\tau b_i\lambda_i$. We know that not all the germs of functions $f_{i,p'}$ are identically zero, for $i=1,2,\ldots\tau$; this would imply that ${\sigma'}^*\psi^* \eta'=0$ and this is not possible since we know that $\psi^*\eta'\ne 0$. Then, some of the $a_i, b_i$ are nonzero and by the non resonance properties either $\mu_1$ or $\mu_2$ are nonzero. Say that $\mu_1\ne 0$, since we are dealing with generalized hypersurfaces, there are no saddle-nodes, and then $\mu_1\mu_2\ne 0$ or we have a non-singular foliation, in the sense that $\mu_2+h(u,v)$ is identically zero. Now, we have that
$$
{\sigma'}^*\psi^*{\mathcal D}'=\mu_1(u=0)+\mu_2(v=0)
$$
 locally at $p'$. Hence ${\sigma'}^*\psi^*{\mathcal D}'$ is a divisorial model for ${\sigma'}^*(\psi^*{\mathcal F}')$ at $p'$. This ends the proof of Theorem  \ref{teo:main}.

\section{Logarithmic Models}
\label{Logarithmic Models}
Let $\mathcal F$ be a generalized hypersurface over $(M,K)$ where $M$ is a germ of non-singular complex analytic variety over a connected and compact analytic subset $K\subset M$. Let us assume that $\mathcal D$ is a divisorial model for ${\mathcal F}$. By definition, any $\mathcal D$-logarithmic foliation $\mathcal L$ on $(M,K)$ is a {\em logarithmic model} for $\mathcal F$.

In the case of $K=\{0\}$ and hence $(M,K)=({\mathbb C}^n,0)$, the existence of logarithmic models is assured. Indeed, by Theorem \ref{teo:main} there is a divisorial model $\mathcal D$ for $\mathcal F$, that we can write
$$
{\mathcal D}=\sum_{i=1}^s\lambda_iS_i.
$$
Choosing a reduced local equation $f_i=0$ for each $S_i$, the closed logarithmic $1$-form $\eta=\sum_{i=1}^sdf_i/f_i$ generates a logarithmic model $\mathcal L$. This gives sense to the main theorem stated in the Introduction.

Let us state certain properties of logarithmic models that are directly deduced from the results presented in this work:
\begin{enumerate}
\item If $(M,K)=({\mathbb C}^2,0)$, a logarithmic foliation $\mathcal L$ is a logarithmic model for a generalized curve $\mathcal F$ if and only if ${\mathcal L}$ and $\mathcal F$ have the same Camacho-Sad indices with respect to the invariant branches.
\item Assume that $\pi: ((M',K'),E',{\mathcal F}')\rightarrow ((M,K),E,{\mathcal F})$ is the composition of a finite sequence of admissible blowing-ups, where $((M,K),E,{\mathcal F})$ is a foliated space of generalized hypersurface type. If $\mathcal L$ is a logarithmic model for $\mathcal F$, then $\pi^*{\mathcal L}$ is a logarithmic model for ${\mathcal F}'$.
\item  Let $\mathcal F$ be a generalized hypersurface on $({\mathbb C}^n,0)$ and denote by $S$ the union of its invariant hypersurfaces. Consider a logarithmic foliation $\mathcal L$ on $({\mathbb C}^n,0)$. The following statements are equivalent:
    \begin{enumerate}
    \item $\mathcal L$ is a logarithmic model for $\mathcal F$.
    \item For any $S$-transverse map $\phi:({\mathbb C}^2,0)\rightarrow ({\mathbb C}^n,0)$ we have that $\phi^*{\mathcal L}$ is a logarithmic model for $\phi^*{\mathcal F}$.
    \end{enumerate}
\end{enumerate}
A question for the future is to develop a similar theory concerning the dicritical case. In dimension two, some results are known \cite{Can-Co}.

\section{Appendix I}
We recover here results in Proposition \ref{prop:appuno} and Proposition \ref{prop:appdos}  concerning the reduction of singularities of lists of functions in dimension two and the lifting of morphisms by a sequence of blowing-ups. These results are well known. We just state the first one and we prove the second one as a consequence of it.
Let us note that there are stronger results on monomialization of morphisms by D. Cutkosky \cite{Cut} and Akbulut and King (Chapter 7 of \cite{Akb-K}, according to \cite{Cut}); we do not need the use of such strong versions in this work.

Let $N$ be a two dimensional complex analytic variety and consider a finite list
${\mathcal L}=\{f_i\}_{i=1}^t$, where $f_i;N\rightarrow {\mathbb C}$ is a holomorphic function for $i=1,2,\ldots,t$. Given a point $p\in N$, denote by ${\mathcal L}_p=\{f_{i,p}\}_{i=1}^t$ the list of germs at $p$ of the functions $f_i$. We say that ${\mathcal L}$ is {\em desingularized } at $p\in N$, or equivalently that {\em the list ${\mathcal L}_p$ is desingularized,\/} if and only if there are local coordinates $(u,v)$ at $p$ such that the following two properties hold:
\begin{enumerate}
 \item For any $i\in\{1,2,\ldots,t\}$ with $f_{i,p}\ne 0$,  there is a unit $U_{i,p}\in {\mathcal O}_{N,p}$ and  $(a_i,b_i)\in {\mathbb Z}^2_{\geq 0}$ such that $f_{i,p}=U_{i,p} u^{a_i}v^{b_i}$.
 \item Given  $i,j\in\{1,2,\ldots,t\}$, if $f_{i,p}$ does not divide $f_{j,p}$, then $f_{j,p}$ divides $f_{i,p}$.
 \end{enumerate}
We say that {\em ${\mathcal L}$ is desingularized} when it is desingularized at any point $p\in N$. This is an open property, in the sense that the points where ${\mathcal L}$ is desingularized are the points of an open subset of $N$. Given a morphism $\sigma:N'\rightarrow N$, the transform $\sigma^*{\mathcal L}$ of the $\mathcal L$ is the list in $N'$ defined by
$$
\sigma^*{\mathcal L}=\{f_i\circ\sigma\}_{i=1}^t
$$
Next, we state a well known result of desingularization of lists of functions:
\begin{proposition}
\label{prop:appuno}
Let ${\mathcal L}=\{f_i\}_{i=1}^t$ be a list of functions on a  non singular two-dimensional holomorphic variety $(N,C)$, that is a germ along a compact set $C\subset N$. There is a morphism
$$
\sigma:(N',\sigma^{-1}(C))\rightarrow (N,C),
$$
that is the composition of a finite sequence of blowing-ups centered at points, in such a way that the transformed list $\sigma^*{\mathcal L}=\{f_i\circ \sigma\}_{i=1}^t$ is desingularized.
\end{proposition}
\begin{proof} This result is an easy consequence of the classical results of desingularization for plane curves. The reader may look at \cite{Lip}.
\end{proof}
\begin{remark}
The above proposition \ref{prop:appuno} is true without restriction on the dimension of $N$ (with similar definition of what is a desingularized list). This is a consequence of Hironaka's reduction of singularities in \cite{Hir}.
\end{remark}

 \begin{proposition}
  \label{prop:appdos} Let $\phi:(N,C)\rightarrow (M,K)$ be a holomorphic map between connected germs of non-singular analytic varieties along compact sets, where $\dim N=2$. Consider a morphism $\pi:(M',\pi^{-1}(K))\rightarrow (M,K) $ that is the composition of a finite sequence of blowing-ups with non singular centers. Let us assume that the image of $\phi$ is not contained in the projection by $\pi$ of the centers of blowing-up. Then, there is a morphism
 $$
 \sigma: (N',\sigma^{-1}(C))\rightarrow (N,C)
 $$
 that is the composition of a finite sequence of blowing-ups centered at points, in such a way that there is a unique morphism
 $
\psi: (N',\sigma^{-1}(C))\rightarrow (M',\pi^{-1}(K))
 $
 such that $\phi\circ\sigma=\pi\circ\psi$.
 \end{proposition}
 \begin{proof} Let us show first that the result is true in the special case that
 $$
 (N,C)=({\mathbb C}^2,0), \quad  (M,K)=({\mathbb C}^n,0)
 $$
 and $\pi$ is the single blowing-up with center $Y=(x_1=x_2=\cdots=x_t=0)$. Consider the
 list of functions ${\mathcal L}=\{\phi_i\}_{i=1}^t$, where  $\phi_i=x_i\circ \phi$. Take a desingularization
 $$
 \sigma: (N',\sigma^{-1}(0))\rightarrow ({\mathbb C}^2,0),
 $$
  of $\mathcal L$ as stated in Proposition \ref{prop:appuno}, where ${\mathcal L}'=\{\varphi_i\}_{i=1}^t$ is the transformed list, recalling that $\varphi_i=\phi_i\circ\sigma$.

  Let us represent $\phi$ by a morphism
 $
 \phi_U:U\rightarrow V
 $
 where $U\subset{\mathbb C}^2$ is a connected open neighborhood of the origin $0\in {\mathbb C}^2$ and $V={\mathbb D}_\epsilon^n\subset  {\mathbb C}^n$ is a poly-cylinder around the origin in such a way that the center of $\pi$ is represented by $$
 Y=(x_1=x_2=\cdots=x_t=0)\subset V.
 $$
 We also consider the morphism $\sigma_U:U'\rightarrow U$ obtained by the same blowing-ups indicated by $\sigma$ and we denote by $\pi_V:V'\rightarrow V$ the blowing-up with center $Y$.

 Let us put $\phi_{i,U}=x_i\circ\phi_U$ and $\varphi_{i,U'}= \phi_{i,U}\circ \sigma_U$. Since the property of being desingularized is open, by taking $U$ small enough, we can assume that the list ${\mathcal L}_{U'}'=\{\varphi_{i,U'}\}_{i=1}^t$ is desingularized at any point of $U'$.

 In this situation, let us show that there is a unique holomorphic map
 $$
 \psi_{U'}:U'\rightarrow V'
 $$ such that $\phi_U\circ \sigma_U=\pi_V\circ\psi_{U'}$. More precisely, we are going to prove that given any nonempty open subset $W\subset U'$, there is a unique holomorphic map
 $$
 \psi_{W}:W\rightarrow V'
 $$
 such that $\phi_U\circ ({\sigma_U}\vert_W)=\pi_V\circ\psi_W$.

 Let us recall how is constructed the blowing-up $\pi_V$. Take the projective space ${\mathbb P}^{t-1}_{\mathbb C}$ and consider the closed subset $Z\subset {\mathbb P}^{t-1}_{\mathbb C}\times {\mathbb D}_\epsilon^t$ given by
 $$
 Z=\{([a_1,a_2,\ldots,a_t], (b_1,b_2,\ldots, b_t)); \quad a_ib_j=a_jb_i, \; 1\leq i,j\leq t\}.
 $$
 The blowing-up $\tilde\pi$ of the origin of ${\mathbb D}^t_{\epsilon}$ is the second projection
 $
 \tilde\pi: Z\rightarrow {\mathbb D}^t_{\epsilon}
 $. The blowing-up of $V={\mathbb D}^t_{\epsilon}\times {\mathbb D}^{n-t}_{\epsilon}$ with center $Y$ is the product
 $$
 \pi=\tilde\pi\times \operatorname{id}_{{\mathbb D}^{n-t}_{\epsilon}}: Z\times {\mathbb D}^{n-t}_{\epsilon}\rightarrow {\mathbb D}^{t}_{\epsilon}\times {\mathbb D}^{n-t}_{\epsilon}=V.
 $$
 Now, let us show the existence and uniqueness of $\psi_{W}$.

 Let us consider the open subset $W_0\subset W$ defined by
 $
 W_0=W\setminus \sigma_{U}^{-1}(\phi_U^{-1}(Y))
 $.
 By hypothesis, we know that $W_0$ is a dense open subset of $W$. The uniqueness of $\psi_W$ is then implied by the uniqueness of $\psi_{W_0}$. Take $p\in W_0$ and consider the vectors
 \begin{eqnarray*}
 {\mathbf v}^t(p)&=&(\varphi_{1,U'}(p), \varphi_{2,U'}(p), \ldots, \varphi_{t,U'}(p))\\
 {\mathbf v}^{n-t}(p)&=&(\varphi_{t+1,U'}(p), \varphi_{t+2,U'}(p), \ldots, \varphi_{n,U'}(p)).
 \end{eqnarray*}
 Since $p\in W_0$, we see that ${\mathbf v}^t(p)$ is not identically zero and  we necessarily have that
 \begin{equation}
 \label{eq:uniciddlevantamiento}
 \psi_W(p)=([{\mathbf v}^t(p)], {\mathbf v}^t(p), {\mathbf v}^{n-t}(p) ).
 \end{equation}
 This shows the uniqueness of $\psi_W$.

 Now take a point $p\in W$; we denote $\varphi_{i,p}$ the germ of $\varphi_{i,U'}$ at $p$, even in the case when $p\notin \sigma^{-1}(0)$. Consider the set
 $$
 I_p=\{i\in \{1,2,\ldots,t\}; \quad \varphi_{i,p} \mbox{ divides }  \varphi_{j,p},\;  \mbox{ for any }  j\in \{1,2,\ldots,t\} \}.
 $$
 Let us note that $I_p\ne \emptyset$. Indeed, saying that $I_p=\emptyset$ means that $\varphi_{i,p}=0$ for all $i=1,2\ldots,t$, this implies that $\varphi_{i,U'}=0$ and thus $\phi_{i,U}=0$  for any $i=1,2,\ldots,t$; this contradicts the hypothesis that the image of $\phi$ is not contained in $Y$.  Let us choose an index $i\in I_p$. Define the germs
 $$
 \varphi_{ji,p}=\left\{
 \begin{array}{ccc}
  \varphi_{j,p}/\varphi_{i,p}&\mbox{ if }& 1\leq j\leq t,\\
  \varphi_{j,p}&\mbox{ if }& t+1\leq j\leq n.\\
 \end{array}
 \right.
 $$
 Note that $\varphi_{ii,p}=1$. We can define the germ $\psi_p$ of $\psi_W$ by
 $$
 \psi_p= ([\varphi_{1i,p},\varphi_{2i,p},\ldots,\varphi_{ti,p}], \varphi_{1i,p},\varphi_{2i,p},\ldots,\varphi_{ti,p},\varphi_{t+1,i,p},\ldots, \varphi_{ni,p}).
 $$
 The definition does not depend on the index $i\in I_p$ and the uniqueness is guarantied since the restriction to $W_0$ is as indicated in Equation  \ref{eq:uniciddlevantamiento}.

Let us  consider now the case where $\pi:(M',\pi^{-1}(K))\rightarrow (M,K) $ is a single blowing-up with center $(Y, Y\cap K)$ and $\phi: (N,C)\rightarrow (M,K)$ is as in the statement. Once this case is solved, we obtain directly the general case by induction on the number of blowing-ups in $\pi$.

In view of the previous result, for any point $p\in N$ there is an open set $U_p\subset N$ with $p\in U_p$ and a finite sequence of blowing-ups over $p$
$$
\sigma_{U_p}:U'_p\rightarrow U_p
$$
such that for any open subset $W'\subset U'_p$ there is a unique map $\psi_{W'}:W'\rightarrow M'$ such that $\phi\circ (\sigma_{U_p}\vert_{W'})=\pi\circ \psi_{W'}$. Note that
$$
\psi_{W'}=\psi_{U'_p}\vert_{W'},
$$
for any open set $W'\subset U'_p$. By the compactness of $C\subset N$, we can cover $C$ by finitely many open subsets of the type $U_p$, with $p\in C$. That is, there are finitely many points $p_1,p_2,\ldots,p_r$ in $C$ such that
$$
C\subset \cup_{i=1}^r U_i,\quad U_i=U_{p_i},\; i=1,2,\ldots,r.
$$
Without loss of generality, we assume that $p_j\notin U_i$, if $j\ne i$. We can glue the morphisms $\sigma_{U_i}: U'_i\rightarrow U_i$ into a morphism
$$
\sigma_U: U'\rightarrow U=\cup_{i=1}^rU_i,
$$
is such a way that we identify $U'_i=\sigma_U^{-1}(U_i)$. Of course, the morphism $\sigma_U$ is the composition of a sequence of blowing-ups points, over the points $p_1,p_2,\ldots,p_r$.
Note that $\sigma_U$ induces a morphism of germs
$$
\sigma: (N',\sigma^{-1}(C))\rightarrow (N,C),
$$
 where $(N',\sigma^{-1}(C))$ is represented by $(U',\sigma_U^{-1}(C)$ and $(N,C)$  by $(U,C)$. On the other hand, by the uniqueness property, we have that
$$
\psi_{U'_i}\vert_{U'_i\cap U'_j}=\psi_{U'_j}\vert_{U'_i\cap U'_j}.
$$
Then, we can also glue the morphisms $\psi_{U'_i}$ to a morphism $\psi_{U'}:U'\rightarrow M'$ such that
$$
\phi\circ \sigma_{U}=\pi\circ \psi_{U'}.
$$
We have an induced morphism of germs $\psi:(N',\sigma^{-1}(C))\rightarrow (M',\pi^{-1}(K))$ with the property that $\pi\circ\psi=\phi\circ\sigma$. This ends the proof.
 \end{proof}

\section{Appendix II}
Here we provide a proof of Proposition \ref{prop:simplepointsandlogorder}.
  Recall that we have  a foliated space $((M,K), E, {\mathcal F})$  of generalized hypersurface type and a point $p\in K$. We have to show that the following statements are equivalent
\begin{enumerate}
\item The point $p$ is a simple point for $((M,K), E, {\mathcal F})$.
\item $\operatorname{LogOrd}_p({\mathcal F},E)=0$.
\end{enumerate}
We have that (1) implies (2) as a direct consequence of the definition of simple point in Subsection  \ref{definicionsimplepoint}.

Let us suppose that $\operatorname{LogOrd}_p({\mathcal F},E)=0$. We assume that the dimensional type $\tau$ is equal to $n$. The case when $\tau<n$ may be done in the same way. Moreover, since we are working locally at $p$, we identify $(M,p)=({\mathbb C}^n,0)$ and we work at the origin of ${\mathbb C}^n$. Choose local coordinates $x_1,x_2,\ldots,x_n$ such that $E=\left(\prod_{i=1}^ex_i=0\right)$.

Let us see first that $n-1\leq e\leq n$. If this is not the case, we have $e\leq n-2$ and one of the following expressions holds for a local generator $\eta$ of $\mathcal F$ (up to a reordering and to multiply by a unit)
\begin{enumerate}
\item[a)]$\eta=dx_1/x_1+\sum_{i=2}^e a_idx_i/x_i+a_{e+1}dx_{e+1}+
a_{e+2}dx_{e+2}+\sum_{i=e+3}^n a_idx_i$.
\item[b)] $\eta=\sum_{i=1}^e a_idx_i/x_i+dx_{e+1}+
a_{e+2}dx_{e+2}+\sum_{i=e+3}^n a_idx_i$.
\end{enumerate}
This situation does not hold, since we can find a non-singular vector field $\xi$ such that $\eta(\xi)=0$, hence $\xi$ trivializes the foliation and $\tau<n$. The vector field $\xi$ can be taken as follows
$$
\xi=\left\{
\begin{array}{cc}
a_{e+1}x_1\partial/\partial x_1-\partial /\partial x_{e+1},&\mbox{ in case a)}\\
a_{e+2}\partial/\partial x_{e+1}-\partial /\partial x_{e+2} ,&\mbox{ in case b)}
\end{array}
\right.
$$
Thus, we conclude that $n-1\leq e\leq n$. Note that, even when $e=n-1$, the case a) above does not hold. Then, we have that $e=n$ or $e=n-1$ and one of the following situations holds:
\begin{enumerate}
\item[i)]$\eta=dx_1/x_1+\sum_{i=2}^n a_idx_i/x_i$.
\item[ii)] $\eta=\sum_{i=1}^{n-1} a_idx_i/x_i+dx_{n}$, with $a_i(0)=0$, for $i=1,2,\ldots,n-1$.
\end{enumerate}
Assume first we are in the situation i) and put $\lambda_1=1$, $\lambda_i=a_i(0)$, for $i=2,3,\ldots,n$. We have to show that there is no resonance $\sum_{i=1}^{n}m_i\lambda_i=0$ with $\mathbf{m}\ne 0$, for $\mathbf{m}=(m_1,m_2,\ldots,m_n)$.
Let us reason by contradiction, assuming that there is such a resonance. Note that there is at least one $m_i\ne 0$ with $2\leq i\leq n$.   Up to a reordering, we assume that $m_2m_3\cdots m_\ell\ne0$ and $m_i=0$ for $\ell <i\leq n$. Consider the map $\phi:({\mathbb C}^2,0)\rightarrow ({\mathbb C}^n,0)$ given by
$$
x_1=uv^{m_1}, x_2=v^{m_2},\ldots,x_\ell=v^{m_\ell}; \quad x_i=0,\mbox{ if }\ell< i\leq n.
$$
Then we have
$$
\phi^*\eta=\frac{du}{u}+b(u,v)\frac{dv}{v},
$$
where $b(u,v)=m_1\lambda_1+\sum_{i=2}^\ell m_i a_i(uv^{m_1},v^{m_2},\ldots,v^{m_\ell},0,\ldots,0)$. Note that $b(0,0)=0$, since $\sum_{i=1}^nm_i\lambda_i=0$. We have two possible situations:
\begin{enumerate}
\item The function $b(u,v)$ is not divisible by $v$. In this case, we have that $\phi^*\eta$ defines a saddle-node. This is not possible, since $\mathcal F$ is complex-hyperbolic.
\item We have that $b(u,v)=vb'(u,v)$. Then $\phi^*{\mathcal F}$ is defined by the non-singular $1$-form  $du+ub'(u,v)dv$. We know that there is a unit $U(u,v)$ such that $du+ub'(u,v)dv=d(uU(u,v))$ and we take new local coordinates $u^*=uU(u,v)$ and $v$. We have that $\phi^*{\mathcal F}=(du^*=0)$ and $\phi(v=0)$ is the origin. This implies that $\mathcal F$ is dicritical, but this is not possible since it is a generalized hypersurface.
\end{enumerate}
Assume now that we are in the situation ii). We are going to show the existence of a non-singular hypersurface $H$ having normal crossings with $E$ such that we get a simple corner for $\mathcal F$ with respect to $E\cup H$. Note that $H$ should have an equation $x_n=f(x_1,x_2,\ldots,x_{n-1})$. If we do this, we are done, since we are in situation i) with respect to $E\cup H$. In particular, we are done when $x_n$ divides $a_i$ for $i=1,2,\dots,n-1$.
Let us rename the variables as $$
\mathbf{y}=(y_1,y_2,\ldots,y_{n-1})=(x_1,x_2,\ldots,x_{n-1}),\quad z=x_n.
$$
We end the proof by providing a coordinate change $z^*=z-f(\mathbf{y})$ with the property that $z^*=0$ is an invariant hypersurface of $\mathcal F$. The existence of such a coordinate change depends upon certain non-resonances in $\eta$, that will nor occur thanks to the hypothesis that $\mathcal F$ is a generalized hypersurface. Let us precise this. We write
\begin{equation}
\label{eq:omega}
\eta=z(\omega_0+\tilde\omega)+\omega'+dz,
\end{equation}
where
$\omega_0=\sum_{i=1}^{n-1}\mu_i{dy_i}/{y_i}$,
$\tilde\omega=\sum_{i=1}^{n-1}\tilde a_i(\mathbf{y},z){dy_i}/{y_i}$ and
$\omega'=\sum_{i=1}^{n-1} a'_i(\mathbf{y}){dy_i}/{y_i}$,
with $\mu_i\in {\mathbb C}$,  $\tilde a_i(\mathbf{0}, 0)=0$ and $a'_i(\mathbf{0})=0$, for $i=1,2,\ldots,n-1$.
\begin{lemma}
 \label{lema:noresonanciatotal}
 In the above situation, for any $ \mathbf{m}=(m_1,m_2,\ldots,m_{n-1})\in {\mathbb Z}^{n-1}_{\geq 0}$, with $\mathbf{m}\not=\mathbf{0}$, we have that
$\omega_0+\sum_{i=1}^{n-1}m_i{dy_i}/{y_i}\ne 0$.
\end{lemma}
\begin{proof} Let us reason by contradiction, assuming that there is  $\mathbf{m}\in {\mathbb Z}^{n-1}_{\geq 0}$, with $\sum_{i=1}^{n-1}m_i=m>0$, such that
$\omega_0=-\sum_{i=1}^{n-1}m_i{dy_i}/{y_i}$. Consider the map
$$
\phi:({\mathbb C}^2,0)\rightarrow ({\mathbb C}^n,0)
$$
 given by $z=u$ and $y_i=v$, for $i=1,2,\ldots,n-1$.
Then we have
$$
\phi^*\eta=
(-mu+ h(v)+ug(u,v))\frac{dv}{v}+du
,
$$
where $g(0,0)=0$ and $h(0)=0$. In particular, this singularity is a pre-simple singularity in dimension two (non-nilpotent linear part) that is not simple, since we have the resonance $1\cdot(-m)+m\cdot 1=0$. These singularities are either dicritical or they have a hidden saddle-node, \cite{Can-C-D}. This is the desired contradiction.
\end{proof}
Let us perform the coordinate change $z\mapsto z^*$ as a Krull limit, in the following way. Assume that the order $\nu_0(\omega')$ of the coefficients of $\omega'$ is $\nu_0(\omega')=m>0$ and let us write
$$
\omega'= \bar\omega+\omega'',
$$
where $\nu_0(\omega'')>m$
and the coefficients of $\bar\omega$ are homogeneous of degree $m$. We are going to show that there is a homogeneous polynomial $p_m(\mathbf{y})$ of degree $m$ such that if $z^{(m)}=z-p_m(\mathbf{y})$, then
$$
\eta=z^{(m)}\left\{\omega_0+\tilde\omega^{(m)}\right\}+{\omega'}^{(m)}+dz^{(m)},
$$
with the same structure as in Equation \ref{eq:omega} but with $\nu_0({\omega'}^{(m)})>m$. Now, we are done by taking the Krull limit of the $z^{(m)}$. Of course, we obtain a formal invariant hypersurface $z^*=0$, but we know that all the formal invariant hypersurfaces of generalized hypersurfaces are in fact convergent ones and thus we are done. Now, looking at the $\mathbf{y}$-homogeneous part of degree $m$ in the Frobenius integrability condition $\eta\wedge d\eta=0$, we have that
$
d\bar\omega=\bar\omega\wedge\omega_0
$.
Write $\eta$ in the coordinates $\mathbf{y},z^{(m)}$:
$$
\eta=z^{(m)}(\omega_0+\tilde\omega)+(p_m\tilde\omega+\omega'')+dz^{(m)}
+ (p_m\omega_0+\bar\omega + dp_m).
$$
If there is $p_m$ such that $p_m\omega_0+\bar\omega + dp_m=0$, then we are done. Let us write
$$
\bar\omega=\sum_{\vert\mathbf{m}\vert=m}\mathbf{y}^\mathbf{m} \bar\omega_\mathbf{m},
\quad
d\bar\omega=\sum_{\vert\mathbf{m}\vert=m}\mathbf{y}^\mathbf{m}\frac{d \mathbf{y}^\mathbf{m}}{\mathbf{y}^\mathbf{m}}\wedge\bar\omega_\mathbf{m},
$$
where the $1$-forms $\bar\omega_\mathbf{m}$ have constant coefficients. Moreover
$$
\bar\omega\wedge\omega_0= \sum_{\vert\mathbf{m}\vert=m}\mathbf{y}^\mathbf{m}
\bar\omega_\mathbf{m}\wedge\omega_0.
$$
Since $
d\bar\omega=\bar\omega\wedge\omega_0
$, we conclude that
$$
 \left\{ ({d \mathbf{y}^\mathbf{m}}/{\mathbf{y}^\mathbf{m}})+\omega_0\right\}\wedge
 \bar\omega_\mathbf{m}=0,
 \quad
 \mbox{ for all } \mathbf{m}\in {\mathbb Z}^{n-1}_{\geq 0} , \mbox{ with } \vert{\mathbf{m}}\vert=m.
$$
By Lemma \ref{lema:noresonanciatotal}, we know that
$
0\ne d \mathbf{y}^\mathbf{m}/\mathbf{y}^\mathbf{m}+\omega_0
$.
Hence, there are constants $c_\mathbf{m}\in {\mathbb C}$ such that
$$
 \bar\omega_\mathbf{m}+c_\mathbf{m} \left\{ ({d \mathbf{y}^\mathbf{m}}/{\mathbf{y}^\mathbf{m}})+\omega_0\right\}=0,
 \quad
 \mbox{ for all } \mathbf{m}\in {\mathbb Z}^{n-1}_{\geq 0} , \mbox{ with } \vert{\mathbf{m}}\vert=m.
$$
Taking $p_m=\sum_{\vert\mathbf{m}\vert=m}c_{\mathbf{m}}\mathbf{y}^{\mathbf{m}}$, we obtain that $p_m\omega_0+\bar\omega + dp_m=0$ and we are done.

\bigskip
{\small 
\noindent{\sc Felipe Cano.} 
{Departamento de
\'Algebra, An\'alisis Matem\'atico,
Geometr\'\i a y Topolog\'\i a. Universidad de Valladolid.
Paseo de Bel\'en 7,
47011 -- Valladolid, SPAIN}  \\
fcano@agt.uva.es \\

\noindent {\sc Nuria Corral.} 
{Departamento de Matem\'aticas, Estad\'\i stica y Computaci\'on.
Universidad de Cantabria.
Avda. de los Castros s/n, 39005 -- Santander, SPAIN} \\
nuria.corral@unican.es

}
\end{document}